\documentclass[10pt]{amsart}
\usepackage{calrsfs}

\usepackage{amsmath,tikz}
\usetikzlibrary{tikzmark,fit}

\usepackage{hyperref}
\usepackage{amsthm}
\usepackage{amssymb}
\usepackage{multirow}
\usepackage{tikz-cd}
\usepackage{amsmath}
\usepackage{todonotes}
\usepackage{amsbsy}
\usepackage[all]{xy}
\usepackage{qtree}
\usepackage{mathtools}
\usepackage{hyperref}
\usepackage[table]{xcolor}
\usepackage{longtable}
\usepackage{array}
\usepackage{booktabs}
\usepackage[margin=2.25cm]{geometry}

\usepackage{tikz}
\usetikzlibrary{shapes.geometric, arrows.meta, positioning}

\tikzstyle{chapter} = [rectangle, rounded corners, minimum width=3cm, minimum height=1cm, text centered, draw=black, fill=headerblue]
\tikzstyle{arrow} = [thick, ->, >=Stealth]

\usetikzlibrary{calc,trees,positioning,arrows,chains,shapes.geometric,  decorations.pathreplacing,decorations.pathmorphing,shapes,  matrix,shapes.symbols}

\definecolor{rowgray}{gray}{0.95}
\definecolor{headerblue}{rgb}{0.8,0.9,1}

\usepackage{enumitem}

\usepackage{xfrac}    

\usepackage{color, colortbl}
\definecolor{LightCyan}{rgb}{0.88,1,1}
\definecolor{Gray}{gray}{0.9}

\usepackage{faktor}

\usepackage{longtable,booktabs}
\definecolor{LightCyan}{rgb}{0.88,1,1}
\definecolor{LightCyan1}{rgb}{0.80,1,1}
\definecolor{Gray}{gray}{0.9}
\definecolor{Gray1}{gray}{0.97}

\usepackage{changes}
\definechangesauthor[color=orange,name={Aristidesd Kontogeorgis}]{AK}
\definechangesauthor[color=blue,name={Alex Terezakis}]{AT}

\newtheorem{theorem}{Theorem}

\newtheorem{lemma}[theorem]{Lemma}

\newtheorem{proposition}[theorem]{Proposition}
\theoremstyle{definition}
\newtheorem{example}[theorem]{Example}
\newtheorem{remark}[theorem]{Remark}
\newtheorem{definition}[theorem]{Definition}

\newtheorem*{thmA}{Theorem~A}
\newtheorem*{thmB}{Theorem~B}
\newtheorem*{thmC}{Theorem~C}
\newtheorem*{thmD}{Theorem~D}
\newtheorem*{thmE}{Theorem~E}
\newtheorem*{thmF}{Theorem~F}
\newtheorem*{propA}{Proposition~A}

\newcommand{\Spe}{ {\mathrm{Spec} } }

\newcommand{\Z}{\mathbb{Z}}

\newcommand{\Q}{\mathbb{Q}}

\newcommand{\A}{{ \mathrm{Aut} }}

\newcommand{\im}{{ \mathrm{Im} }}
\renewcommand{\O}{{\mathcal{O}}}

\newcommand{\Gal}{{ \mathrm{Gal}}}
\newcommand{\Out}{{ \mathrm{Out}}}
\renewcommand{\mod}{{\;\mathrm{mod}}}

\date{\today}

\title[Descent theory and Equivariant Categories]{An Arithmetic Topology viewpoint on Descent theory and Equivariant Categories}

\author[M. Karakikes]{Miltiadis Karakikes}
\address{Department of Mathematics, National and Kapodistrian  University of Athens
Pane\-pist\-imioupolis, 15784 Athens, Greece}
\email{miltoskar@math.uoa.gr}

\author[S. Karanikolopoulos]{Sotiris Karanikolopoulos}
\address{Department of Mathematics, National and Kapodistrian  University of Athens
Pane\-pist\-imioupolis, 15784 Athens, Greece}
\email{sotiriskaran@gmail.com}

\author[A. Kontogeorgis]{Aristides Kontogeorgis}
\address{Department of Mathematics, National and Kapodistrian  University of Athens
Pane\-pist\-imioupolis, 15784 Athens, Greece}
\email{kontogar@math.uoa.gr}

\author[D. Noulas]{Dimitrios Noulas}
\address{Department of Mathematics, National and Kapodistrian  University of Athens
Pane\-pist\-imioupolis, 15784 Athens, Greece}
\email{dnoulas@math.uoa.gr}

\date \today



\begin{document}



\begin{abstract}
We establish a unified group-theoretic framework bridging the arithmetic homotopy exact sequence of a variety and the Birman exact sequence of a surface. Within this framework, we reinterpret classical arithmetic notions - such as the descent of varieties and of covers - and construct their topological analogues. We formalize the parallel setting between closed subgroups of the absolute Galois group and subgroups of the Mapping Class Group of a base space and their actions on fundamental groups. This provides an analogy between arithmetic and topological invariants, allowing us to define the groups of moduli, definition, and invariance in both settings. 

Using this unified perspective, some purely group-theoretic proofs provide results in both settings simultaneously. Applications include a topological analogue of Weil’s Descent Theorem for mapping class groups
and an adaptation of Dèbes and Douai's cohomological obstructions regarding descent of algebraic covers to the topological setting.

Finally, we elevate these results to the categorical level. We demonstrate that the classical Weil cocycle condition is equivalent to the existence of a linearization in the language of equivariant categories. Applying this perspective to the bounded derived category of coherent sheaves $\mathsf{D^b}(X)$, we show that the equivariant derived category $\mathsf{D^b}(X)^G$, under the action induced by a Weil descent datum, recovers the derived category of the descended variety.
\end{abstract}

\maketitle

\tableofcontents

\section{Introduction}

The study of the arithmetic of algebraic curves is intrinsically linked to the symmetries of their fundamental groups. This relationship is most famously encapsulated in Grothendieck's \textit{anabelian geometry} \cite{MR1483107}, where the absolute Galois group $\Gal(\overline{\mathbb{Q}}/\mathbb{Q})$ is studied through its outer action on the profinite completion of the fundamental group of a punctured  variety. Parallel to this, in the realm of low-dimensional topology, the Mapping Class Group $\mathrm{Mod}(S)$ acts similarly by outer automorphisms on the topological fundamental group of the surface $S$, providing a discrete analogue to the Galois actions found in number theory. 
This analogy complements the Mazur-Morishita-Kapranov-Reznikov dictionary \cite{Morishita2011-yw, MR1925911, MazurAlex} of \textit{arithmetic topology}, which is the study between similarities of knots in $3$-dimensional manifolds and primes in rings of integers.
Braids and Galois elements can be considered to be a part of this dictionary: a prime corresponds to a conjugacy class of a Frobenius element in $\Gal(\overline{\mathbb{Q}}/\mathbb{Q})$, while a knot corresponds to a Markov equivalence class of braids in $B_n$ \cite{MR4117575}. This parallel was significantly deepened by Y. Ihara \cite{Ihara1985-it, IharaCruz, MR1159208, MorishitaATIT}, who reinterpreted the absolute Galois group as a \textit{profinite braid} group by exploiting the geometry of $\mathbb{P}^1-\{0,1,\infty\}$. 
Namely, Ihara demonstrated that the elements of the absolute Galois group act as automorphisms of a profinite free group, forming a discrete analogue of Artin's representation of the braid group.

In this article, we pursue this analogy further by transferring classical arithmetic notions to the topological setting. A central object of study in the arithmetic of curves is the ``field of moduli/definition problem''. Our primary objective is to bridge the two perspectives by formulating the arithmetic setup of field descent in a purely group-theoretic language, effectively defining the analogous theory for surfaces. Our approach is guided by the principle that \textbf{fields descend as groups ascend}.

By interpreting the choice of a field extension as the choice of a subgroup of the absolute Galois group—and, by analogy, a subgroup of the Mapping Class Group—we show that many deep results in algebraic geometry regarding group extensions and equivariant structures have natural topological incarnations. Specifically, we propose that the arithmetic homotopy exact sequence \cite{SGA1}:
\begin{equation}\label{intro:1}
1 \to \Pi_{K_s}(B^*) \to \Pi_K(B^*) \to \Gal(K_s/K) \to 1,
\end{equation}
where $B^*$ represents the punctured base variety,
is analogous to the Birman exact sequence \cite{FarbMagalit} in topology:
\begin{equation}\label{intro:2}
    1 \to \pi_1(B^*,b_0) \to \mathrm{Mod}(B^*,b_0) \to \mathrm{Mod}(B^*) \to 1.
\end{equation}

For instance, it is well known that the existence of a $K$-rational point yields the splitting of (\ref{intro:1}). We demonstrate in Proposition \ref{prop:ses split} that this splitting can also be established via a group-theoretic framework that applies with equal rigor to (\ref{intro:2}). Furthermore, we note that the splitting of these sequences is generally obstructed in both settings. On the one hand, the condition that rational points exist is often restrictive; on the other hand, the Birman exact sequence is known not to split for closed surfaces.

A central theme of this work is the reinterpretation of \textit{descent theory}. Historically, the problem of determining the field of definition of an algebraic variety was solved by A. Weil in his seminal 1956 work \cite{Weil56}. Weil introduced the concept of \textit{descent data}—a family of isomorphisms satisfying a specific cocycle condition—which allows one to ``descend'' a variety defined over a field $L$ to a smaller subfield $L_0$. While Weil's theorem is a cornerstone of arithmetic geometry, its topological counterpart—descending a cover to a larger group—has remained less formalized in the language of descent data. We aim to show that the Weil cocycle condition, the standard tool for the descent of varieties, can be more naturally expressed as a linearization in the language of \textit{equivariant categories} (Proposition \ref{prop:equivariantweildescent}).

Equivariant categories arise naturally in the study of sheaves on orbit spaces. This concept traces back to Grothendieck's seminal T\^ohoku paper \cite{GroTo}, where he investigated the relationship between $G$-equivariant sheaves on a topological space $X$ and sheaves on the orbit space $X/G$. This framework was formalized by Mumford and Fogarty in their Geometric Invariant Theory \cite{MR1304906}, establishing the definitive notion of a $G$-equivariant sheaf. The theory has since flourished; notably, Bernstein and Lunts \cite{BerLuntsEquiv} pioneered the definition of equivariant bounded derived categories for $G$-spaces, a framework later extended to the scheme-theoretic setting by Achar \cite{Achar}. Consequently, the study of group actions has been generalized to arbitrary categories (Definition \ref{def:catactions}), giving rise to abstract equivariant categories (Definition \ref{def:equivcats}). These categories are now viewed as non-commutative categorical quotients, forming a distinct and active area of research. For instance, Elagin \cite{elagin2015equivarianttriangulatedcategories} established criteria for the equivariant category of a triangulated category to inherit a canonical triangulated structure, while Beckmann and Oberdieck \cite{MR4589277} determined the equivariant categories of elliptic curves with respect to so-called Calabi–Yau group actions.

A classical algebraic prototype is provided by a finite group acting on a ring $R$: the equivariant category $(\mathrm{Mod-}R)^G$ is equivalent to the category of modules over the skew group algebra $R*G$, a correspondence extensively studied by Reiten and Riedtmann \cite{ReitenRiedtmann} (see Example \ref{exm:quotaffine}). In the geometric setting, it is well known that a $G$-action on a scheme $X$ induces an action on $\mathsf{Coh}(X)$ such that the equivariant category $\mathsf{Coh}^G(X)$ is equivalent to the category of coherent sheaves on the quotient stack $[X/G]$ (see Examples \ref{actiononcoh1}, \ref{actiononcoh2}, \ref{actiononcoh3}).

We push this philosophy further by extending the descent formalism to the \textit{bounded derived category} $\mathsf{D^b}(X)$. Rather than viewing the derived category merely as a homological invariant, we treat it as a geometric object in its own right, subject to Galois actions and descent. We construct a canonical group action on $\mathsf{D^b}(X)$ arising solely from a Weil descent datum on $X$, thereby providing a categorical proof of Weil's theorem (Proposition \ref{prop:equivariant descent theorem}) that naturally encompasses derived invariants.

\subsubsection*{Main Results \& Ideas Employed}
We now outline the primary contributions of this article, organized by the central themes that bridge the arithmetic and topological perspectives.

\subsection*{The Arithmetic-Topology Viewpoint}

The first theme of this work concerns the structural identity between the arithmetic, geometric and topological fundamental groups. Grothendieck's theory of the fundamental group provides the arithmetic short exact sequence:
\begin{equation*}
1 \to \Pi_{K_s}(B^*) \to \Pi_K(B^*) \to \Gal(K_s/K) \to 1.
\end{equation*}
which yields the continuous homomorphism of profinite groups
\[\Gal(K_s/K) \rightarrow \mathrm{Out}(\Pi_{K_s}(B^*)).\]

Similarly, in the topological setting, the Dehn-Nielsen-Baer theorem (Theorem \ref{thm:Dehn-Nielsen-Baer}) in conjunction with the Birman exact sequence yield a homomorphism
\[\mathrm{Mod}(B^*) \rightarrow \mathrm{Out}(\pi_1(B^*)).\]

As mentioned previously, the existence of a $\Gal(K_s/K)$-fixed point leads to the splitting of the arithmetic sequence (\ref{intro:1}). We can apply this principle to the topological setting using purely group-theoretic language:

\begin{propA}(Proposition \ref{prop:ses split}).
\textit{The short exact sequence
\[
1\rightarrow \pi_1(B^*,b_0) \rightarrow \Pi_{\mathrm{Mod}(B^*)} \rightarrow \mathrm{Mod}(B^*)\rightarrow 1.
\]
splits if there is a point $P$ such that all $\sigma \in \mathrm{Mod}(B^*)$ have a homeomorphism representative $\hat{\sigma}$ that fixes $P$.
}
\end{propA}

This is a special case of Proposition \ref{prop:ses split} whose general statement considers any subgroup of $\mathrm{Mod}(B^*)$.
We observe that a significant number of proofs regarding sequences (\ref{intro:1}), (\ref{intro:2}) and their associated outer automorphisms can be formulated in a unified way that applies to both settings simultaneously.
The power of this unified language extends far beyond the splitting of exact sequences. By abstracting the notion of ``field'' into ``group action'', we recover several deep geometric properties purely through group theory:

\begin{enumerate}
    \item \textbf{Regularity:} The arithmetic condition that a field extension $K(X)/K(B)$ contains no new constants is classically verified by intersecting with the separable closure $K_s$ of $K$. In Proposition \ref{prop:equalityofindices}, we show that this is equivalent to the purely group-theoretic condition of index preservation: $[\Pi_A : H] = [\Pi_1 : R]$. This allows us to define and study the notion of ``regularity'' of covers in topology, where no base field exists.
    
    \item \textbf{Monodromy Actions:} Through the theory of Dèbes et al. \cite{DebesDouai97,DebesEmsalem, MR1671938}, the field extensions $K(X)/K(B)$ correspond to monodromy representations of the arithmetic fundamental group $\Pi_K(B^*)$. We demonstrate that the same correspondence applies to the Birman exact sequence and covers $X\rightarrow B$ ``defined'' over $A\subseteq \mathrm{Mod}(B^*)$ (Definition \ref{def:topological_defined}) correspond to monodromy representations of $\Pi_A$ (Proposition \ref{prop:mereCoverstransTop}). In both settings, the groups $\Gal(K_s/K)$ and $\mathrm{Mod}(B^*)$ act on the respective representations. This allows us to treat arithmetic and topological symmetries on equal footing.
    
    \item \textbf{Cohomological Obstructions:} Finally, the obstruction of descending an arithmetic cover to its field of moduli-also related to the splitting of (\ref{intro:1})-is identified in Section \ref{ch:descend_on_monodromy} as a specific Galois cohomology class, originally by Dèbes and Douai \cite{DebesDouai97}. This group cohomological description (Theorems \ref{thm:DebesDouaiThB}, \ref{thm:DebesDouaiTHC}) applies verbatim to both the arithmetic descent of varieties and the topological group ascent of covers via the idea described next.
\end{enumerate}

\subsection*{Fields Descend as Groups Ascend}
This simple idea follows from the following observation. Given a variety $X$ defined over a field $K$ and a field extension $L/K$, then one can extend it to a variety $X_L \coloneqq X \times_K L $. One could also extend it similarly to a variety $X_{K_s}$, where $K_s$ denotes a separable closure of $K$ that contains $L$. Then we know that $X_{K_s}$ is invariant under the action of the absolute Galois group $\Gal(K_s/K)$. However, if we view $X_{K_s}$ as arising from $X_L$, it is effectively invariant under the subgroup $\Gal(K_s/L)$.
Thus, the smaller the field of definition (i.e., descending from $L$ to $K$), the larger the group of invariance (i.e., ascending from $\Gal(K_s/L)$ to $\Gal(K_s/K)$).

This inverse relationship allows us to translate the arithmetic problem of finding a minimal field of definition into the group-theoretic problem of finding a maximal subgroup of invariance (or definition) within a larger ambient group. We apply this philosophy to the topological setting by replacing the Galois tower with subgroups of the Mapping Class Group $\mathrm{Mod}(B^*)$.

\subsection*{Fields \& Groups of Moduli, Definition \& Invariance}
In Section \ref{sec:field of moduli, definition, invariance}, we refine the notion of symmetry by distinguishing between three levels of group actions, applicable to both arithmetic and topology. For an object $X$ and a subgroup $A$ of either $\mathrm{Mod}(B^*)$ or $\Gal(K_s / K)$:
\begin{enumerate}
    \item The \textit{Group of Moduli} $A_X^{\mathrm{mod}}$ consists of elements $\sigma \in A$ such that ${^\sigma \! X}$ is isomorphic to $X$. In the arithmetic setting, its fixed field is the \textit{field of moduli}.
    \item The \textit{Group of Definition} $A_X^{\mathrm{def}}$ is the subgroup where these isomorphisms satisfy the Weil cocycle condition (or equivalently, admit the structure of a linearization of $X$). This corresponds to the \textit{field of definition}.
    \item The \textit{Group of Invariance} $A_X^{\mathrm{inv}}$ consists of the elements preserving strictly $X$. This corresponds to the {\em field of invariance}.
\end{enumerate}
We clarify the subtle distinction between the field of moduli and the field of definition. While they might coincide in most cases, for example for elliptic curves, we discuss a specific example (Subsection \ref{exm:field of moduli vs field of definition}) where the field of moduli is strictly smaller than the field of definition, a phenomenon detected by the obstruction of the Weil cocycle in the group of definition. 

\subsection*{Weil's Topological Ascent and Definability}
In the topological realm, the absence of an underlying field structure necessitates a purely group-theoretic reconstruction of the notion of ``field of definition.'' We address this by replacing field extensions with subgroups of the Mapping Class Group $\mathrm{Mod}(B^*)$.

We introduce the notation $X \to (B,A)$ to denote a cover defined over a subgroup $A \leq \mathrm{Mod}(B^*)$ 
(Definition \ref{def:topological_defined}). This definition captures the topological essence of a curve being defined over a subfield: the geometric object remains invariant under the symmetries corresponding to that subfield.

A central result of this paper is the formulation of a topological analogue to Weil's descent theorem. In the arithmetic setting, descending a variety from a field $L$ to a subfield $L_0$ requires a \textbf{Weil descent datum}: a family of isomorphisms $\{f_\sigma\colon X \xrightarrow{\simeq} {^\sigma \! X}\}_{\sigma \in \Gal(L/L_0)}$ satisfying the cocycle condition
\[
    f_{\sigma \tau} = {^\sigma \! f_\tau} \circ f_\sigma.
\]
In our topological framework, this transforms into an \textit{ascent} problem: extending the definition of a cover from a group $A$ to a larger group $A'$. We define a \textbf{Mapping Class Group ascent datum} as a family of pairs of homeomorphisms $\{(f_\sigma, \hat{\sigma})\}_{\sigma \in A'/A}$, where $f_\sigma\colon X \to {^\sigma \! X}$ maps the cover to its twist, satisfying the analogous cocycle condition:
\[
    (f_{\sigma \tau}, \widehat{\sigma\tau}) = ({^\sigma \! f_\tau}, {^\sigma \! \hat{\tau}}) \circ (f_\sigma, \hat{\sigma}).
\]
The comparison reveals a striking duality: the arithmetic descent datum allows a geometric object to descend to a smaller field, while the topological ascent datum allows the object to ascend to a larger group of symmetries.

\begin{thmA} (Theorem \ref{thm:Weil topological}).
\textit{Let $A \lhd_{\mathrm{f}} A' \leq \mathrm{Mod}(B^*)$ and let $X \to (B,A)$ be a cover defined over $A$. The cover is definable over the larger group $A'$ if and only if it admits a Mapping Class Group ascent datum with respect to the quotient $A'/A$.
}
\end{thmA}

This result, together with the topological formulation of Dèbes' theorems (Theorems \ref{thm:Mere1}, \ref{thm:DebesDouaiThB}, \ref{thm:DebesDouaiTHC}), completes our arithmetic-topology viewpoint, establishing that the cohomological obstructions to enlarging the group of definition in topology are formally identical to the obstructions of descending the field of definition in arithmetic.

\subsection*{Descent on Monodromy}
In the arithmetic setting, the descent theory of covers is governed by the results of Dèbes and Douai \cite{DebesDouai97}. Their \textit{Main Theorem} establishes that for a cover $X \to B$ defined over $K_s$ with field of moduli $K$, the obstruction to $K$ being a field of definition lies in a specific cohomology class in $H^2(K, Z(G))$, where $G$ is the automorphism group of the Galois closure. We adapt this result to the topological setting utilizing purely group-theoretic techniques. A crucial step is the group-theoretic characterization of the ``extension of constants in the Galois closure'' which plays a central role in the obstruction. By replacing the Galois action of $\Gal(K_s/K)$ with the action of a subgroup $A \leq \mathrm{Mod}(B^*)$, we apply the same techniques to the monodromy representations $\Pi_A \to G$ to derive the topological analogues. 

Let $f\colon X\rightarrow B$ be cover defined over $A\leq \mathrm{Mod}(B^*)$ with group of moduli $A^\prime$ such that $A^\prime/A$ is a finite group, and let $G$ be the associated automorphism group of the Galois closure. Let $N,C$ denote the normalizer and centralizer of $G$ in $S_d$, respectively, where $d$ is the degree of the cover.
The first obstruction, which is non-trivial only for non-Galois covers, relates the extension of constants in the Galois closure to a homomorphism $\Lambda \colon A^\prime/A\rightarrow N/G$.

\begin{thmB}[Theorem \ref{thm:Mere1}] Mapping class group formulation of \cite[\textit{Main Theorem} (I)]{DebesDouai97}

Let $f \colon X \to (B,A)$ be a ($G$)-cover with group of moduli $A^\prime$ and $\overline{\phi}\colon \Pi_{A^\prime} \to N/C $ associated to the monodromy representation $\Pi_A\rightarrow N$ of $f$. Then there exists a unique homomorphism $\lambda \colon A^\prime/A \to N/CG$ compatible with $\overline{\phi}$. Each $A^\prime$-model of $f$ yields a lift of $\lambda$ to $\Lambda\colon A^\prime/A\rightarrow N/G$. Thus, the existence of a homomorphism $\Lambda$ extending $\lambda$ is necessary for the cover to be defined over its group of moduli.
\end{thmB}

Intuitively, this condition requires that the extension to the Galois closure is compatible with the $G$-action. The subsequent theorem, which identifies the \textit{Main Obstruction}, determines whether such a lift $\Lambda\colon A^\prime/A\rightarrow N/G$  arises from an $A^\prime$-model of the cover.

\begin{thmC}[Theorem \ref{thm:DebesDouaiThB}] Mapping class group formulation of \cite[\textit{Main Theorem} (II)]{DebesDouai97}

Let $A^\prime$ be the group of moduli of the ($G$)-cover of $f \colon X \to (B,A)$ relative to $A^\prime/A$.  Assume that the necessary condition of Theorem B is satisfied and fix a lift $\Lambda \colon A^\prime/A \to N/G$ of $\lambda$. The group of moduli $A^\prime$ is a group of definition for $f$ if and only if a specific $2$-cocycle $\Omega_\Lambda$ has its inverse in the image of the coboundary operator $\delta\colon H^1(A^\prime/A, CG/G)\to H^2(A^\prime/A,Z(G))$. That is, there is a $1$-cocycle $\theta$ such that $\delta(\theta)\Omega_\Lambda$ is trivial in $H^2(A^\prime/A,Z(G))$.
\end{thmC}

The next result provides a simplified criterion for the case where the exact sequence $1\to \Pi_A\to\Pi_{A^\prime}\to A^\prime/A\to 1$ splits. In particular, the $A$-model $f\colon X\rightarrow B$ with monodromy representation $\Psi \colon \Pi_A\rightarrow N$ can ascend to the group of moduli $A^\prime$ if and only if there  exists an action of $A^\prime/A$ on $N$ compatible with $\Psi$, meaning it respects the semi-direct product structure $\Pi_{A^\prime} \cong \Pi_A \rtimes A^\prime/A$.

\begin{thmD}[Theorem \ref{thm:DebesDouaiTHC}] Mapping class group formulation of \cite[\textit{Main Theorem} (III)]{DebesDouai97}

Retain the assumptions of Theorem B and assume further that $s\colon A^\prime/A \to \Pi_{A^\prime}$ is a section. The cover is defined over its group of moduli $A^\prime$ if and only if $\overline\phi \circ s \colon A^\prime/A \rightarrow N/C$ admits a lift $\phi:A^\prime/A\rightarrow N$.
\end{thmD}

\subsection*{Birman-Hilden Property}
We address the relationship between the mapping class group of a cover and that of the base. The classical \textit{Birman-Hilden property} asks when mapping classes of the covering space surject to mapping classes of the base space. In our framework, the choice of a cover $X \to (B,A)$ defined over $A \leq \mathrm{Mod}(B^*)$ induces a homomorphism $A \to \mathrm{Out}(\pi_1(X))$. We view this as a dual property to the Birman-Hilden one. While the well-known theory of the Birman-Hilden property focuses on compatibility of isotopy representatives when mapping $\mathrm{Mod}(X)$ to $\mathrm{Mod}(B^*)$, we provide a classical reinterpretation in group theoretic terms to accompany the rest of the article. Furthermore, we explore the connection between this property and the action of $\mathrm{Mod}(B^*)$ on the Teichmüller space $T(X)$ of $X$ as a genus $g$ surface.

\subsection*{Equivariant Categories and Weil's Theorem}
The framework of equivariant categories provides the natural setting for our unification of descent theory. 
A core theme of this work is the realization that Weil's classical descent condition—historically formulated via cocycles $\{ f_{\sigma \tau} = {^\sigma \! f_\tau} \circ f_\sigma \}$—is precisely the condition for an object to admit a \textit{linearization} in an equivariant category. 
By identifying either Weil's descent data or its topological counterpart, mapping class group ascent data, with linearizations, we obtain the following categorical equivalences for both arithmetic varieties and topological covers. 

We denote by $\mathsf{Var}_L$ the category of $L$-varieties and by $\mathsf{Cov}_{(B,A)}$ the category of Galois covers $X\to(B,A)$ with branch divisor in $B$ and defined over $A$ (Definition~\ref{def:category of covers}). 

\begin{thmE} (Proposition \ref{prop:equivariantweildescent} \& \ref{prop:equivariantweildescent-top}).
\textit{Let $L/L_0$ a finite Galois field extension and $A \lhd_{\mathrm{f}} A'$ be subgroups of $\mathrm{Mod}(B^*)$. 
The category of varieties $\mathsf{Var}_L$ is equipped with a natural strict Galois $L/L_0$ action and the category of covers $\mathsf{Cov}_{(B,A)}$ is equipped with a natural non-strict action of the mapping class group quotient $A'/A$. An object $X$ of either of the categories is definable over the base ($L_0$ or $A'$, Definitions \ref{def:fieldofdef} and \ref{def:definable_top}, respectively) if and only if it admits a linearization with respect to the natural group action on the category. In particular, we obtain equivalences of categories:
\[
\mathsf{Var}_L^{\Gal(L/L_0)} \simeq \mathsf{Var}_{L_0} \quad \text{and} \quad \mathsf{Cov}_{(B,A)}^{A'/A} \simeq \mathsf{Cov}_{(B,A')}.
\]}
\end{thmE}

This categorical perspective clarifies the definitions provides a different viewpoint to the group of definition.
Note that both parts of this theorem are induced by Weil's arithmetic and topological theorems respectively.

\subsection*{Equivariant Derived Categories and Weil's Descent}
A finite group action on the category of coherent sheaves $\mathsf{Coh}(X)$ naturally induces an action on its bounded derived category $\mathsf{D^b}(X)$. While the definition of an equivariant category is straightforward for abelian categories, identifying the ``correct'' equivariant structure in the triangulated setting is more subtle. In general, the derived category of the equivariant abelian category, $\mathsf{D^b}(\mathcal{A}^G)$, and the equivariant category of the derived category, $\mathsf{D^b}(\mathcal{A})^G$, are distinct. However, in our setting—where a finite group acts with order coprime to the characteristic of the base field—these two perspectives coincide. Specifically, there is a canonical equivalence $\mathsf{D^b}(\mathsf{Coh}^G(X)) \simeq \mathsf{D^b}(X)^G$, allowing us to interpret the equivariant derived category geometrically as the derived category of the quotient stack $[X/G]$. We thus treat the derived category as a geometric invariant in its own right and investigate its behavior under descent.

Our main goal here is to lift the Weil descent datum from the variety $X$ to its derived category. We first show that a Weil descent datum $\{f_\sigma\}$ on $X$ allows us to construct a canonical action of $G=\Gal(L/L_0)$ by automorphisms on $X$, such that the quotient $X/G$ is isomorphic to the descended variety $Y$. This leads to a modern, categorical extension of Weil's theorem:

\begin{thmF}(Propositions \ref{prop: group action by weil descent}, \ref{prop:actiononvariety} \& \ref{prop:equivariant descent theorem}).
\textit{Let $X$ be a variety defined over $L$ equipped with a Weil descent datum relative to a finite Galois extension $L/L_0$, and let $Y$ be the corresponding $L_0$-variety. Assume, also that $\mathrm{char}(L_0) \not| \; \big|G \big|$, where $G = \Gal(L/L_0)$.
\begin{enumerate}
    \item The descent datum induces a free $G$-action on $X$ such that $X/G \simeq Y$.
    \item There is a well-defined $G$-action on $\mathsf{D^b}(X)$ such that the equivariant derived category is isomorphic to the derived category of the base:
    \[ \mathsf{D^b}(X)^G \simeq \mathsf{D^b}(Y), \]
    provided that $|G|$ is invertible in $L_0$.
\end{enumerate}
}
\end{thmF}

This theorem conveys two messages. First, that the Weil descent datum on $X$ is enough to construct a group action on $X$ such that the quotient $X/G$ is the variety defined over the subfield $L_0$. Second, that given the variety $Y$ defined over $L_0$, then the extension of scalars of $Y \times_{L_0} L $ can be viewed as a ``de-equivariantization'' process in this particular setting, see Remark \ref{rem: de-equivariantization}. 

Finally, in Remark \ref{rem:bondal orlov remark}, we address an ambiguity in the seminal work of Bondal and Orlov \cite{MR1818984} regarding the group of autoequivalences $\mathrm{Aut}(X) \ltimes (\mathrm{Pic}(X) \oplus \mathbb{Z})$ of a smooth projective variety with ample (anti-)canonical sheaf. We clarify that the group $\mathrm{Aut}(X)$ must be interpreted as the group of general scheme automorphisms, rather than $k$-automorphisms, contrary to what one would expect.

\subsection*{Braid Group Actions on Invariants of Elliptic Curves}
In \cite{MR1831820} P. Seidel, R. Thomas, showed that there is a faithful action of the Braid group on Calabi-Yau manifolds of dimension $\geq 2$, which was an indication towards the validity of the homological mirror symmetry conjecture. This work was based on interpreting spherical lagrangian manifolds as spherical objects in $\mathsf{D^b}(X)$ and also the generalized Dehn twists in the lagrangian side as autoequivalences in the $\mathsf{D^b}(X)$ side satisfying the braid relations. This article is mainly focused on curves and the notion of Calabi-Yau manifold reduces to the theory of elliptic curves. However, Seidel and Thomas show that in this case there exists a non faithful $B_4$ action on $\mathsf{D^b}(E)$, where $E$ is an elliptic curve. In Subsection \ref{sec:nonfaithfulelliptic}, we examine the action $B_4$ on $\pi_1(E)$, and show that it is also non faithful. We want to point out that the similarity of these two actions is not only the fact that are both non faithful actions of $B_4$ on some invariant of the elliptic curve, but also the fact that both are not induced from an action on the elliptic curve itself.

\subsection*{Section-by-Section Analysis}
All chapters, with the exception of the final two, follow a tripartite structure: an \textit{arithmetic} subsection, a \textit{topological} subsection, and a \textit{comparison} subsection. The purpose of this structure is to highlight the common ground between the two worlds and introduce a unified notation, which is summarized in tables at the end of each comparison section. The logical flow of the paper is as follows:

\[\begin{tikzcd}
            &                        & 3 \arrow[d]            & 6 \arrow[r]                        & 7 \\
1 \arrow[r] & 2 \arrow[ru] \arrow[r] & 4 \arrow[r] \arrow[ru] & 5 \arrow[ru] \arrow[r] \arrow[rd] & 9 \arrow[r] & 10 \\
            &                        &                        &                                   & 8
\end{tikzcd}
\]

In Section \ref{sec:covers}, we lay the foundation by establishing the correspondence between finite extensions of function fields and branched covers of Riemann surfaces. We explicitly distinguish between \textit{mere} covers and \textit{Galois} covers, setting the stage for the parallel study of the arithmetic fundamental group $\Pi_K(B^*)$ and the topological fundamental group $\pi_1(B^*)$. A key output of this section is the construction of a unified ``Dictionary'' of notation, which allows us to treat arithmetic extensions $K(X)/K(B)$ and topological covers $X \to B$ using a single formalism throughout the sequel.

Section \ref{ch:profinite_braids} is devoted to the study of the absolute Galois group $\Gal(K_s/K)$ as a group of profinite braids. We review the Ihara representation and compare it with the classical Artin representation of the Braid group. We also recall the definition of the Mapping Class Group $\mathrm{Mod}(S)$ and formulate the Dehn-Nielsen-Baer Theorem~\ref{thm:Dehn-Nielsen-Baer}, which serves as our primary tool in the topological setting. The main focus is the comparison of Grothendieck's arithmetic homotopy exact sequence and the Birman exact sequence. We prove Proposition~\ref{prop:ses split}, which offers a unified group-theoretic proof for the splitting of the fundamental exact sequences in both settings, relying on the existence of fixed points. This section also includes Subsection \ref{sec:nonfaithfulelliptic}, where we construct the non-faithful action of $B_4$ on the fundamental group of elliptic curves.

In Section \ref{ch:groupactions}, we develop the theory of group actions on curves and varieties. We begin with a short introduction discussing actions by homeomorphisms on ringed spaces, concluding that non-trivial twisted actions are essential for descent theory. We proceed by defining $\Gal(K_s/K)$-actions on varieties and address the technical challenge of defining $\mathrm{Mod}(B^*)$-actions in the topological setting where no function fields exist. We establish the correspondence between actions on a variety $X$ and actions on its monodromy representation $\Psi$. The main result, Proposition \ref{prop:gal_action_monodromy}, demonstrates that the natural Galois action on a variety is equivalent (up to conjugation) to the action on its monodromy representation, providing a coordinate-free way to study descent problems. Another key point is Definition \ref{def:topological_defined}, where we formally define what it means for a topological cover $X \to B$ to be ``defined over'' a subgroup $A \leq \mathrm{Mod}(B^*)$.

Section \ref{sec:field of moduli, definition, invariance} is devoted to defining clearly the groups \textit{moduli}, \textit{definition}, and \textit{invariance} for both the arithmetic and topological settings. We thoroughly investigate these groups in the arithmetic case and provide suitable definitions for the topological one. We reformulate Weil's descent conditions in terms of these groups and recall his Descent Theorem \ref{thm:weil's descent}. We then prove Theorem \ref{thm:Weil topological}, the topological analogue of Weil’s Descent Theorem, showing that a cover defined over a subgroup $A$ extends to a larger group $A'$ (where $A \lhd_\mathrm{f} A'$) if and only if it admits a ``Mapping Class Group ascent datum.'' We conclude by presenting examples illustrating the distinction between the field of moduli and the field of definition in Subsections~\ref{exm: moduli = definition} and~\ref{exm:field of moduli vs field of definition}.

In Section \ref{sec:regularity}, we analyze the \textit{regularity} of covers—the condition that the constant field extension is trivial (i.e., $K_s \cap K(X) = K$). We demonstrate that this arithmetic property can be captured entirely by group indices. Specifically, Proposition \ref{prop:equalityofindices} establishes that a cover is regular if and only if the index of the geometric fundamental group inside the arithmetic group is preserved ($[\Pi_A:H] = [\Pi_1:R]$). 

Section \ref{ch:descend_on_monodromy} applies the descent theory of Dèbes and Douai to the topological setting. We formulate the ``Group of Moduli'' condition \textbf{(CMod)} in Definition \ref{def:groupofModuliCovers} and systematically analyze the obstructions to the existence of models. The main results, Theorems \ref{thm:Mere1} and \ref{thm:DebesDouaiThB}, decompose the obstruction into two parts: a central constraint and a cohomological class lying in $H^2(A'/A, Z(G))$. This provides a rigorous method for constructing covers where the group of moduli is strictly larger than the group of definition.

In Section \ref{ch:Hilden}, we review the classical \textit{Birman-Hilden property}—the surjection of isotopy classes of homeomorphisms from the cover to the base—and provide a group-theoretic interpretation in order to discuss its connections with the other aspects of the article. We also provide an elementary result (Lemma \ref{lemma:extRout}) concerning the interaction of the marked points of $B^*$ and their non-marked fibers in $X$. 
Specifically, we analyze how the correspoding subgroups of $\mathrm{Out}(\pi_1(X))$ and $\mathrm{Out}(\pi_1(B^*))$ interact in the case of Galois covers. We conclude by raising a question regarding the existence of a profinite analogue to this property.

Section \ref{sec:equivCat} introduces the framework of equivariant categories. After providing the necessary definitions and some well-known examples of group action on categories of sheaves and their derived categories -which are needed for the final Section- we show in Proposition \ref{prop:equivariantweildescent} that the category of varieties defined over a field $L_0$ is equivalent to the equivariant category $\mathsf{Var}_L^{\Gal(L/L_0)}$. Finally, we establish the analogous result for topological covers in Proposition \ref{prop:equivariantweildescent-top}, reinterpreting the Weil cocycle condition in both the arithmetic and topological settings as a linearization of the category.

Finally, in Section \ref{sec: descent using derived categories} we construct a canonical $\Gal(L/L_0)$-action from a Weil descent datum on an $L$-variety $X$ (Proposition \ref{prop: group action by weil descent}) and show that the quotient $X/G$ is the descent $L_0$-variety (Proposition \ref{prop:actiononvariety}). We extend these results to the bounded derived category of coherent sheaves, concluding with Proposition \ref{prop:equivariant descent theorem}, which proves that the equivariant derived category $\mathsf{D^b}(X)^G$ recovers the derived category of the descended variety $\mathsf{D^b}(Y)$. Along the way, we clarify Bondal \& Orlov's work on the group of automorphisms (Remark \ref{rem:bondal orlov remark}) and propose an idea regarding de-equivariantization via extension of scalars (Remark \ref{rem: de-equivariantization}).

\subsection*{Notation \& Conventions}
Throughout the article we fix primes $\ell \neq p$, a field $K$ of characteristic $p$, unless otherwise stated. We denote $K_s$ a separable closure of $K$ (which will be fixed later in the article) and $K \subseteq L \subseteq K_s$ a field extension. 
By algebraic variety (defined) over $L$ we mean an integral scheme $X$ with structure morphism $p\colon X \to \Spe(L)$ which is separated and of finite type.
The symbols $X$, $Y$ and $B$ denote varieties defined over some field in this extension, for example usually $X$ will be over defined over $K_s$, $Y$ over $L$ and $B$ over $K$.
If the extension $L/K$ is Galois, we will denote by $\Gal (L/K)$ the Galois group.
The natural map $K \hookrightarrow L$ corresponds (contravariantly) to map $\Spe L \twoheadrightarrow \Spe K$ of affine varieties, thus given a variety $B$ defined over $K$ we extend it by scalars to the variety $B \times_{\Spe K} \Spe L$ defined over $L$, which by abuse of notation we will denote $B \times_K L$.

We use the symbol $<$ (resp. $\leq$) for strict (resp. non-strict) subgroup. Similarly, we write $\subset$ (resp. $\subseteq$) to indicate a strict (resp. non-strict) inclusion of sets—this applies in particular to subsets, open subsets, and subfields.

\subsection*{Acknowledgments}
The research project is implemented in the framework of H.F.R.I Call “Basic research Financing (Horizontal support of all Sciences)” under the National Recovery and Resilience Plan “Greece 2.0” funded by the European Union Next Generation EU(H.F.R.I.  
Project Number: 14907). 

\begin{center}
\includegraphics[scale=0.4]{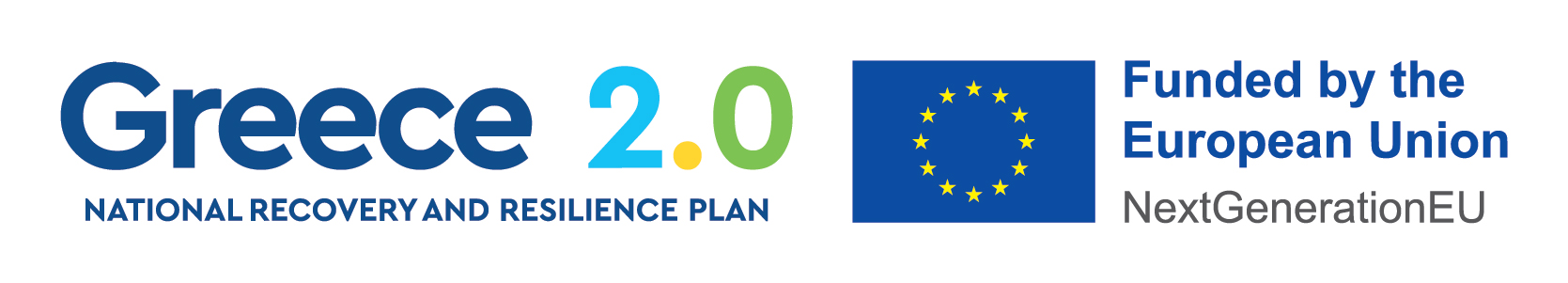}
\hskip 1cm
\includegraphics[scale=0.05]{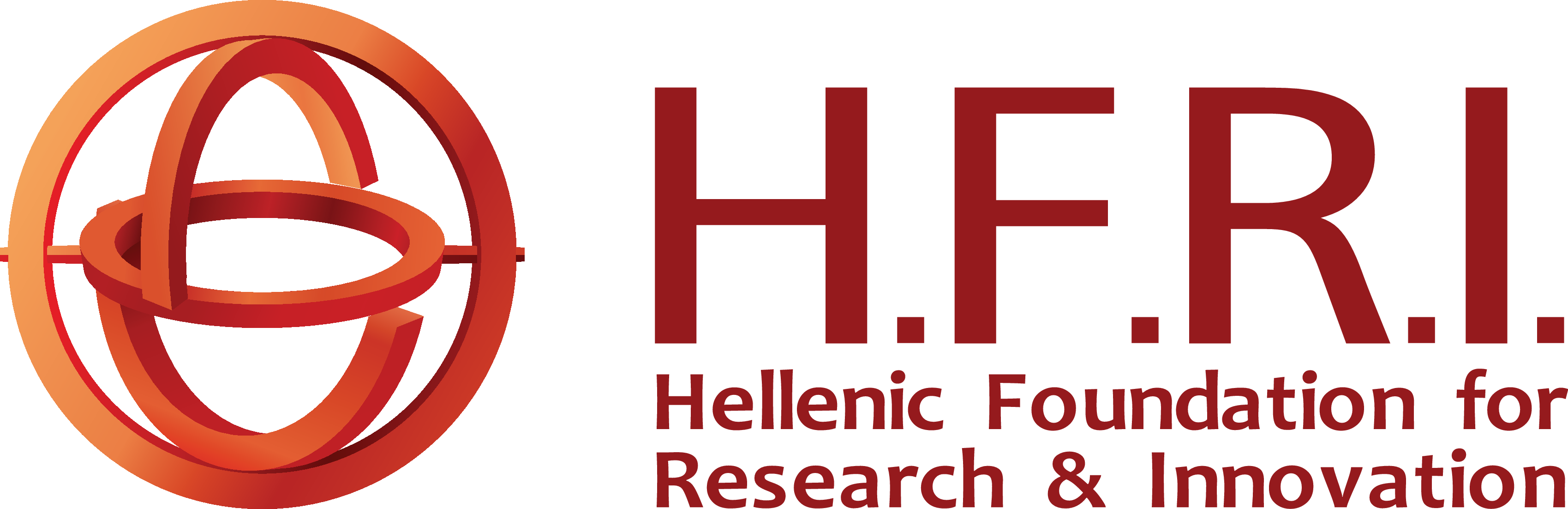}
\end{center}

We also thank Ilias Andreou for inspiring conversations during our time on this project.

\section{Covers}\label{sec:covers}

We begin by establishing the dictionary between arithmetic covers of varieties (viewed as function field extensions) and topological branched covers of complex manifolds. By distinguishing between the topological, geometric, and arithmetic fundamental groups—which classify distinct types of covers—we set the stage for comparing the arithmetic actions of the absolute Galois group $\Gal(K_s/K)$ with the topological actions of the Mapping Class Group. As the Braid group is a special case of the Mapping Class Group, this correspondence naturally motivates the study of the absolute Galois group as a profinite braid group in Chapter \ref{ch:profinite_braids}.

\subsection{Arithmetic Covers}\label{subsec:arcovers}

We follow the approach of \cite{DebesDouai97} in order to define covers in the arithmetic setting.

\begin{definition}
Let $B$ be a regular, projective, geometrically irreducible $K$-variety. 
\begin{itemize}
  \item 
A  {\em mere cover} of $B$ over $K$ is a finite and generically unramified morphism $f\colon X \rightarrow B$  of $K$-varieties, where $X$ is assumed to be normal and geometrically irreducible. 
\item  A {\em Galois cover} is a mere cover $X \to B$ with $ |\Gal (K(X) / K(B)) |= [ K(X) : K(B)]  $, where $K(X)$ and $K(B)$ denote the function fields of the varieties $X$ and $B$, respectively. 
 \item 
{\em $G$-cover} of $B$ is a Galois cover $X \rightarrow B$ over $K$, together with an isomorphism $h\colon G \rightarrow \Gal(K(X)/K(B))$.
\end{itemize}
\end{definition}
Note that regularity of $B$ implies that $K_s \cap K(B) = K$, for any separable closure $K_s$ of $K$. To extend this we introduce the following definition, which we will explore more in Section \ref{sec:regularity}. 
\begin{definition}
  The extension $K(X)/K(B)$ is called {\em regular} if  $K_s \cap K(X)= K$.
\end{definition} 

There is a bijection realized by the function field functor:
\[
\Biggl\{\text{ mere cover } f \colon X \to B \ \Biggl\} \ \ \longleftrightarrow \ \  \Biggl\{ \begin{array}{c}
\text{finite, separable, regular field} \\
\text{extensions } K(X)/K(B)
\end{array} \Biggl\}
\]

When the cover is defined over $K$, it can be viewed as being defined over any intermediate field extension $K \subset L \subset K_s$ by extension of scalars $X\times_K L \to B \times_K L$. Moreover, if $B$ is a $K$-variety and $X$ is an $L$-variety, then the cover $X \to B \times_K L$ is a cover over $L$. By abuse of notation, we denote it also by $X \to B$.

To every mere cover $f \colon X\rightarrow B$ over $K$, and further extended over $K_s$,  we assign the Galois group of the cover, that is the group 
$\Gal(\widehat{K_s(X)}/K_s(B))$, where  $\widehat{K_s(X)}$ denotes the Galois closure of the extension $K_s(X)/K_s(B)$. 
We also assign the {\em ramification locus }  $D \subset B$, which is a divisor with simple components, i.e. each component appears at most once in the formal sum. The ramification divisor of a cover defined over $K$ is also defined over $L$. Throughout the article, we will assume that $D$ is $\Gal (K_s/K)$-invariant. This assumption may be recalled explicitly in later sections as needed. We denote the complement of the ramification divisor by $B^*= B - D$. 

Fix a separable closure $K(B)_s$ of $K(B)$ and let $K_s \subset K(B)_s$ be the separable closure of $K $ inside $K(B)_s$. For every intermediate field $K \subset L \subset K_s$, the {\em arithmetic fundamental group} of $B^*$ over $L$, denoted by $ \Pi_L(B^*)$, is defined as follows.

 Let $\Omega_D \subset K(B)_s$ be the maximal algebraic separable extension of $K(B)$ that is unramified above $B^*$. Then, $\Pi_L(B^*)\coloneqq \Gal(\Omega_D/L(B))$, as shown in the following diagram.
\begin{equation} \label{eq:diagFields}
  \xymatrix{
  & K(B)_s \ar@{-}[d]
  \\
  & \Omega_D \ar@{-}[d]^{\Pi_{L}(B^*)} \ar@{-}@/_1.5pc/[ldd]_{\Pi_K(B^*)} 
  \\
  & L(B) \ar@{-}[d] \ar@{-}[dl]
  \\
  K(B) \ar@{-}[d] & L \ar@{-}[dl]
  \\
  K
  }
\end{equation}

\vspace{0.5cm}

We will focus mostly in the case $L = K_s$, i.e. on the group $\Pi_{K_s}(B^*)$, which is called the \textit{geometric fundamental group}.

\begin{remark}{\label{rem:SGA}}
In the context of the Grothendieck's SGA1 \cite{SGA1} the group $\Pi_K(B^*)$ is denoted as $\pi_1(B^*,\xi)$, where $\xi$ is the geometric generic point of $B$ corresponding to the choice $K(B)\xhookrightarrow{} K(B)_s$, or equivalently to the map $\Spe(K(B)_s)\rightarrow \Spe(K(B))\rightarrow B^*$. 
\end{remark}

\subsection{Topological Covers}

Let us recall some basic notions from algebraic topology, as a reference containing the theory of Riemann surfaces and their branched covers we will rely upon, we refer to \cite{MR1185074}. Let $X,B$ be topological spaces, where $B$ is connected. 
Let $X,B$ be topological spaces, with $B$ assumed to be connected. 

A continuous map $f\colon X\rightarrow B$ is called a {\em topological covering} (or {\em unramified covering}) if, for every point $b \in B$, there exists an open neighborhood $V_b \subset B$ such that $f^{-1}(V_b)$ is a disjoint union of open sets $U_i \subset X$, for some indexing set $I$, 
and the restriction $f|_{U_i} \colon U_i \to V_b$ is a homeomorphism for every $i \in I$. 
The cardinality of the indexing set $I$ is equal to the number of the preimages of $f^{-1}(b)$. Note that this does not depend on $b$. If $\# f^{-1}(b)~=~d~<~\infty$, we say that $f$ is called a {\em finite cover of degree} $d$.

If $X$ and $B$ are complex manifolds and $f \colon X \to B$ is a holomorphic map of degree $d$ (in usual complex analytic sense; e.g. \cite[Def. 1.6]{MR1185074}), we define the {\em ramification locus} $D \subset B$ to be the set of points $b \in B$ such that $\# f^{-1}(b) < d$. 
Let $X^\circ = X-f^{-1}(D)$ and $B^* = B-D$. We say that $f$ is a {\em finite branched cover of degree} $d$ if the restriction $f|_{X^\circ} \colon X^\circ \rightarrow B^*$ is a topological cover of degree $d$, as described above. 
 Recall the following notion:
\begin{definition}
For a topological (resp. branched) cover $X \rightarrow B$, the {\em group of deck transformations} $\operatorname{Deck}(X/B)$ consists of all homeomorphisms (resp. holomorphic maps) $h\colon X \rightarrow X$ which make the following diagram commute:
\[
\xymatrix{ X \ar[rr]^h \ar[dr] &  &  X \ar[dl] \\
& B & }
\]
\end{definition}
For a branched cover $f \colon X\rightarrow B$ and its associated unramified cover 
$f|_{X^\circ} \colon X^\circ \rightarrow B^*$, the groups of deck transformations are the same, that is,  
$\operatorname{Deck}(X/B) = \operatorname{Deck}(X^\circ/B^*)$, see 
\cite[Thm. 8.5]{MR1185074}.

The loops on $B^*$, which are continuously deformable paths with same starting and endpoint $b_0$, form a group (with multiplication given by concatenation of paths) known as the {\em topological fundamental group} $\pi_1^{\operatorname{top}}(B^*,b_0)$.  The choice of base point is irrelevant, but in many cases we will need to keep track of this initial choice, otherwise we will simply write $\pi_1(B^*)$. 
There exists a space $\widetilde{B^*}$ over $B^*$, known as the {\em universal covering space}, which realizes the fundamental group as $\pi_1(B^*) = \operatorname{Deck}(\widetilde{B^*}/B^*)$.
Covering space theory asserts the following statement, which is similar to Galois theory: each (finite) covering $X^\circ \to B^*$ corresponds to a (finite index) subgroup $R \leq \pi_1(B^*)$, where $R = \operatorname{Deck}(\widetilde{B^*}/X^\circ)$. Furthermore, $\pi_1(B^*)$ acts on the fiber $f^{-1}(b)$ of the basepoint. If the action is transitive, we say that $X^\circ \to B^*$ is Galois, which is equivalent to $R$ being a normal in $\pi_1(B^*)$, see \cite[Ch. 2, Prop. 2.2.7]{MR2548205}. In this case, we have an isomorphism $\pi_1(B^*)/R \simeq \operatorname{Deck}(X^\circ/B^*)$. If the subgroup $R$ is not normal, then $\operatorname{Deck}(X^\circ/B^*) \simeq  N_{\pi_1(B^*)}R/R$, where $N_{\pi_1(B^*)}R$ denotes the normalizer of $R$ in $\pi_1(B^*)$, see \cite[Exer. 13.12, P.184]{MR1343250}. Thus, we arrive at the following terminology.

\begin{itemize}
\item A  {\em mere cover} of $B$ is a (possibly branched) cover $f\colon X \rightarrow B$  of topological spaces over $\mathbb{C}$, with associated cover $X^\circ \rightarrow B^*$.
\item  A {\em Galois cover} is a mere cover $X \to B$, where $R$ is a normal subgroup of $\pi_1(B^*)$ corresponding to the associated cover.
\item A {\em $G$-cover} of $B$ is a Galois cover $X \rightarrow B$, together with an isomorphism $h\colon G \rightarrow \operatorname{Deck}(X/B)$.
\end{itemize}

\subsection{Comparison - Common Notation: Covers}

By the previous discussion, we observe that on the one hand we have finite extensions of function fields 
$K_s(X)/K_s(B)$, in which only finitely many places ramify. On the other hand, we have branched covers $X\rightarrow B$, in which ramification occurs over a finite set of points (the branch locus). These two types of geometric covers are classified by the arithmetic fundamental group $\Pi_{K_s}(B^*)$ and the topological fundamental group $\pi_1(B^*)$, respectively.

If the field $K$ has characteristic zero, then by the GAGA theorems, there is a correspondence: covers $K_s(X)/K_s(B)$ correspond to branched covers of analytic spaces $X(\mathbb{C}) \rightarrow B(\mathbb{C})$ unramified outside $D$, which correspond to topological covers $X^\circ \rightarrow B^*$ with the complex topology. In this case, $\Pi_{K_s}(B^*)$ is the profinite completion of $\pi_1(B^*)$. 

\begin{remark}{\label{rem:profComp}}
The same principle holds in other settings where a geometric covering theory exists,
for example, regarding Mumford curves see 
\cite{CKNotices}, \cite{CKK}, \cite{mumford-curves}, or in the case of Berkovich spaces \cite{Berk:90}. In such cases as well, the arithmetic fundamental group $\Pi_{K_s}(B^*)$ is the profinite completion of $\pi_1(B^*)$.
\end{remark}

In practice, one often considers covers of degree a power of $\ell$. We can define $\Omega_{D,\ell}$ to be the $\ell$-part of $\Omega_D$, such that for every finite subextension $L(X)/K(B)$ in $\Omega_{D,\ell}$ we have that $[L(X):K(B)]=\ell^m$, for some $m\geq 1$. Then $\Gal(\Omega_{D,\ell}/K_s(B))$ is the pro-$\ell$ completion of $\pi_1(B^*)$.

{\small
\begin{longtable}[c]{>{\raggedright\arraybackslash}p{5cm} >{\raggedright\arraybackslash}p{4.5cm} >{\raggedright\arraybackslash}p{3.5cm}}

\rowcolor{headerblue}
\textbf{Number Theory} & \textbf{Topology} & \textbf{Common Notation} \\
\endfirsthead

\rowcolor{headerblue}
\textbf{Number Theory} & \textbf{Topology} & \textbf{Common Notation} \\
\endhead

$K(X)/K(B)$ \linebreak regular extension & 
$f\colon X \to B$ \linebreak branched cover & 
$f\colon X \to B$, \linebreak ramification locus $D$, \linebreak $B^* = B - D$ \\

\rowcolor{rowgray}
\begin{tabular}[t]{@{}l@{}}
$\Omega_D =$ Maximal unramified \\
extension of $K(B)$ outside $D$
\end{tabular} &
\begin{tabular}[t]{@{}l@{}}
$\widetilde{B^*} =$ Universal \\
covering space of $B - D$
\end{tabular} &
\\

\begin{tabular}[t]{@{}l@{}}
$\Pi_{K_s}(B^*) =$ Geometric \\
fundamental group
\end{tabular} &
\begin{tabular}[t]{@{}l@{}}
$\pi_1(B^*) =$ Topological \\
fundamental group
\end{tabular} &
\\

\rowcolor{rowgray}
$\Gal(\Omega_D/K_s(X))$ &
$\operatorname{Deck}(\widetilde{B^*}/X)$ &
\begin{tabular}[t]{@{}l@{}}
$R \leq \Pi_{K_s}(B^*)$, or \\
$R \leq \pi_1(B^*)$
\end{tabular}
\\

\end{longtable}
}

\section{Braids and Profinite Braids}\label{ch:profinite_braids}
\subsection{Arithmetic} \label{sec:profinite_braids}
In this section, let $B$ be a regular, projective, geometrically irreducible $K$-variety. This implies that for any field extension $L/K$, we have $L \cap K(B) = K$. Therefore, $\Gal (L(B) / K(B)) \simeq \Gal(L/K)$.
By the diagram of fields shown in (\ref{eq:diagFields}), we have that 
the arithmetic fundamental groups for the extension $L/K$ fit into the following short exact sequence:
\[
  1 \rightarrow \Pi_L(B^*) \rightarrow \Pi_K(B^*) \rightarrow \Gal(L/K) \rightarrow 1
\]

where the quotient map is simply the restriction $\sigma\mapsto \sigma|_{L(B)}$ from $\Omega_D$ to $L$. In particular, for $L=K_s$, we obtain the short exact sequence:
\begin{equation}
\label{ses}
  1 \rightarrow \Pi_{K_s}(B^*) \rightarrow \Pi_K(B^*) \rightarrow \Gal(K_s/K) \rightarrow 1.
\end{equation}

Whenever we have an exact sequence of profinite groups of the form:
\[1\rightarrow N \rightarrow G \rightarrow H\rightarrow 1,\] 
the quotient $H$ acts on $N$ by conjugation via its preimages in $G$. However, this action is only well-defined  up to inner automorphisms of $N$, which gives rise to a continuous homomorphism $H \rightarrow \Out(N)$. 
In our case, this map is: 
\[\phi_K \colon \Gal(K_s/K) \longrightarrow \Out(\Pi_{K_s}(B^*)).\]
Interesting questions arise regarding this homomorphism, such as: what is its image or kernel for various fields $K$ and varieties $B^*$? 

A profound application of this construction is the case $K=\mathbb{Q}$ and $B=\mathbb{P}^1$ with ramification locus $D = \{P_1,\ldots,P_s\}$. 
The topological fundamental group of the corresponding punctured curve, $B^*=\mathbb{P}^1-\{P_1,\ldots,P_s\}$,  is the free group $F_{s-1}$ on $s-1$ free generators, with presentation
\[
F_{s-1} = \langle x_1,x_2,\ldots,x_{s-1},x_s \mid x_1x_2\cdots x_{s-1}x_s\rangle ,
\] 
where each generator $x_i$ represents a homotopy class of loops around the puncture $P_i$. Therefore, $\Pi_{\overline{\mathbb{Q}}}(B^*)$ becomes the free profinite group $\mathfrak{F}_{s-1}$ on $s-1$ generators. 
If we instead replace $\Omega_D$ with $\Omega_{D,\ell}$, then $\Pi_{\overline{\mathbb{Q}}}(B^*)$  becomes the free pro-$\ell$ group on $s-1$ generators. 

In the 1980's several distinguished mathematicians, including Grothendieck \cite{MR1483107}, Deligne \cite{Deligne89}, Belyi \cite{Belyi1} and Ihara \cite{Ihara1985-it}, independently studied these types of large Galois representations. For the purposes of this article we adopt Ihara's viewpoint on profinite braids, as developed in \cite{Ihara1985-it}. For a modern textbook treatment, we refer the reader to \cite[Ch.4]{MR2548205}. 

In the seminal paper \cite{Ihara1985-it}, Ihara considered the case $B^* = \mathbb{P}^1-\{0,1,\infty\}$ in the pro-$\ell$ setting and proved that the image of $\phi_{\mathbb{Q}} \colon \Gal(\overline{\Q}/\Q) \to  \Out(\mathfrak{F}_2)$ is given by
\[ \im (\phi_\Q) = \{\sigma \in \A(\mathfrak{F}_2):  \sigma(x_i)\cong x_i^\alpha, 1\leq i \leq 3, \ \alpha \in \Z_\ell^* \}/\mathrm{Inn}(\mathfrak{F}_2),\]
where the symbol ``$\cong$" denotes conjugation in $ \mathfrak{F}_2$ and $\alpha$ depends only on $\sigma$, not on the choise of generators $x_i$. In fact, $\alpha$ coincides with the $\ell$-cyclotomic character $\Gal(\overline{\Q}/\Q)\rightarrow \Z_\ell^*$. 

More generally, Ihara considered a pro-$\ell$ group $\mathfrak{B}$ and, for elements $x_1, \ldots, x_s$ generating $\mathfrak{B}$, defined the group of {\em profinite braids} as
\[
  \mathrm{Brd}(\mathfrak{B};x_1, \ldots, x_s)=\{\sigma \in \mathrm{Aut} \mathfrak{B}: \sigma(x_i) \cong  x_i^\alpha, 1\leq i \leq s
  \text{ for some } \alpha \in \mathbb{Z}_\ell^*
  \}/\mathrm{Inn} \mathfrak{B},
\] 
which contains the normal subgroup
\[
  \mathrm{Brd}_1(\mathfrak{B};x_1, \ldots, x_s)=\{\sigma \in \mathrm{Aut} \mathfrak{B}: \sigma(x_i) \cong  x_i, 1\leq i \leq s
  \}/\mathrm{Inn} \mathfrak{B}.
\]
In this particular case, one can define a norm map:
\begin{align*}
  N \colon \mathrm{Brd}(\mathfrak{B}; x_1, \ldots, x_s) &\longrightarrow \mathbb{Z}_\ell^* \notag \\
  \sigma &\longmapsto \alpha \quad \text{such that } \sigma(x_i) \cong x_i^\alpha \text{ for all } i.
\end{align*}
Ihara went on to prove that this yields the short exact sequence

\begin{equation}\label{BrdNorm}
 1 \rightarrow 
 \mathrm{Brd}_1(\mathfrak{B};x_1, \ldots, x_s) \longrightarrow 
 \mathrm{Brd}(\mathfrak{B};x_1, \ldots, x_s) 
 \stackrel{N}{\longrightarrow} \Z_\ell^* \longrightarrow  1.
\end{equation}

Once we have discussed mapping class groups, we will be able to view the group $\mathrm{Brd}_1(\mathfrak{B};x_1, \ldots, x_s)$
as a profinite analogue of the unoriented mapping class group of $B^*$. 
In the pro-$\ell$ setting, for $B^* = \mathbb{P}^1-\{P_1,\ldots,P_s\}$, we have that

\begin{align*}
    \phi_K (\Gal(K_s /K))  &\leq \mathrm{Brd}(\mathfrak{F}_{s-1};x_1, \ldots, x_s), \\
    \phi_K (\Gal(K_s /K(\mu_{\ell^\infty})) &\leq \mathrm{Brd}_1(\mathfrak{F}_{s-1};x_1, \ldots, x_s),
\end{align*}
where $K(\mu_{\ell^{\infty}})=\bigcup\limits_{n=1}^\infty K(\mu_{\ell^n})$ and $\mu_{\ell^n}$ denotes the group of $\ell^n$-th roots of unity.

The refinement of the target of $\phi_K$ to $\mathrm{Brd}(\mathfrak{F}_{s-1};x_1, \ldots, x_s)$ is known as {\em Ihara's representation} and it is the map $\mathrm{Ih} \colon \Gal(K_s/K) \to \mathrm{Brd}(\mathfrak{F}_{s-1};x_1, \ldots, x_s)$. We refer the elements in the image of $\mathrm{Ih}$ as ``arithmetic''. 

\begin{remark}
    For $B^* = \mathbb{P}^1-\{0,1,\infty\}$, Ihara and Anderson in \cite{AndersonIhara88} gave a complete description of the kernel and the fixed field $\Omega_{D,\ell}^{\ker \phi_\Q}$. They showed that this fixed field is a non-abelian pro-$\ell$ extension of $\Q(\mu_{\ell^\infty})$, unramified outside $\ell$. Then they posed the question:
    Is $\Omega_{D,\ell}^{\ker \phi_\Q}$ the maximal pro-$\ell$ extension of $\Q(\mu_{\ell^\infty})$ unramified outside $\ell$? 

    In the case where $\ell$ is an odd regular prime, the answer is yes.  This is a deep result that follows as a consequence, by work of Sharifi in \cite{MR1935409} of the Deligne-Ihara Conjecture, proved in the recent years by Brown \cite{MR2993755}. For all primes, the question remains yet open. Furthermore, related to Ihara's question are the torsion points of abelian varieties arising as Jacobians of curves being covers of $\mathbb{P}^1-\{0,1,\infty\}$, in the sense that there are cases of Jacobian varieties $J$ with $\ell$-power torsion points $J[\ell^\infty]$ such that $\mathbb{Q}(J[\ell^n])$ is contained in the maximal pro-$\ell$ extension of $\mathbb{Q}(\mu_{\ell^\infty})$ unramified outside $\ell$. Thus the question can possibly be approached by examples of such curves. See \cite{MR2470396} \cite{MR3592515} for the recent notion of {\em heavenly} abelian varieties $A$ and related conjectures, which satisfy that $K(A[\ell^\infty]) \subset$ maximal extension of $K(\mu_{\ell^\infty})$ unramified outside $\ell$.
\end{remark}

In the profinite setting, one has to consider $\Omega_D$ and $\hat{\Z}$, instead of $\Omega_{D,\ell}$ and $\Z_\ell$, and a profinite group $\mathfrak{B}$, instead of a pro-$\ell$ group, in the construction of the profinite braid group. However, the previous remark for $B^* = \mathbb{P}^1-\{0,1,\infty\}$ translates very differently in the profinite case, due to Belyi's Theorem \cite{Belyi1}. 
This theorem states that every algebraic curve $X$ over $\mathbb{C}$ admits a $\overline{\Q}$-model if and only if there exists a branched covering map $f\colon X\rightarrow \mathbb{P}^1$, branched at three points in $\mathbb{P}^1$. Consequently, the map $\phi_\Q $  becomes injective, i.e. $ \phi_\Q \colon  \Gal( \overline{\Q} / \Q) \hookrightarrow \Out (\mathfrak{F}_2) $, which is very interesting in its own right, as it embeds the absolute Galois group over $\mathbb{Q}$, which originally arises from the study of number fields, into the outer automorphism group of the free profinite group on two generators, which is a purely combinatorial object.

Belyi's Theorem also provides a useful section of the exact sequence \ref{ses}, the so-called {\em Belyi representative}, which Ihara in \cite{Ihara1985-it} employed to study profinite braids, by lifting the image of $\phi_\Q$ to $\A(\mathfrak{F}_2)$. Another similar and important construction construction is that of Deligne's {\em tangential base points}, introduced in \cite{Deligne89}. 
\begin{remark}{\label{rem:seqsplit}}
  In full generality, it is a well-known fact that a choice of a $K$-rational point as a basepoint in $B^*$ makes the exact sequence (\ref{ses}) split. This is proved by embedding $\Omega_D$ inside the completed field $K_s((t))$, where $t$ is a local uniformizer at the $K$-rational point; see \cite[p.220]{MR1352273}.
\end{remark}

In the topological setting, a similar splitting condition, can be formulated in terms of mapping class groups. In the following subsection, we will provide a group-theoretic proof that applies to both the topological and arithmetic contexts, see Proposition~\ref{prop:ses split}.

\subsection{Topological}
\label{sec:braids}

Consider an oriented surface $S$, and let
$\mathrm{Homeo}^+(S)$ be the group of isotopy classes of orientation preserving homeomorphisms of $S$. Let $\mathrm{Homeo}^0(S)$
be the connected component of the identity in the compact-open topology.
The {\em mapping class group} $\mathrm{Mod}(S)$ of $S$ is the quotient:
\[
\mathrm{Mod}(S)=\mathrm{Homeo}^+(S)/\mathrm{Homeo}^0(S).
\]
We also denote by $\mathrm{Mod}^{\pm}(S)$ the extended mapping class group, which is the group of isotopy classes of all homeomorphisms of $S$, including the orientation-reversing ones. 
There is a well-known exact sequence of groups
  \begin{equation}
    \label{eq:sesMod}
    1  \rightarrow \mathrm{Mod}(S) \rightarrow \mathrm{Mod}^{\pm}(S)
    \rightarrow \Z^* \rightarrow 1,
  \end{equation}
where $\Z^*=\{\pm 1\}$ is the group of units of the ring $\Z$; see \cite[ch. 8]{FarbMagalit}. This can be thought of as the discrete analogue of the exact sequence (\ref{BrdNorm}), once we interpret these groups via outer automorphisms of the topological fundamental group and explore their connection with braid groups. 

The {\em Artin braid group} $B_{s-1}$ on $s-1$ strands, is the group generated by $s-2$ generators $\sigma_1,\ldots,\sigma_{s-2}$ subject to the braid relations:
\begin{align*}
    \sigma_i\sigma_j &= \sigma_j \sigma_i, & \textrm{ if } \ |i-j|\geq 2, \\
    \sigma_i \sigma_{i+1}\sigma_i &= \sigma_{i+1}\sigma_i\sigma_{i+1}, & i =1,2,\ldots,s-2.
\end{align*}

To connect with the number-theoretic perspective, we consider its faithful representation as the group of automorphisms of a free group $F_{s-1}$, known as {\em Artin representation}; see \cite[Cor. 1.8.3]{BirmanBraids}. More precisely, there is a group homomorphism
\[
  \mathrm{Ar}\colon B_{s-1} \to \mathrm{Aut}(F_{s-1}),
\]
where, by slight abuse of notation, each braid generator $\sigma_i$ is mapped to an automorphism $\sigma_i$ of $F_{s-1} = \langle x_1,x_2,\ldots,x_{s-1},x_s \mid x_1x_2\cdots x_{s-1}x_s=1\rangle $ satisfying
\[\sigma_i(x_k) = \begin{cases} 
x_i & \textrm{ if } \ k=i+1, \\
x_i x_{i+1}x_i^{-1} & \textrm{ if } \ k=i, \\
x_k & \textrm{ otherwise.}
\end{cases}\] 
In particular, there are $s-2$ generators $\sigma_i$, for all of which holds $\sigma_i(x_1x_2\cdots x_{s-1}) = x_1x_2\cdots x_{s-1}$. There is a normal subgroup $P_{s-1}$ of $B_{s-1}$, the so-called {\em pure braid group} on $s-1$ strands, which, under Artin's representation, is identified with
\[ P_{s-1} = \{ \sigma \in \A (F_{s-1}) : \sigma(x_i) \cong x_i, \  i=1,2,\ldots,s \},\] where, again, ``$\cong$" denotes conjugation. This justifies the ``profinite braids" term for the corresponding arithmetic objects. Although there is no need to quotient by inner automorphisms in this case, a quotient of the braid group will arise as the mapping class group of $\mathbb{P}^1$ minus $s$ points, which will properly align with the arithmetic setting.  

We now explain how braid groups arise topologically as mapping class groups. Let $D_{s-1}$ denote the disc with $s-1$ marked points and let $\overline{D}_{s-1}$ be the disc with $s-1$ open disks removed, creating $s-1$ internal cyclic boundaries. 
For $D_{s-1}$, if we assume that each homeomorphism preserves the set of marked points, then we obtain the full braid group $B_{s-1}$ as $\mathrm{Mod}(D_{s-1})$, see \cite[9.1.3]{FarbMagalit}, \cite[Ch.1]{MR2435235}. 
If we instead assume that the set of marked point is fixed pointwise, then we obtain the pure braid group $P_{s-1}$ as the mapping class group. 
For, $\overline{D}_{s-1}$ if we require that each homeomorphism reduces to the identity on the boundary,
then we obtain as mapping class group the pure framed braid group $\mathcal{H}_{s-1}=\mathbb{Z}^{s-1}
\rtimes ~P_{s-1}$; see \cite{Bodigheimer2012-gr}, \cite[Ch. 2, P. 45]{FarbMagalit}, \cite[Th. 7.6]{Prasolov97}.

\begin{remark}
The spaces $D_{s-1}$ and $\overline{D}_{s-1}$ have isomorphic fundamental groups, but distinct mapping class groups.  Indeed, there is the ``capping map'' obtained by capping  each boundary component by a once-punctured  disk. By \cite[Proposotion~3.19]{FarbMagalit} we have the following exact sequence:
\[
  1 \rightarrow \langle T_b \rangle \rightarrow \mathrm{Mod}(S,\{p_1,\ldots,p_k\})
  \rightarrow 
  \mathrm{Mod}(S',\{p_0,\ldots,p_k\})  \rightarrow 1, 
\]
where $S'$ is the surface obtained by capping the boundary component $b$ with a once-marked disk, which we call $p_0$. Let $\mathrm{Mod}(S,\{p_1,\ldots,p_k\})$ denote the subgroup of $\mathrm{Mod}(S)$ consisting of the elements fixing the punctures $p_1,\ldots,p_k$, $k\geq 1$. The kernel $\langle T_b \rangle$ consists of the Dehn twist corresponding to the boundary component $b$. 
\end{remark}

Recall that $S_{g,n}$ denotes a surface, where $g$ is the genus of the surface and $n$ is the number of punctures. 

\begin{example} \label{rem:sphere}
The mapping class group of $\mathbb{P}^1-\{P_1,\ldots,P_s\}$ is the spherical braid group on $s$ strands, denoted $\mathrm{Mod}(S_{0,s})$. It is the quotient of the braid group $B_{s-1}$ by two additional relations, and has the following presentation \cite[P. 123]{FarbMagalit}:
\begin{equation}
\label{eq:sphMod}
  \mathrm{Mod}(S_{0,s})=
\left\langle \sigma_1,\ldots,\sigma_{s-1} \ \Big| \ 
\begin{array}{l} 
\sigma_i \sigma_j = \sigma_j \sigma_i, \ |i-j|\geq 2,  
\\
\sigma_i \sigma_{i+1} \sigma_i = \sigma_{i+1} \sigma_i \sigma_{i+1}, \ 1\leq i \leq s-2, 
\\
(\sigma_1 \cdots \sigma_{s-1})^s=1,
\\  (\sigma_1 \cdots \sigma_{s-1} \sigma_{s-1}\cdots \sigma_1)=1
\end{array}
\right \rangle.
\end{equation}

It is interesting to consider the situation where one puncture, say $P_s = \infty$, is fixed pointwise, reflecting the kind of fixed behavior one might expect under the action of $\Gal(\overline{\Q}/\Q)$.
Observe that $\mathbb{P}^1-\{\infty\}$ is homeomorphic to the open disk $D^\circ$.
The remaining punctures $\{P_2,\ldots,P_s\} \subset D^\circ $ can be treated as marked points, and if we allow homeomorphisms that permute these points, the mapping class group becomes 
$$\mathrm{Mod}(D^\circ - \{P_2,\ldots,P_s\}) = B_{s-1}/Z(B_{s-1}),$$ 
where $Z(B_n)$ is the center of $B_n$, generated by the Dehn twist $(\sigma_1 \cdots \sigma_{n-1})^n$. This is done by capping the boundary at the $\infty$ puncture from the closed disk $D_{s-1}$. The capping operation yields the exact sequence 
\[
  1\rightarrow Z(B_{s-1}) \rightarrow \mathrm{Mod}(D_{s-1})=B_{s-1} \rightarrow \mathrm{Mod}(S_{0,s}),
\] 
which shows that $B_{s-1}/Z(B_{s-1})$ is isomorphic to a subgroup of $\mathrm{Mod}(S_{0,n})$ consisting of elements that fix the puncture at $\infty$. Moreover, this subgroup if of index $s-1$. 
See \cite[P. 248]{FarbMagalit} for a more detailed discussion. 
This is aligned to the fact that we can pick $x_s \in F_{s-1} = \pi_1(\mathbb{P}^1-\{P_1,\ldots,P_{s-1},\infty\})$ to be the homotopy class of loops around the puncture at $\infty$. Then, since any braid in $B_{s-1}$ fixes $x_1x_2\cdots x_{s-1}$, it also fixes $x_s$.
\end{example}

The algebraic structure of mapping class groups of surfaces $S_{g,n}$, both punctured and unpunctured, is given by the following theorem. We use the notation $S_g=S_{g,0}$ when there are no punctures.
Recall that $S_{g,n}$ is {\em hyperbolic} if $2-2g-n<0$. The group $\mathrm{Out}^*(\pi_1(S_{g,n},b_0))$ denotes the subgroup of $\mathrm{Out}(\pi_1(S_{g,n},b_0))$ consisting of all the elements that preserve the set of conjugacy classes of the simple closed curves surrounding individual punctures.
\begin{theorem}[Dehn-Nielsen-Baer]{\label{thm:Dehn-Nielsen-Baer}
  For $g\geq 1$ and $S_g$ ``closed" surface, that is a compact surface without boundary, there is an isomorphism 
  \[
    \mathrm{Mod}^{\pm}(S_g) \cong \mathrm{Out}(\pi_1(S_g,b_0)).
  \]

  For $g\geq 0$ and $S_{g,n}$ hyperbolic surface, there is a canonical isomorphisms of groups
  \[
    \mathrm{Mod}^{\pm}(S_{g,n}) \cong \mathrm{Out}^*(\pi_1(S_{g,n},b_0)).
  \]
}
\end{theorem}
\begin{proof}
  See \cite[Th. 8.1]{FarbMagalit} and \cite[Th. 8.8]{FarbMagalit}.
\end{proof}

Building on the above discussion, we now construct an exact sequence that serves as a topological analogue of the arithmetic short exact sequence (\ref{ses}).

For an element $h$ in $\mathrm{Mod}(B^*)$ we denote by $h_* \in \mathrm{Out}^*(\pi_1(S_{g,n},b_0)) $ its corresponding class of automorphisms. Let $\pi \colon X^\circ \rightarrow B^*$ be a topological cover of hyperbolic surfaces and let $R\leq \pi_1(B^*,b_0)$ be a subgroup such that $\widetilde{B^*}/R = X^\circ$. 
Assume that there is a subgroup $A \leq \mathrm{Mod}(B^*)$ such that for all $a$ in $A$, the induced automorphism $a_*([R]) = [R]$, where $[R]$ denotes the conjugacy class of $R \leq \pi_1(B^*,b_0)$. Associated to this data, i.e. the cover $\pi\colon X^\circ \rightarrow B^*$ corresponding to $R$ and the group $A$, we assign the exact sequence
\begin{equation}
    1\rightarrow \pi_1(B^*,b_0) \rightarrow \Pi_{A,X^\circ} \rightarrow A \rightarrow 1,
\end{equation}
where $\Pi_{A,X^\circ}$ is the subgroup of $\{ \sigma \in \A(\pi_1(B^*,b_0))\colon \ \sigma([R]) = [R] \}$ corresponding to $A$ under the isomorphism of Theorem \ref{thm:Dehn-Nielsen-Baer}, since $\mathrm{Mod}(B^*) \hookrightarrow \mathrm{Mod}^{\pm}(B^*)$. The kernel of $\Pi_{A,X^\circ} \to A$ is the group 
$\mathrm{Inn}(\pi_1(B^*,b_0))$. However, in the hyperbolic case, the group $\pi_1(B^*, b_0)$, or
more generally any Fuchsian group, is centerless. Thus its inner automorphisms coincide with the group itself. 
If we consider the case where $X^{\circ}$ is the universal cover of $B^*$, i.e. when $R=1$, every subgroup of $\mathrm{Mod}(B^*)$ preserves $[R]$. Hence, for any $A \leq \mathrm{Mod}(B^*)$, there exists an automorphism group $\Pi_A \leq \A(\pi_1(B^*,b_0))$ fitting into the exact sequence:
\begin{equation}\label{eq:septp} 
    1\rightarrow \pi_1(B^*,b_0) \rightarrow \Pi_{A} \rightarrow A \rightarrow 1,
\end{equation} 
In particular, for $A = \mathrm{Mod}(B^*)$, we obtain:
\begin{equation}\label{ses:Pi_Mod}
    1\rightarrow \pi_1(B^*,b_0) \rightarrow \Pi_{\mathrm{Mod}(B^*)} \rightarrow \mathrm{Mod}(B^*)\rightarrow 1.
\end{equation}
This sequence mirrors the arithmetic exact sequence \eqref{ses} and highlights the analogy between the role of Galois groups in the arithmetic setting and mapping class groups in the topological setting.

Observe that in the special case where $A=\{1\}$, we obtain that $\Pi_1=\pi_1(B^*)$. 
We now provide a criterion for the splitting of eq. (\ref{eq:septp}) analogue to Remark \ref{rem:seqsplit}, the proof of which adresses both the arithmetic and topological cases.

\begin{proposition}\label{prop:ses split}
The short exact sequence of eq. (\ref{eq:septp}) 
\[
1 \rightarrow \pi_1(B^*) \rightarrow \Pi_A \rightarrow A \rightarrow 1.
\]
splits if there is a point $P$ such that all $\sigma \in A$ have a homeomorphism representative $\hat{\sigma}$ that fixes $P$. 

\end{proposition}

\begin{proof}
Let $s \colon A \rightarrow \mathrm{Homeo}_+(B^*)$ be a set-theoretic section and let $P \in B^*$ be a point fixed by $A$ through $s(\sigma) = \hat{\sigma}$, i.e. for all $\sigma \in A$. Denote by $\pi\colon  \widetilde{B^*}\rightarrow B^*$ the projection from the universal covering space. Note that any point in $B^*$ is unramified, thus $\Pi_1$ acts transitively in the fiber $\pi^{-1}(b)$ of any point $b \in B^*$. 
Furthermore, via $s$ the point $P$ is $A$-invariant and $A$ acts on the fiber $\pi^{-1}(P)$. Let $Q$ be a point in the fiber $\pi^{-1}(P)$. By the covering space lifting property \cite[13.5]{MR1343250} there is a unique homeomorphism $\tilde{\sigma}:\widetilde{B^*}\rightarrow \widetilde{B^*}$ such that $\tilde{\sigma}(Q) = Q$ and $\hat\sigma \circ \pi = \pi \circ \tilde\sigma$. The uniqueness follows from the fact that fixing a point $Q$ in the fiber is equivalent to fixing a basepoint in $\widetilde{B^*}$ over the basepoint in $B^*$. In particular, each $\tilde\sigma$ is an element in the normalizer of $\mathrm{Deck}(\widetilde{B^*}/B^*)$ in the orientation-preserving homeomorphism group $\mathrm{Homeo}_+(\widetilde{B^*})$. Consider the cocycle $c(\sigma,\sigma')=\tilde\sigma \tilde{\sigma}^\prime \widetilde{\sigma \sigma}^{\prime -1}$. This is an element of $\Pi_1 \cong \mathrm{Deck}(\widetilde{B^*}/B^*)$ since it restricts to the identity map on $B^*$ through $\pi$. From its definition $c(\sigma,\sigma')$ stabilizes $Q$, however, as $P$ is unramified $c(\sigma,\sigma')$ is trivial due to the the transitivity of the action of $\Pi_1$ on the fibers. Thus, $\widetilde{\sigma \sigma}^\prime = \tilde\sigma \tilde\sigma^\prime$ and from uniqueness $\widehat{\sigma\sigma}^\prime = \hat\sigma \hat\sigma^\prime$, i.e. $s$ is a homomorphism. The result follows as the map $A\rightarrow \Pi_A \leq \mathrm{Aut}(\pi_1(B^*))$ given by $\sigma \mapsto s(\sigma)_*$ is also a homomorphism.

\end{proof}

 There is a topological interpretation of the exact sequence \ref{ses:Pi_Mod} that is due to Birman. If we mark a new point $b_0$ in $B^*$ and consider the group $\mathrm{Mod}(B^*,b_0)$, we can surject it to $\mathrm{Mod}(B^*)$ by forgetting the marked point. The kernel of this map turns out to be $\pi_1(B^*,b_0)$, by considering loops of $b_0$ as isotopies of the point and pushing them to mapping classes. 
 See \cite[Ch. 4.2.1]{FarbMagalit} for the precise details of the construction. This produces the {\em Birman exact sequence}:
 \begin{equation}\label{ses:Birman}
     1\rightarrow \pi_1(B^*,b_0) \rightarrow \mathrm{Mod}(B^*,b_0) \rightarrow \mathrm{Mod}(B^*)\rightarrow 1.
 \end{equation}
 
Let $\mathrm{Aut}^*(\pi_1(B^*,b_0))$ denote the subgroup of automorphisms that corresponds to $\mathrm{Mod}(B^*)$ via the Dehn-Nielsen-Baer theorem. We adapt to our context the construction in \cite[Pg. 235]{FarbMagalit}. Consider the following diagram:
\begin{equation*}
\begin{tikzcd}[column sep=1.4em]
1 \arrow[r] & {\pi_1(B^*,b_0)} \arrow[d, "\cong"'] \arrow[r] & {\mathrm{Mod}(B^*,b_0)} \arrow[d, dashed] \arrow[r] & \mathrm{Mod}(B^*) \arrow[d, "\cong"] \arrow[r] & 1 \\
1 \arrow[r] & {\mathrm{Inn}(\pi_1(B^*,b_0))} \arrow[r]       & {\mathrm{Aut}^*(\pi_1(B^*,b_0))} \arrow[r]          & {\mathrm{Out}^*(\pi_1(B^*,b_0))} \arrow[r]     & 1
\end{tikzcd}
\end{equation*}

If there exists a homomorphism in place of the dashed arrow that makes the diagram commute, then the short $5$-lemma implies that it is an isomorphism. This homomorphism does indeed exist, by specifying that the marked point $b_0$ is the base point of the fundamental group. The fact that we demand $b_0$ to remain fixed as the single marked point ensures that we get a genuine automorphism, rather than an outer automorphism, since conjugations of $\pi_1(B^*,b_0)$ would change the base point. 
More precisely, we have $\mathrm{Mod}(B^*,b_0) \cong \mathrm{Out}^*(\pi_1(B^*-\{b_0\}))$ and for each outer automorphism class $\phi$, there exists a unique representative that fixes the loop around $b_0$ pointwise. 
Moreover, this representative is the desired automorphism of $\pi_1(B^*,b_0)$. 
Therefore, the Birman exact sequence \ref{ses:Birman} and the exact sequence \ref{ses:Pi_Mod} are equivalent.
It is worth noting that the version of the Birman exact sequence for a surface $S_g$ with no punctures does not split, proved also in \cite{FarbMagalit}. 
However, in the punctured case, splitting may occur, although it is rare. 
For instance, in the case of the pure mapping class group of the $s$-times punctured disk, we have the exact sequence of pure braid groups
\[1\rightarrow F_{s-1} \rightarrow P_s \rightarrow P_{s-1} \rightarrow 1,\] 
which splits by embedding $P_{s-1}$ in $P_s$ by adding an extra strand.
However, we cannot replicate this argument for the punctured projective line due to the representation of $\mathrm{Mod}(S_{0,s})$.

This phenomenon mirrors the situation in the arithmetic setting: sections of the arithmetic exact sequence \eqref{ses} correspond to $K$-rational points, which often do not exist. Thus, in both the topological and arithmetic contexts, we have a significant obstruction to the splitting of the exact sequences under consideration.

\subsection{Common Notation} In the previous subsections, we explored the similarities between braids in $B_{s-1}$ and elements in $\Gal(\overline{\mathbb{Q}}/\mathbb{Q})$. 
Namely, both can be viewed as elements in $\mathrm{Aut}(F_{s-1})$ and $\mathrm{Aut}(\mathfrak{F}_{s-1})$, respectively, under a suitable lifts through the Artin and Ihara representations, respectively. The pro-$\ell$ free group $\mathfrak{F}_{s-1}$, along with its automorphism group, is an interpolation of the discrete case, i.e. $F_{s-1}~\hookrightarrow~\mathfrak{F}_{s-1}$. 
This situation is reminiscent of rigid or formal geometry, see \cite{KatoRigid}, \cite[Ch. 9]{Hartshorne:77}. 
This similarity stems from the deeper fact that both the mapping class group and $\Gal(\overline{\mathbb{Q}}/\mathbb{Q})$ act as outer automorphisms on fundamental groups and their profinite completions, respectively.

In more detail, elements of the pure braid group $P_{s-1}$ may be seen as elements in the mapping class group of the punctured disk which is homeomorphic to the projective line minus $s$-points. Elements in the full braid group act like permutations on the set of punctures $\Sigma$. 
These elements also act like as intricate homeomorphisms on the complement $D_{s-1}$ of the $s-1$ points. 
The group $\Gal(\overline{\Q}/\Q)$ acts exactly in the same way, that is, by permutations on the set of the branched $(s-1)$-points, where the point at $\infty$ is $\Gal(\bar{\Q}/\Q)$-invariant with respect to the action of the Galois group $\Gal(L/\Q)$ for some finite extension $L/\Q$. Moreover, $\Gal(\overline{\Q}/\Q)$ acts on the complement $\mathbb{P}^1_{\overline{\mathbb{Q}}}-\Sigma$ in quite ``mysterious" a way. 

In terms of arithmetic topology, that is the study of the analogies between prime numbers and knots, see \cite{Morishita2011-yw}, we provide the following similarity. 
Both knots and primes are conjugacy classes of group elements in either $B_{s-1}$ or $\Gal(\overline{\Q}/\Q)$. 
Indeed, a knot is formally an injection of $S^1$ on a $3$-manifold and a link is a disjoint union of them, which both can be cut to provide a braid by Alexander's theorem \cite[Thm. 2.3]{MR2435235}. 
By Markov's theorem \cite[Thm. 2.8]{MR2435235}, the closure of two braids (strictly) in $B_{s-1}$ induces the same knot or link if and only if the braids are conjugate. In arithmetic terms, pick a prime $p$ over $\overline{\Q}$, which corresponds to a choice $\overline{\Q} \xhookrightarrow{} \overline{\Q}_p$. 
The group $\Gal( \overline{\Q}_p/\Q_p)$ is isomorphic to the decomposition group of $D_p$ of the prime $p$, and up to conjugacy it is embedded into $\Gal(\overline{\Q}/\Q)$. If $I_p$ is the inertia group of $p$, then there is an isomorphism $D_p/I_p \cong \Gal(\overline{\mathbb{F}}_p/\mathbb{F}_p)\cong \hat{\Z}$, where the latter is topologically generated by the Frobenius automorphism of $p$. 
That is, after lifting from the quotient, we can see the prime number $p$ as a conjugacy class in $\Gal(\overline{\Q}/\Q)$. 
\begin{remark}
    The étale fundamental group of $\mathrm{Spec}(\mathbb{F}_p)$ is $\Gal(\overline{\mathbb{F}}_p/\mathbb{F}_p)\cong \hat{\Z}$, while $\pi_1(S^1) = \Z$. This further reflects the analogy between primes and knots, and highlights how profinite completions interpolate between discrete and continuous worlds.
\end{remark}

We summarize the core correspondences from the last two sections in the following table:

{\small
\begin{longtable}[c]{>{\raggedright\arraybackslash}p{6cm}  >{\raggedright\arraybackslash}p{5cm}  >{\raggedright\arraybackslash}p{4.5cm}}
\rowcolor{headerblue}
\textbf{Number Theory} & \textbf{Topology}  \\
\endfirsthead

\rowcolor{headerblue}
\textbf{Number Theory} & \textbf{Topology}  \\
\endhead

\rowcolor{rowgray}
$\Gal(\overline{\mathbb{Q}}/\mathbb{Q})$ 
  & $\mathrm{Mod}(B^*)$ 
  \\

$\Pi_K/\Pi_{K_s} \cong \Gal(K_s/K)$ 
  & $\Pi_{\mathrm{Mod}(B^*)}/\pi_1(B^*) \cong \mathrm{Mod}(B^*)$ 
   \\

\rowcolor{rowgray}
$\phi_K(\Gal(K_s/K)) \leq \mathrm{Out}(\Pi_{K_s}(B^*))$ 
  & $\mathrm{Mod}(B^*) \leq \mathrm{Out}(\pi_1(B^*))$ 
   \\

$\mathrm{Brd}/\mathrm{Brd}_1 \cong \mathbb{Z}_\ell^*$ 
  & $\mathrm{Mod}^\pm(B^*)/\mathrm{Mod}(B^*) \cong \mathbb{Z}^*$ 
   \\

\rowcolor{rowgray}
\begin{tabular}[t]{@{}l@{}}
$\mathrm{Ih} \colon \Gal(K_s/K) \to \mathrm{Brd}(\mathfrak{F}_{s-1})$ \\ 
and basepoint lift to $\A(\mathfrak{F}_{s-1})$
\end{tabular}
  & $\mathrm{Ar} \colon B_{s-1} \to \A(F_{s-1})$ 
   \\

Prime: conjugacy class in $\Gal(\overline{\mathbb{Q}}/\mathbb{Q})$ 
  & Knot: conjugacy class in $B_{s-1}$ 
   \\
\end{longtable}
}

\subsection{Example of Non Faithful Action on Elliptic Curves} \label{sec:nonfaithfulelliptic}
In this example, we study the case of elliptic curves together with the braid group action on their fundamental group. Let $E$ be an elliptic curve over $\mathbb{C}$ with short Weierstrass equation $y^2=x(x-1)(x-\lambda)$. The map $E\rightarrow \mathbb{P}^1_\mathbb{C}$, defined by $(x,y)\rightarrow x$, yields a $\mathbb{Z}/2\mathbb{Z}$-branched cover of the projective line, ramified above the four points $\{0,1,\lambda,\infty\}$. 
The topological fundamental group of $\mathbb{P}^1 \setminus \{0,1,\lambda, \infty\}$ is the free group
\[
  F_3=\langle x_1,x_2,x_3,x_4 : x_1x_2x_3x_4=1 \rangle,
\]
where $x_1, x_2, x_3$ are the loops around the points $0,1,\lambda$ and $x_4=(x_1x_2x_3)^{-1}$ is the loop around the point at $\infty$. Adjusting our notation so that it corresponds to the notation in \cite{MR4117575}, the open unramified cover $E^\circ$ corresponds to the free group $R_2$. Using the Reidemeister-Schreier method \cite{bogoGrp} a set of free generators for $R_2$ is
\[
  \{x_jx_1^{-1}, j=2,3\} \cup    
  \{x_1x_j, j=1,2,3\}, 
\]
see \cite[Lemma 13]{MR4117575}. On the one hand, the braid group $B_3$ acts faithfully on $F_3$ via the Artin representation, keeping $\infty$ invariant. On the other hand, the braid group $B_4$ acts non-faithfully, since the mapping class group of the punctured projective line is the spherical braid group given in eq. (\ref{eq:sphMod}), that is 
\[
  \mathrm{Mod}(S_{0,4})=B_4/
  \langle 
  (\sigma_1 \sigma_2 \sigma_3)^4, 
  (\sigma_1 \sigma_2 \sigma_3^2 \sigma_2 \sigma_1)
  \rangle,
\] see also \cite[Remark 3]{MR4117575}. The elements $\sigma_1, \sigma_2, \sigma_3$, being automorphisms of the group $F_3$, act in the following way:
\begin{align*}
\sigma_1(x_1)=x_1x_2 x_1^{-1},&&
\sigma_1(x_2)=x_1,&&
\sigma_1(x_3)=x_3,&& \sigma_1(x_4)=x_4,
\\ 
\sigma_2(x_1)=x_1,&&
\sigma_2(x_2)=x_2x_3x_2^{-1},&&
\sigma_2(x_3)=x_2,
&& \sigma_2(x_4)=x_4,
\\ 
\sigma_3(x_1)=x_1,&&
\sigma_3(x_2)=x_2,&&
\sigma_3(x_3)=x_3x_4x_3^{-1},&& 
\sigma_3(x_4)=x_3.
\end{align*}
Since $x_4 = x_3^{-1}x_2^{-1}x_1^{-1}$, the action on $x_4$ is determined by the first three columns. In particular, $\sigma_3(x_3) = x_2^{-1}x_1^{-1}x_3^{-1}$.

By straightforward computation, it can be checked that the group $B_4$ acts on the free group $R_2$, that is $\sigma(R_2)=R_2$ for all $\sigma\in B_4$, and that both $(\sigma_1\sigma_2\sigma_3)^4$ and $(\sigma_1 \sigma_2 \sigma_3^2 \sigma_2 \sigma_1) $ act trivilly. For instance, 
\begin{align*}
  \sigma_1(x_2x_1^{-1})&=x_1 x_1 x_2^{-1}x_1^{-1}= (x_1)^2(x_1x_2)^{-1}.
\\
\sigma_2(x_3x_1^{-1})&=x_2 x_1^{-1}.
\end{align*}
Therefore, we have a non-faithful action of $B_4$ on $\pi_1(E^\circ) \simeq R_2\leq \pi_1(\mathbb{P}^1-\{0,1,\lambda,\infty\})$. Furthermore, the action of $B_4$ preserves the normal closure $\langle x_i^2, \ 1\leq i \leq 4 \rangle$ inside $F_3$, as well as $R_2^\prime = [R_2,R_2]$. By compactifying the cusps, we have that
\[\pi_1(E(\mathbb{C})) = H_1(E(\mathbb{C}), \Z) \cong \frac{R_2}{R_2^\prime \cdot\langle x_1^2,x_2^2,x_3^2,x_4^2\rangle}, \] which can be realized as a quotient of $\pi_1(E^\circ)$, thus we have a non-faithful action of $B_4$ on $\pi_1(E(\mathbb{C}))$.

Note that this does not induce an action of $B_4$ on the elliptic curve $E$ itself via some endomorphisms. 
By \cite[pg. 20]{MR4117575}, the two elements representing $\pi_1(E(\mathbb{C}))$ as $\Z\times \Z$ are $x_2x_1$ and $x_3x_1$ in the quotient of $R_2$. The action on this basis yields a representation in $\mathrm{SL}_2(\Z)$,
\[\sigma_1,\sigma_3 \mapsto 
\begin{pmatrix}
1 & 1 \\ 
0 & 1
\end{pmatrix}, \ 
\sigma_2 \mapsto 
\begin{pmatrix}
2 & 1 \\ 
-1 & 0
\end{pmatrix}, \ 
\] 
where, for instance, $\sigma_2(x_2x_1) = x_2x_3x_2^{-1}x_1 = 2 x_2x_1 - x_3x_1 $ where the equality holds modulo $R_2^\prime\langle x_1^2,x_2^2,x_3^2,x_4^2\rangle$. In particular, $\sigma_1\sigma_3^{-1}$ acts trivially.

The obstruction is that neither of these matrices corresponds to multiplication by a complex number. Recall, by \cite[Ch. 4]{Silverman92}, that $E$ is isomorphic to $\mathbb{C}/\Lambda$, where $\Lambda$ is a lattice, say generated by $\omega_1,\omega_2$. By \cite[Prop. 5.6]{Silverman92}, $H_1(E(\mathbb{C}),\Z)$ is isomorphic to $\Lambda$ as an abelian group and by \cite[Thm. 4.1 ]{Silverman92} endomorphisms of $E$ correspond to complex numbers $\alpha$ such that $\alpha \Lambda \subseteq \Lambda$. However, the images $\overline{\sigma}_i$ in $\mathrm{SL}_2(\Z)$ do not satisfy $\overline{\sigma}_i (\omega_1,\omega_2)^T = \alpha (\omega_1,\omega_2)^T$ for any $\alpha \in \mathbb{C}$. Thus, the $B_4$ action in this setting does not arise from endomorphisms.

\begin{remark}\label{rem:action on elliptic curves}
We want to point out the relation to the action considered in Seidel and Thomas' paper \cite{MR1831820}, who showed that there is a faithful action on the bounded derived category of coherent sheaves $\mathsf{D^b}(X)$ of Calabi-Yau manifolds of dimension $\geq 2$. However, there is a non faithful action of $B_4$ on $\mathsf{D^b}(E)$ where $E$ is an elliptic curve, see \cite[Section 3d]{MR1831820}. It is interesting to observe that there is also a non faithful action of $\pi_1(E(\mathbb{C}))$, but not an action on the curve itself. The similarity between both actions is precisely that, i.e. there are non faithful action on invariants of $E$, which are not induced from an action on $E$ itself.
\end{remark}

\section{Group Actions On Curves, Varieties and their Monodromy}\label{ch:groupactions}

WWe begin this section with a lengthy introduction intended to motivate the ideas underlying the question: how should one think of a group action on a (locally ringed) topological space? More precisely, what is the “correct" group action one should study?

So, let us begin with a (locally) ringed space $(X, \mathcal{O}_X)$ and a second topological space denoted $^h\!X$, together with a homeomorphism $h \colon X \rightarrow {^h \! X}$ between them. The notation will become clearer in the following chapters. For now we may think of $^h\!X$ as the topological space obtained by the action of a group element $h$ on $X$. By abuse of notation, we denote the resulting homeomorphism also by $h$.

Our goal is to define a (locally) ringed space structure on $^h \! X$.
To that end, we define the structure sheaf $\mathcal{O}_{^h \! X}$ of ${^h\!X}=h(X)$ as the induced structure sheaf, by pulling back sections along $h$. Specifically, for any open set $U \subseteq {^h\!X}$, we set: 
\[
  \O_{{^h \! X}} (U) = \O_X (h^{-1}U).  
\]
This defines a sheaf of rings on ${^h \! X}$ rendering $(^h \! X, \O_{^h \! X})$ a (locally) ringed space. We denote the corresponding morphism of (locally) ringed spaces by:
\[
\tilde{h} \colon (X, \O_X) 
\to (^h \! X, \O_{^h \! X}).  
\]

Recall that for a sheaf $\mathcal{F}$, we denote by  
$\tilde{h}^{-1}\mathcal{F}$ its inverse image under $\tilde{h}$. By definition this means:
\[
  \tilde{h}^{-1} \mathcal{F}(U)= \mathcal{F}(h(U)) \text{, for all open sets } U\subseteq X. 
\]
Note that there is no need to consider a limit over open sets covering $h(U)$ as in \cite[p.65]{Hartshorne:77}, since $h$ is a homeomorphism and $h(U)$ is a priori open. Observe also that $\tilde{h}^{-1} \O_{^h\!X} = \O_X$. 
Therefore, the map $\tilde{h}\colon X \to {^h \!X}$ is a surjective open immersion in the sense of locally ringed spaces, 
i.e. $h$ is a homeomorphism and $h^{-1} \O_{{^h \!X}} \to \O_X$ is an isomorphism.  
It follows that, in this setting, the pullback functor $\tilde h^*$ coincides with $\tilde h^{-1}$, since the inverse image already produces sheaves of $\O_X$-modules. 

Now suppose $X$ has the structure of a complex manifold and $(U,z_U)$ is a system of holomorphic charts on $X$, with $z_U\colon U \rightarrow \mathbb{C}^n$,  then ${^h\!X}$ also inherits the structure of a complex manifold with charts $(h(U),  z_{U} \circ h^{-1})$,

\[
\begin{tikzcd}
{^h\!X} \arrow[r, "h^{-1}"] & X \\
h(U) \arrow[r, "h^{-1}"] \arrow[u, hook] & U \arrow[r, "z_U"] \arrow[u, hook] & \mathbb{C}^n
\end{tikzcd}
\]

Note that all transition functions in ${^h  \! X}$ are holomorphic, even though $h$ itself need not be.
Indeed, if $h(U_a)\cap h(U_b)\neq \varnothing$, the transition map on the intersection is given by $z_{U_a}\circ h^{-1} \circ (z_{U_b}\circ h^{-1})^{-1} = z_{U_a} \circ z^{-1}_{U_b}$ which is holomorphic since it is a transition map from the original manifold $X$.

In particular, for algebraic curves over $\mathbb{C}$, equivalently for compact Riemann surfaces, this construction respects the local analytic structure. If $z_P$ is a local uniformizer at $P\in X$, then $h(z_P)$ is a local uniformizer at $h(P) \in {^h \! X}$.
This means that the elements $f=\sum_{\nu =0}^\infty a_\nu z_P^\nu\in \widehat{\O}_P$ are mapped to $h(f)=\sum_{\nu =0}^\infty a_\nu h(z_P)^\nu\in \widehat{\O}_{hP}$.

\begin{remark}
The basic idea we wish to capture is that one can ``twist" the space $X$ by acting on it via isomorphisms or homeomorphisms, and then consider the induced structure on the resulting space. However, this alone is not particularly interesting: such a twist yields a space that is trivially isomorphic to the original as a (locally) ringed space.

What we will instead do in the sequel is twist both the topological space $X$ and its structure sheaf. In this case, the resulting (locally) ringed space is no longer a priori isomorphic to the original one. This opens the door to studying nontrivial symmetries or deformations that genuinely affect the structure, rather than simply relabeling it.

\end{remark}

\subsection{Arithmetic Actions}\label{subsec:aractions}

We begin by considering the action on classical algebraic sets (affine case), and  then extend to the more abstract scheme world. 

Recall that $K$ is a field of characteristic $p \geq 0$ and  $K_s$ is a fixed separable closure of $K$. Any $L$-scheme, for $K \subseteq L \subseteq K_s$ can be viewed as a $K_s$-scheme by extension of scalars. Therefore, throughout this chapter we work over $K_s$. Consider the absolute Galois group $\Gal(K_s/K)$.  Each $\sigma \in \Gal(K_s/K)$ defines a bijection 
\begin{align*}
  \widehat{\sigma} \colon K_s^n &\longrightarrow K_s^n \\
    (y_1,\ldots,y_n) & \longmapsto \big( \sigma(y_1),\ldots, \sigma(y_n) \big).
\end{align*}
Moreover, for each polynomial $P\in K_s[z_1,\ldots,z_n]$, we denote by ${^\sigma \!}{P}$ the polynomial obtained by applying $\sigma$ to the coefficients of $P$. 
For every $\sigma \in \Gal(K_s/K)$, we denote by $\hat{\sigma }$ the map $K_s[x_1, \ldots , x_n] \rightarrow K_s[x_1, \ldots , x_n]$ sending the polynomial $P$ to $\hat{\sigma }(P) \coloneqq {^\sigma \! P}$.
The following diagram commutes:
\begin{equation}
\label{eq:commute-eta}
  \xymatrix{
  K_s^n  \ar[r]^{P} \ar[d]_{\widehat{\sigma}} & K_s \ar[d]^{\sigma}\\
  K_s^n \ar[r]^{^\sigma\!P} & K_s
  }
\end{equation}
where $P$ (resp. $^{\sigma}\!P$) denotes the polynomial function induced by $P$ (resp. $^{\sigma}\!P$).
If $X$ is a variety defined as the zero locus of a set of polynomials $P_1,\ldots,P_r$, that is 
\[
  X=\{(y_1,\ldots,y_n) \in K_s^n: P_j(y_1,\ldots,y_n)=0, 1 \leq j \leq r\}, 
\]
then for $\sigma \in \Gal (K_s / K)$ we define the variety
\[
  {^\sigma \! X} =
  \{(y_1,\ldots,y_n) \in K_s^n: {^\sigma \!P}_j(y_1,\ldots,y_n)=0, 1 \leq j \leq r\}.
\]
Using eq. (\ref{eq:commute-eta}) we see that $P \in X $ if and only if $\widehat{\sigma}(P) \in {^\sigma \!X}$.

Now, let $p\colon X \to \Spe(K_s)$ be a $K_s$-scheme. The group $\Gal(K_s/K)$ acts on $X$ as follows.  
The ring homomorphism $\sigma \colon K_s \to K_s$ corresponds (contravariantly) to the map $\tilde{\sigma}\colon \Spe(K_s) \rightarrow \Spe(K_s)$,
i.e. we have that $\widetilde{\sigma \tau}= \tilde{\tau} \tilde{\sigma}$, for $\tau \in \Gal(K_s/K)$. 
The scheme $ {^\sigma \! X}$ is defined as the fiber product
 \({^\sigma \! X}=X \times_{\tilde{\sigma}} \Spe(K_s) \), that is, the unique (up to isomorphism) scheme making the following diagram commute:
\begin{equation}
  \label{eq:schemeActi}
  \xymatrix{
   X  \ar[d]_{p}  &  ^\sigma\!{X} \ar[d]^{\pi_2} \ar[l]_{\pi_1} & \\
   \Spe(K_s)  & \Spe(K_s) \ar[l]^{\tilde{\sigma}} & 
  }
\end{equation}
\begin{remark}\label{rem:non isomorphic K_s-schemes}
At this point, we emphasize that although $X \simeq {^\sigma\!X}$ as schemes, they are not necessarily isomorphic as $K_s$-schemes when the structure morphism of ${^\sigma X}$ is taken to be the twisted one, denoted $\pi_2$.
More precisely, the natural projection $\pi_1 \colon X \to \! ^{\sigma}\!X $ is an isomorphism of schemes, since isomorphisms are stable under base change. 
However, this isomorphism does not, in general, respect the structure of $K_s$-schemes, as we explain below. 
This phenomenon occurs naturally as there are strictly more morphisms of schemes than there are morphisms of $K_s$-schemes.

Recall that the category of $L$-schemes consists of objects that are maps $X~\rightarrow~\Spe L$ and morphisms that are maps $X_1 \to X_2 $ such that the following triangle commutes:
\[
  \xymatrix{
    X_1 
    \ar[dr] 
    \ar[rr] 
    & & 
    X_2 
    \ar[ld] \\
    & \Spe L  &
  }
\]  
\end{remark}
The scheme ${^\sigma \! X} \rightarrow \Spe(K_s)$, admits two natural choices for its structure map, namely \(\tilde{\sigma} \circ \pi_2 \)  and \(\pi _2\),
giving rise to two different elements in the category of $K_s$-schemes, $^\sigma \!X \stackrel{ \tilde{\sigma }\pi_2}{\longrightarrow } \Spe(K_s )$ and $^\sigma \!X \stackrel{ \pi_2}{\longrightarrow } \Spe(K_s )$, respectively. 
\begin{itemize}
    \item Using \(\tilde{\sigma} \circ \pi_2 \) we obtain a trivial isomorphism ${^\sigma \! X} \simeq X$ of $K_s $-schemes.
    \item Using \( \pi_2 \), the projection map $\pi_1$ is still an isomorphism ${ ^\sigma \! X} \rightarrow X$ of schemes but not an isomorphism of $K_s$-schemes.
\end{itemize}
This discrepancy arises because of the twisted $K_s$-algebra structure on $^\sigma \! X$.
To understand this phenomenon, we compare with the affine case described earlier. The diagram from equation \eqref{eq:schemeActi} (with reversed arrows contravariantly) illustrates how the Galois action affects coordinate rings:
  \[
    \xymatrix{
     X  
     \stackrel{\text{corresp.}}\longleftrightarrow K_s[z_1,\ldots,z_n]/\langle P_1,\ldots,P_r \rangle
      \ar[r]^{\hat{\sigma}} &  
     K_s[z_1,\ldots,z_n]/\langle {^\sigma \! P_1},\ldots, {^\sigma \! P}_r \rangle 
     \stackrel{\text{corresp.}}\longleftrightarrow  {^{\sigma}\!X}
       & \\
     K_s  \ar[u] \ar[r]^{\sigma} & K_s \ar[u],
    }
  \]
  where the $K_s$-algebra $K_s[z_1,\ldots,z_n]/\langle {^\sigma \! P_1},\ldots, {^\sigma \! P}_r \rangle$ is equipped with the $K_s$-algebra structure using the composition $\tilde{\sigma}\circ \pi_2$, that is by the following rule:
    \[
     \lambda \odot \hat{\sigma}(\omega) = \sigma(\lambda)\hat{\sigma}(\omega) = \hat{\sigma}(\lambda \omega) 
    \]
    for $\lambda \in K_s$ and $\omega \in K_s[z_1,\ldots,z_n]/\langle P_1,\ldots,P_r \rangle$.

This means that $^{\sigma}\!X$ and $X$ might not be isomorphic as $K_s$-schemes. And if an isomorphism of $K_s$-schemes does exist, it is non-canonical, i.e. the existence of a morphism $f_{\sigma} \colon X \to \! ^{\sigma}\!X$ such that the following diagram is commutative
\begin{equation*}
    \begin{tikzcd}
        X \arrow[rr, dashrightarrow, "f_{\sigma}"]  \arrow[dr, "p"'] &&  ^{\sigma}\!X \arrow[dl, "\pi_2"] \\
         & \Spe(K_s)
    \end{tikzcd}
\end{equation*}    
is not guaranteed.

\begin{definition}{\label{def:schemeactedupon}}
Throughout this article, we use the object \( {^\sigma\!X} \to \Spe(K_s) \), where the structure morphism is taken to be \( \pi_2 \), and we denote it simply by \( {^\sigma\!X} \).

The projection map \( \pi_1 \colon {^\sigma\!X} \to X \), which arises from the base change along the field automorphism \( \sigma \), is denoted (by slight abuse of notation) by \( \tilde{\sigma} \).
\end{definition}

The Galois action described above also extends naturally to morphisms of $K_s$-schemes. Namely, a $K_s$-morphism $f\colon X \to Y$ is mapped to the (uniquely defined) $K_s$-morphism ${^\sigma \! f \colon  {^\sigma \! X} \to {^\sigma Y}}$, which is induced by the universal property of the fiber product.

Moreover, for $\sigma, \tau \in \Gal(K_s/K)$ we have the following commutative diagram of schemes:
\begin{equation*}
    \begin{tikzcd}
        X  \arrow[d] & ^\sigma X \arrow[d] \arrow[l, "\widetilde{\sigma}"'] & ^\tau( ^\sigma \! X) \arrow[d]  \arrow[l, "\widetilde{\tau}"']\\
        \Spe(K_s) & \Spe(K_s) \arrow[l, "\widetilde{\sigma}"']& \Spe(K_s) \arrow[l, "\widetilde{\tau}"'] \arrow[ll, bend left =20, "\widetilde{\sigma\tau}"]
    \end{tikzcd}
\end{equation*}
Since $\widetilde{\sigma \tau} = \widetilde{\tau}\widetilde{\sigma}$, we obtain the following identity of fiber products: 
\begin{equation}\label{eq:strictness of action}
^\tau( ^\sigma \! X) = (X \times_{\tilde{\sigma}} \Spe(K_s) )\times_{\widetilde{\tau}} \Spe(K_s) = X \times_{\widetilde{\sigma\tau}}\Spe(K_s) = { ^{\sigma \tau}\! X} 
\end{equation}

The following proposition tells us that extending by scalars an $L$-variety to a $K_s$-variety is ``harmless".

\begin{proposition} \label{prop:invArith}
Let \( X \) be an \( L \)-variety. Then \( X \) is defined over a subfield \( L_0 \subseteq L \) if and only if
\[
{^\sigma\!X} = X \  \text{ for all } \sigma \in \Gal(L/L_0).
\]
In this case, \( X \) descends to an \( L_0 \)-variety.
\end{proposition}
\begin{proof}
Suppose \( X_0 \) is a variety defined over \( L_0 \), and let \( X = X_0 \times_{L_0} L \) be its base change to \( L \). For any \( \sigma \in \Gal(L/L_0) \), we compute:
\[
{^\sigma\!X} = \left(X_0 \times_{L_0} L\right) \times_{\sigma} L = X_0 \times_{L_0} (L \otimes_{\sigma} L).
\]
But \( L \otimes_{\sigma} L = L \), since the left field  $L$ is considered as an $L$-module using the  identity map while the right field $L$  is an $L$-module using the $\sigma \colon L\rightarrow L$ map.

Conversely, if \( {^\sigma\!X} = X \) for all \( \sigma \in \Gal(L/L_0) \), then the descent data is trivial, and \( X \) descends to an \( L_0 \)-variety by standard faithfully flat descent (see \cite[SGA1]{SGA1}).
\end{proof}

\begin{remark}\label{rem:invariantactionofsubgroup}
Hence, if we start with an $L$-scheme $X$ and extend it by scalars to a $K_s$-scheme, then $\Gal(K_s/L)$ leaves $X$ invariant, i.e. $^\sigma \! X = X$ as $K_s$-schemes. In the language of algebraic varieties, this simply reflects the fact that $X$, being defined as the zero locus of polynomials over $L$, does not have this set of polynomials affected by the Galois action on their coefficients. Hence, a large finite index subgroup of $\Gal(K_s/L)$ preserves $X$.

The key point is that the part of the action of the full Galois group $\Gal(K_s/K)$ not preserving $X$ factors through the finite quotient $\Gal(K_s /K) / \Gal(K_s/L) \simeq \Gal(L/K)$.

Concisely, this means that we could have defined a $\Gal(L/K)$-action on an $L$-scheme $X$, following the same idea described above for $K_s$-schemes. An equivalent reformulation is that the action of $\Gal(L/K)$ which is obtained by restricting the action of $\Gal(K_s/K)$, is well defined.
Later on, we will use this finite group $\Gal(L/K)$ to study these actions since it captures the essential information of the action, for example see Section \ref{sec:FieldMODDefinition} or Proposition \ref{prop:equivariantweildescent}.
\end{remark}

The following remark slightly generalizes the above.
\begin{remark}\label{rem:generalizedgroupaction}
    Note that, for some field extension $K \subseteq L_0 \subseteq K_s$, if we have a subgroup $\Gal (K_s / L_0) \leq \Gal (K_s/K)$, then this subgroup also induces an action on $X$.
    Note also that as in the previous Remark \ref{rem:invariantactionofsubgroup}, only a finite subgroup of the form $\Gal(L / L_0) $ acts without preserving $X$.
\end{remark}

We end this section with a final remark, that explains how the action moves the $K_s$-points of the scheme, to cultivate the intuition needed for the topological setting. 

\begin{remark}{\label{rem:schemeaction}}
In the language of schemes a geometric point $P$ corresponds to a map $P \colon \Spe(K_s) \rightarrow X$. Thus we have the following commutative diagram
\begin{equation*} \label{eq:schematicpt }
\begin{tikzcd}
	&   \Spe(K_s) \arrow[ddl, "\tilde{\sigma}^{-1} \times P" description] \arrow[ddr, "P" description] \arrow[dddl, bend right = 70, "\tilde{\sigma}^{-1}"'] \arrow[dddr, bend left = 70, "\mathrm{Id}"] \\
	\\
	{^\sigma \! X} \arrow[d, "\pi_2"] \arrow[rr, "\tilde{\sigma}"] & & X \arrow[d, "p"'] \\
	\Spe(K_s) \arrow[rr, "\tilde{\sigma}"] & & \Spe(K_s)
\end{tikzcd}
\end{equation*}
where the map $\tilde{\sigma}^{-1} \times P \colon \Spe(K_s) \rightarrow  {^\sigma \! X }$ is the one induced by the universal property of the fiber product, applied on the maps $P\colon \Spe(K_s)\rightarrow X$ and $\tilde{\sigma}^{-1} \colon \Spe(K_s)\rightarrow \Spe(K_s)$.
\end{remark}

\subsection{Arithmetic Mondoromy}
Recall from Section \ref{subsec:arcovers} that by a mere cover $X\! \rightarrow\!B$ over $K$, we mean a morphism of $K$-schemes, which we can see over any field $K\subseteq L\subseteq K_s$ via extension of scalars. Also recall that if $X \to B$ is an $L$-cover, $B$ can be a $K$-variety that we extend by scalars to an $L$-variety.

\begin{definition}
Two mere covers $f \colon X \to B$ and $f \colon X^\prime \to B$ over $K$ are said to be isomorphic if there exists an isomorphism $\phi \colon X \to X^\prime$ making the following diagram commutative:
\[
  \xymatrix{
  X \ar[rr]^\phi \ar[dr]_f& & X' \ar[dl]^{f'}
  \\
  & B & 
  }
\]
An isomorphism of $G$-covers is an isomorphism of mere covers that is also compatible with the $G$-action.     
\end{definition}

The action of $\Gal(K_s / K)$ described in the previous section naturally extends to an action on covers $f \colon X \to B$ over $K_s$. That is, an element $\sigma \in \Gal(K_s/K) $ maps a cover $f \colon X \to B$ to $^\sigma \! f \colon  {^\sigma \! X} \to { ^\sigma\! B}$. Hence, there is no reason for the cover ${^\sigma \! f}$ to be isomorphic to $f$ since the base $B$ has changed. To remedy this, we require the base to remain fixed under the action of $\sigma$ i.e. equivalently $^\sigma \! B = B$, or equivalently $B$ is $\Gal(K_s/K)$-invariant. This condition is automatically satisfied when $B$ is defined over $K$, in which case, by Proposition \ref{prop:invArith}, we have that $B$ is $\Gal(K_s / K)$-invariant.

We now begin developing the dictionary between mere covers and monodromy representations arising from the arithmetic fundamental group $\Pi_K(B^*)$. 
Specifically, every finite cover $X\to B$ of degree $d$ corresponds to a representation $\Psi \colon \Pi_K(B^*) \to S_d$, where $S_d$ denotes the symmetric group in $d$ elements, see Proposition \ref{prop:merecovertrasnArith}. Upon restricted to the geometric fundamental group $\Pi_{K_s}$, these representations become profinite analogues of the classical monodromy representations of $\pi_1(B^*)$ in topology.
Once this dictionary is in place, we will examine the actions on the monodromy representations and how they relate to the original actions on the covers. 
As is well-known, the natural action of $\sigma \in \Gal(K_s / K)$ on the cover $X \rightarrow B$, maps it to $^\sigma \! X \rightarrow B$ with corresponding representation $\Psi^\prime$. 
If we denote by ${^\sigma \Psi}$ the action of $\sigma \in \Gal(K_s /K)$ on the monodromy representation, then these representations are conjugate in $S_d$, i.e. ${^\sigma \Psi} = \sigma \Psi' \sigma^{-1}$ for some $h \in S_d$. 
More precicely, ${^\sigma \Psi}$ is the representation of ${^\sigma X } \to B$ up to conjugation. Another way to look at this is that the cover that corresponds to ${^\sigma \Psi}$ is isomorphic to ${^\sigma X \to B}$, but not necessarily equal. This is made precise in Proposition \ref{prop:gal_action_monodromy} where we assume that both $B$ and the ramification divisor $D$ are fixed by action. 
The next proposition is the first step towards establishing the dictionary.

First, recall from Section \ref{sec:covers} that a degree $d$-mere cover corresponds to a finite regular subextension $E/K(B)$ of degree $d$ in $\Omega_D/K(B)$. Moreover, the Galois closure $\hat{E}/K(B)$ of $E$ corresponds to a quotient of $\Pi_K(B^*)$.
\begin{proposition}
  \label{prop:merecovertrasnArith}
Degree $d$ mere covers of $B$ over $K$ with ramification locus in $D$ correspond to transitive representations
\[
  \Psi \colon \Pi_K(B^*) \longrightarrow S_d,
\]
such that the restriction to $\Pi_{K_s}(B^*)$ is also transitive. 
\end{proposition}

A proof is given in \cite{DebesDouai97}, which we repeat in order to introduce the relation of the groups $\Gal(\Omega_D/\hat{E})$ and $\Gal(\Omega_D/E)$.

\begin{proof}
 The Galois closure $\hat{E}$ corresponds to a normal subgroup $\Gal(\Omega_D/\hat{E})$, or equivalently to a surjective homomorphism 
\[
\Phi \colon \Pi_K(B^*) \rightarrow G=\Gal(\hat{E}/K(B)).
\] 
We have the following tower of fields below in Figure \ref{fig:fieldcurveTower}.
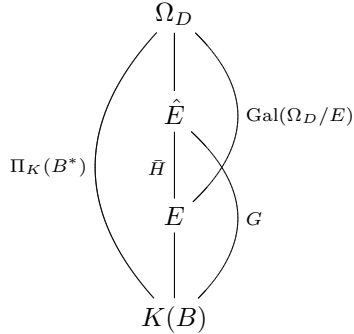
\begin{figure}[ht!]
    \centering
    \[
  \xymatrix{
  \Omega_D  \ar@{-}[d]  \ar@{-}@/_2.5pc/[ddd]_{\Pi_K(B^*)}
  \ar@{-}[d] \ar@{-}@/^2pc/[dd]^{\Gal(\Omega_D/E)} 
  \\
  \hat{E} \ar@{-}[d]_{\bar{H}} \ar@{-}@/^2pc/[dd]^{G} 
  \\
  E \ar@{-}[d] 
  \\
  K(B)  
  }
\]
\caption{Field tower.}
\label{fig:fieldcurveTower}
\end{figure}

It is known that
\begin{equation}
\label{eq:NCgroup}
\Gal(\Omega_D/\hat{E})= \bigcap_{g \in \Pi_K(B^*)} g \Gal(\Omega_D/E) g^{-1},
\end{equation}
see  \cite[Ch. I, Exer. 8, P. 60]{MR1410264}.
Let $\bar{H}=\Gal(\hat{E}/E)$,
there is a faithful action of $G$ on the cosets of $\bar{H}$ in $G$, that is on the set $\bar{H}, g_2\bar{H}, \ldots, g_d \bar{H}$. Indeed, if some $g \in G$ acts trivially on the cosets, then 
\[
  g g_i \bar{H} = g_i \bar{H} \text{, for all } g_i \in \{ g_1\!=\!e, g_2,\ldots,g_d \}.
\]
Thus, we have that
\[
  g \in \bigcap_{\nu=1}^d g_\nu \bar{H} g_\nu^{-1}=\Gal(\hat{E}/\hat{E})=\{1\},
\]
that is, the action is faithful. Therefore, we have a representation $i\colon  G \rightarrow S_d$ and $\Psi= i \circ \Phi$ is the desired transitive representation $\Pi_K(B^*) \rightarrow S_d$. 

Now, we go on to show that the restriction to $\Pi_{K_s} \coloneqq \Pi_{K_s}(B^*)$ is also transitive. Denote by $H$ the group $\Gal(\Omega_D/ E)$ such that
  \[
  \bar{H}=\frac{\Gal(\Omega_D/E)}{\Gal(\Omega_D/\widehat{E})}=
  \frac{H}{\bigcap_{g \in \Pi_K} g H g^{-1}}.
  \]
The restriction of $\Phi$ to $\Pi_{K_s}$ is the quotient map 
\[
\Phi | _{\Pi_{K_s}} \colon \Pi_{K_s} 
\rightarrow 
\frac {\Pi_{K_s} }{\Gal(\Omega_D/K_s\cdot \hat{E})},
\]
since $\Pi_{K_s}\cap \Gal(\Omega_D/\hat{E}) = \Gal(\Omega_D/K_s\cdot \hat{E})$. The mere cover over $K_s$ by extension of scalars corresponds to $H\cap \Pi_{K_s} = \Gal(\Omega_D/K_s\cdot E)$ which we denote by $R$. The regularity assumptions on the mere cover and its Galois closure provide that $G \cong \Gal(K_s\cdot \hat{E}/K_s(B))$ and 
\[\bar{H} \cong  \Gal(K_s\cdot \hat{E}/ K_s \cdot E) = \frac{\Gal(\Omega_D/K_s \cdot E)}{\Gal(\Omega_D/K_s \cdot \widehat{E})}=
  \frac{R}{\bigcap_{g \in \Pi_{K_s}} g R g^{-1}},\] where the isomorphism maps are the restrictions from $K_s\cdot \hat{E}$ to $\hat{E}$.

This proves that the action of $\Pi_{K_S}$ on the cosets of $\bar{H}$ in $G$ is also transitive, by repeating the previous argument used for $\Pi_K$.

Conversely, given a map $\Psi\colon \Pi_K(B^*) \rightarrow S_d$ we can consider the stabilizer of $1$, that is the group 
\begin{equation}
\label{eq:defP}
\Pi_K(1)\coloneqq \{ g\in \Pi_K(B^*): \Psi(g)(1)=1\}
\end{equation}
and then consider the fixed field extension $\Omega_D^{\Pi_K(1)}/K(B)$. This corresponds to a proper mere cover, i.e. it is regular by transitivity of the restriction $\Psi \!\! \mid_{\Pi_{K_s}}$.
\end{proof}

Note that the map $\Psi$ is defined up to a choice of labeling for the cosets of $G/\bar{H}$. However, different labelings  correspond to conjugation by elements of the the normalizer of $G$ in $S_d$. In particular, while $\Psi(x)$ is defined up to conjugation in $S_d$, the image $\Psi(\Pi_K) = G$ is well-defined.

Once the regularity condition has been reformulated in purely group-theoretic terms in Section \ref{sec:regularity}, we will present an alternative proof of the transitivity of the restriction  $\Psi|_{\Pi_{K_s}}$. 
The construction above motivates the following lemma, which will also apply in the topological setting, as it is phrased entirely in group-theoretic terms.
Recall that the normalizer $\mathrm{N}_{S_d}G$ of $G$ in $S_d$ is the group 
\[
  \mathrm{N}_{S_d}G=\{s \in S_d: sGs^{-1}=G\},
\]
and the centralizer $\mathrm{Cen}_{S_d}G$ is the group
\[
\mathrm{Cen}_{S_d}G=\{s \in S_d: sgs^{-1}=g \text{ for all } g\in G\}.
\]
\begin{lemma}
  \label{lemma:conjNorm}
Let $G$ be a group and $\bar{H}$ be a subgroup of $G$ of index $d$. Consider the natural left action of $G$ on cosets $\bar{H}, g_2 \bar{H}, \ldots, g_d \bar H$, which induces a representation
\[
  \Psi_{\bar H} \colon G \longrightarrow S_d. 
\]
For a conjugate subgroup $n\bar{H}n^{-1}$ of $\bar{H}$, the corresponding representation
\[
  \Psi_{n \bar H n^{-1}} \colon G \longrightarrow S_d
\]
is a conjugate of $\Psi_{\bar{H}}$ by an element $\varphi \in \mathrm{N}_{S_d}G$.
\end{lemma}
\begin{proof} By definition, for $g\in G$ we have 
\[
  g g_i =g_{\sigma(g)(i)} h, \ \text{for some } h \in \bar{H}
\]
and $\sigma(g)\in S_d$ is the permutation corresponding to $g$ under the injection $G\xhookrightarrow{}S_d$. 

For $n\bar{H}n^{-1}$, we can choose as coset representatives the elements
$n g_i n^{-1}$, for $i=1,\ldots,d$
 and the corresponding action of $ngn^{-1}$ is given by
\[
  n g  n^{-1} ng_in^{-1}  = n g_{\sigma(g)(i)} n^{-1} nhn^{-1}.
\]
This means that the element $ngn^{-1}$ acts on the cosets of $n\bar{H}n^{-1}$ in the same way that $ g $ acts on the cosets of $\bar{H}$ in $G$. This proves that $\Psi_{\bar H} (g) = \Psi_{n\bar H n^{-1}} (ngn^{-1})$, (up to conjugation by an element in $\mathrm{N}_{S_d}(G)$ for the specific choice of labeling of the cosets of $n\bar{H}n^{-1})$. Replace $g$ by $n^{-1}gn =\colon g^n $, so that $\Psi_{\bar H} (g^n) = \Psi_{n\bar H n^{-1}} (g)$.
The conjugation action by $n$ introduces an element $\tau\in S_d$, by writing
\[
  g_i^n = n^{-1}g_i n = g_{\tau(i)} h, \text{ for some element } h \in \bar{H}. 
\]
In the following diagram, we keep track of the coset representatives of $\bar{H}$ in the inner commutative diagram, whilst keeping track of the permutations in $S_d$ that arise after each operation on the representatives.
{\tiny
\[
\xymatrix{
  (1,\ldots,d) \ar[rrr]^{\tau}
    \ar[ddd]_{\sigma(g)} \ar@{~}[dr]  & & & (\tau(1),\ldots,\tau(d)) \ar[ddd]_{  \sigma(g^n) }
    \\
   &     \{g_1,\ldots,g_d\} 
  \ar[d]^{g_i \rightarrow g\cdot g_i}
  \ar[r]^{g_i\rightarrow g_i^n} & \{g_1^n,\ldots,g_d^n\} 
 \ar[d]^{g_i \rightarrow g^n\cdot g_i}
  \ar@{~}[ur] & 
  \\
    &     \{gg_1,\ldots,gg_d\} \ar[r]_{g_i\rightarrow g_i^n} & \{g^ng_1^n,\ldots,g^ng_d^n\} \ar@{~}[dr]
\\
(\sigma(g)(1),\ldots,\sigma(g)(d)) 
   \ar@{~}[ur]
   \ar[rrr]_\tau
& & & 
\substack{
  \{g_{\tau\sigma(g)(1)},\ldots,g_{\tau\sigma(g)(d)}\} \\
  = \\
   \{g_{\sigma(g^n)\tau(1)},\ldots,g_{\sigma(g^n)\tau(d)}\} 
}
}
\]
}

The commutativity of the above diagram, is interpreted by the relation
\[
  \sigma(g)\tau=\tau   \sigma(g^n),
\]
that is the element $n$ corresponds to  the element $\tau \in \mathrm{N}_{S_d}G$ as required.  
\end{proof}

\begin{proposition}
  \label{prop:isomCovers}
Two mere covers $X \rightarrow B$, $X' \rightarrow B$ over $K$ are isomorphic, as $K$-covers if and only if the 
corresponding representations $\Psi,\Psi'\colon  \Pi_K(B^*) \rightarrow S_d$ are conjugate by an element $\varphi$ of the normalizer $\mathrm{N}_{S_d}(G)$ of the image group $G=\Psi(\Pi_{K_s}) = \Psi'(\Pi_{K_s})$, that is
\[
  \Psi'(x) = \varphi \Psi(x) \varphi^{-1} \text{ for all } x \in \Pi_K(B^*). 
\] 
\end{proposition}
\begin{proof}
Indeed, if the covers are isomorphic, then their function fields $K(X), K(X')$ correspond to conjugate subgroups
of  $\Gal(\Omega_D/K(B))= \Pi_K(B^*)$, 
that is there is an element $n\in \Pi_K(B^*)$
such that 
$\Gal(\Omega_D/K(X'))= n \Gal(\Omega_D/K(X))n^{-1}$.
Let $\widehat{K(X)}$ and $\widehat{K(X')}$ be the Galois closures of $K(X)$ and $K(X')$ in $\Omega_D$, respectively. 
By eq. (\ref{eq:NCgroup}) it is clear that $\widehat{K(X)}= \widehat{K(X')}$. Let $\bar H = \Gal(\widehat{K(X)}/ K(X))$,  then
we have the following diagram.
\[
  \xymatrix{
  &\Omega_D  \ar@{-}[d]  
   \ar@{-}@/_7.5pc/[ddd]_{\Pi_K(B^*)}
  \\
& \widehat{K(X)}   \ar@{-}[dl]^{\bar{H}}  \ar@{-}[dr]_{n\bar{H}n^{-1}} 
\ar@{-}@/^7pc/[dd]^{G}
  \\
  K(X) \ar@{-}[rd] \ar[rr]_n & & K(X')  \ar@{-}[ld]
  \\
  & K(B)
  }
\]
The existence of $\phi$ follows by Lemma \ref{lemma:conjNorm}.

Conversely, consider the subgroups $H,H^\prime$ of $\Pi_K$, which are the stabilizers of $1$ for $\Psi, \Psi^\prime$, respectively, as in eq. (\ref{eq:defP}). These subgroups correspond to the covers $X\to B, \ X^\prime \rightarrow B$, respectively. The element $\phi \in \mathrm{N}_{S_d}(G)$ determines a labeling of the cosets of $H$ and $H^\prime$ in $\Pi_K$, since relabeling the fiber corresponds to conjugating $\phi$ by another element of $\mathrm{N}_{S_d}(G)$. 
Set $j=\phi^{-1}(1)$.
Under this choice, $H^\prime$ coincides with the stabilizer of $j$ under the image of $\Psi$. Pick any $n \in \Pi_K$ such that $\Psi(n)(j) = 1$, then $nH^\prime n^{-1} = H$, which implies that the covers are isomorphic.
\end{proof}
We aim to define an action of the Galois group $\Gal(K_s/K)$ on the function fields $K(X)$ and $K_s(X)$. Let us denote by  $R=\Gal\big(\Omega_D/K_s(X)\big)$ and $H = \Gal\big(\Omega_D/K(X)\big)$ and let $K\subseteq L \subseteq K_s$ be a field extension. We obtain the following diagram of function fields:

\begin{equation}
\label{eq:diagFF}
\begin{tikzcd}[row sep=large, column sep=large]
\Omega_D \arrow[ddrr, no head, "\Pi_{K_s}" description, crossing over] \arrow[drr, no head, "R"]  \arrow[dd, no head, "H"]
\arrow[dddd, no head, bend right, "\Pi_{K}" description] \\ 
& &  K_s(X) \arrow[d,no head] \arrow[dll,no head] 
\\
K(X) \arrow[dd,no head] &  & K_s(B) \arrow[ddll,no head,bend left, "\Gal(K_s/K)" description] \arrow[dd,no head] \\
& L(B) \arrow[ul,no head] \arrow[ur, no head] \arrow[dl,no head] &\\
K(B) \arrow[d,no head]  &  &  K_s \arrow[dll,no head,"\Gal(K_s/K)"]\\
K & & 
\end{tikzcd}
\end{equation}

Consider the short exact sequence given in eq. (\ref{ses}). 
Elements  $h\in \Gal(K_s/K)$ lift  to elements $\tilde{h} \in \Pi_{K}=\Pi_{K}(B^*)$, but if the short exact sequence does not split, this lift is not canonical and may vary by an element $\tau\in \Pi_{K_s}=\Pi_{K_s}(B^*)$. 

Therefore, there is no canonical way to define an action of $\Gal(K_s/K)$ on $R$ by conjugation, unless we are in one of the following cases:

\begin{itemize}
    \item  $R$ is a normal subgroup of $\Pi_{K_s}$, that is $K_s(X)/K_s(B)$ is Galois. In this case $\tau(K(X))=K(X)$ for all $\tau \in \Pi_{K_s}$.
    \item The short exact sequence in eq. (\ref{ses}) splits. Then there exists a canonical selection of a section $s\colon \Gal(K_s/K) \rightarrow \Pi_{K}$ and the action of $h \in \Gal(K_s/K)$ can be defined as $h(R)=s(h) R s(h)^{-1}$. 
\end{itemize}

On the other hand, observe that the group $\Pi_K$ acts by conjugation on the set of subgroups of $\Gal(\Omega_D/K(B))$, and hence on $H  = \Gal(\Omega_D / K(X))$. Thus, the following set of subgroups is formed:
\begin{align*}
 S_H &= \{ \tau H \tau^{-1} \}_{\tau \in   \Pi_{K}/ \mathrm{N}_{\Pi_{K}}H},\\  
 S_R &= \{ \tau R \tau^{-1} \}_{\tau \in \Pi_{K_s}/ \mathrm{N}_{\Pi_{K_s}}R},
\end{align*}
where by $\mathrm{N}_{\Pi_{K}}H = \{g \in \Pi_{K} : gHg^{-1} = H \}$ (resp. $\mathrm{N}_{\Pi_{K_s}}R$) we denote the normalizer of $H$ (resp. $R$) in $\Pi_{K}$ (resp. $\Pi_{K_s}$). Thus, the group $\Gal(K_s/K)$ acts on $S_R$ and $S_H$ and the action is well-defined and independent of the choice of section. This aligns with the outer automorphism viewpoint we discussed in Subsection \ref{sec:profinite_braids}
and further illustrates that we should focus on the Galois closure of the cover corresponding to $R$, that is, $\cap_{g \in \Pi_{K_s}} gRg^{-1}$. More generally, the correct invariants to track under Galois action are the conjugacy classes of subgroups
$R \subseteq \Pi_{K_s}$ and $ H \subseteq \Pi_{K}$.

A useful way to resolve this ambiguity arising from inner automorphisms of $\Pi_K$ and $\Pi_{K_s}$, is via the Galois action on the monodromy representations. Following \cite{DebesDouai97}, we adopt the terminology $(G)$-cover to refer to either a $G$-cover or a mere cover. In both cases $G$ denotes the group of the cover, which is
\[
 G =
  \begin{cases}
  \Gal(K(X)/K(B)) & \text{for Galois covers} 
  \\
  \Gal(\widehat{K(X)}/K(B)) & \text{for mere covers}
  \end{cases}
\]
We also set
\[
  N=
  \begin{cases}
  G & \text{for Galois covers} 
  \\
  \mathrm{N}_{S_d}G & \text{for mere covers}
  \end{cases}
\]
and
\[
  C=
  \begin{cases}
  Z(G) & \text{for Galois covers} 
  \\
  \mathrm{Cen}_{S_d}G & \text{for mere covers}
  \end{cases}
\]
where $Z(G)$ is the center of $G$, and $\mathrm{N}_{S_d}G$ and $\mathrm{Cen}_{S_d}(G)$ denote the normalizer and centralizer of $G$ in $S_d$, respectively. Given a cover $\phi\colon \Pi_{K_s} \twoheadrightarrow G \leq N$ and $h \in \Gal(K_s/K)$ we define the action $h\cdot \phi$ as follows: 
choose a lift $\hat{h} \in \Pi_K$ of $h$ and let $\rho_{\hat{h}}$ denote the inner automorphism of $\Pi_K$, given by conjugation by $\hat{h}$. Then, the new cover $h\cdot \phi$ is defined by the composition $\phi \circ \rho_{\hat{h}}$, as shown in the following diagram:
\[
\begin{tikzcd}
\Pi_{K_s} \arrow[d, "\rho_{\hat{h}}"'] \arrow[rr, two heads, "\phi \circ \rho_{\hat{h}}"] & & G \\
\Pi_{K_s} \arrow[rru, two heads, bend right, "\phi"'] &
\end{tikzcd}
\]
Explicitly, this means:
\begin{equation}\label{eq:gal_monodromy}
h \cdot \phi(x) = \phi( \hat{h} x \hat{h}^{-1}),
\end{equation}
This definition is independent of the choice of lift $\hat{h}$. Indeed, the image $\Psi(\Pi_K) \leq S_d$ is well-defined up to conjugation, and thus for any $\tau \in \Pi_{K_s}$ we have $\phi(\hat h \tau x \tau^{-1} \hat{h}^{-1}) = \phi( \hat{h} x \hat{h}^{-1}) $ up to conjugation by elements in $N$.

Recall now the natural Galois action on $K_s$-varieties we discussed in the previous subsection and the induced action on the covers 
$X \to B$ to $^h \! X \to B$ for every $h$ in $\Gal(K_s/K)$, assuming that $B$ is a $K$-variety. Let $\phi$ be the corresponding representation of the mere cover $X\rightarrow B$ and denote by $\phi_h$ the one that corresponds to $^h \! X \to B$. 
As we will show in the next proposition, assuming that the ramification $D$ remains invariant, the covers $^h \! X\to B$ and the one corresponding to $h\cdot \phi$, defined by eq. \eqref{eq:gal_monodromy}, are isomorphic. 

\begin{proposition}\label{prop:gal_action_monodromy}
Let $X\rightarrow B$ be a mere cover of degree $d$ defined over $K_s$ and $h \in\Gal(K_s/K)$. Assume that both $B$ and the ramification divisor $D$ are $\Gal(K_s/K)$-invariant. 
Then, the maps $\phi_h$ and $h\cdot \phi$ have conjugate images by a fixed element in $ N_{S_d}G$. In particular, the cover $^h \! X \to B$ and the cover corresponding to $h\cdot \phi$ are isomorphic as $(G)$-covers.
\end{proposition}

\begin{proof}
To justify the final assertion, observe that if the maps $\phi_h$ and $h \cdot \phi_h$ have conjugate images by a fixed element in $N$, then the covers are isomorphic by Proposition \ref{prop:isomCovers}. 

For the first assertion, recall that $\Pi_{K_s}(1)= \{x \in \Pi_{K_s}: \phi(x)(1)=1 \}$, that is, the stabilizer of $1$ in $S_d$. Then, the cover $K_s(X)$ is $\Omega_D^{\Pi_{K_s}(1)}$ which means  $\Pi_{K_s}(1) =\Gal(\Omega_D/K_s(X))$. 
Fix a lift $\hat{h} \in \Pi_k$ of $h \in \Gal(K_s /K)$. 
Then the function field $K_s(^h \! X)$ is the fixed field of $\Omega_D$ by the conjugate subgroup:
$$\hat h \Pi_{K_s}(1) \hat h^{-1} = \Gal(\Omega_D/K_s(^h \! X))$$ 
Let $\hat{X}$ be the common Galois closure of the covers $X \to B$ and $ {^h X} \rightarrow B$ and denote by $G$ the group $\Gal(K_s(\hat{X})/K_s(B))$ which is a quotient of $\Pi_{K_s}$.
Under the quotient map $\Pi_{K_s} \twoheadrightarrow G$, the subgroups $H \coloneqq \Pi_{K_s}(1)$ and $ H^\prime \coloneqq \hat h \Pi_{K_s}(1) \hat h^{-1}$ of $\Pi_{K_s}$ descend to subgroups of $G$. Then, we have that $H^\prime  = n H n^{-1}$ where $n \coloneqq \hat{h}|_{K_s(\hat{X})}$ is the class $[\hat{h}]$ under the quotient map.
 The result follows by Lemma \ref{lemma:conjNorm}.
\end{proof}

We now see that there are two equivalent perspectives on the Galois action on $K_s$-varieties: 
the natural action on coefficients, and the induced action on monodromy representations, which is independent of a choice of rational basepoint of $B^*$.
Observe that in eq. \ref{eq:gal_monodromy}) the conjugation by $\hat{h}$ is not a conjugation inside $\Pi_{K_s}$, since $\hat{h}$ does not necessarily belong in $\Pi_{K_s}$, but a lift in $\Pi_K$. This distinction is crucial when interpreting the action on monodromy.
Furthermore, if we start with a $(G)$-cover $X\to B$ defined over $L$ and the corresponding homomorphism $\phi \colon \Pi_L \rightarrow N$ and choose a lift $\hat{h} \in \Pi_L$ of an element $h\in \Gal(K_s/L)$, then we have that 
\[
h \cdot \phi (x) = \phi(\hat{h} )\cdot  \phi(x) \cdot \phi (\hat{h})^{-1}, 
\] 
that is $h\cdot \phi$ and $\phi$ are conjugate in $N$. In other words, the cover corresponding to $h\cdot \phi$ will always remain isomorphic to $X\rightarrow B$ over $L$. To explain this further, if the $(G)$-cover defined over $L$ corresponds to $R \leq \Pi_{K_s}$, then for any element $h \in \Gal(K_s/L)$, the conjugation action of a lift $\hat{h}\in \Pi_L$ preserves the conjugacy class of $R$ in $\Pi_{K_s}$. This roughly translates to the statement that the action of the arithmetic information in $\Gal(K_s/L)$ will keep $R$ invariant, and the lift $\hat{h}$ acts as a deck transformation of the cover. Therefore, even though $X\rightarrow B$ is defined over $L$, the $\Gal(K_s/L)$-action on its monodromy representation will not stabilize the cover itself, but its isomorphism class.

We will use this perspective in the topological setting, where it will manifest itself as a condition of definability and as a requirement on the action of the mapping class group.

To finish this section, its useful to think of definability of a cover over a field $L$ in terms of the monodromy representation, so we state the following as a proposition, although it follows from basic Galois correspondence.

\begin{proposition}
  \label{prop:DebesArith}
Let $f\colon X \rightarrow B$ be a  $(G)$-cover over $K_s$, corresponding to a homomorphism $\phi \colon \Pi_{K_s} \twoheadrightarrow G \leq N$. The $(G)$-cover $f$ is defined over the field $L$ if and only if the homomorphism $\phi$ extends to a homomorphism $\Pi_L \rightarrow N$. 

\end{proposition}

\subsection{Topological Actions}
In this section, we introduce the classical theory of monodromy representations of topological and branched covers $X\rightarrow B$ of compact Riemann surfaces, following \cite[Ch. 3.4]{MR1326604}. Our goal is to construct the analogues of the previous section in terms of mapping class groups. From now on, we restrict ourselves to algebraic curves in every topological setting. 
That is, since we our approach relies on the punctured version of the Dehn-Nielsen-Baer Theorem \ref{thm:Dehn-Nielsen-Baer}, we assume that the induced topological covers $X^\circ \rightarrow B^*$ occur between hyperbolic surfaces. Keep in mind, as a general guiding principle, that the fundamental groups of varieties of dimension greater than one can be finite or too hard to compute, see e.g. the theorem of Serre that for any finite group $G$ there exists a smooth projective variety over $\mathbb{C}$ with its topological fundamental group being $G$ \cite[Prop. 15]{SerreMexico}. In contrast, the fundamental groups of curves are fully classified by their genus, thus, this restriction remains interesting.

Recall that two topological covers $f\colon X^\circ \rightarrow B^*$ and $ f^\prime \colon Y^\circ \rightarrow B^*$ are isomorphic if there exists a homeomorphism $g \colon X^\circ \rightarrow Y^\circ$ such that $f^\prime \circ g = f$. If the covers correspond, respectively, to subgroups $R, R^\prime \leq \pi_1(B^*,b_0)$, then the covers are isomorphic if and only if $R$ and $R^\prime$ are conjugate in $\pi_1(B^*, b_0)$. This leads to the bijection:

\[
\left\{
\begin{array}{c}
\text{Isomorphism classes}  \\
\text{of connected coverings} \\
f\colon X^\circ \to B^*
\end{array}
\right\}
\; \longleftrightarrow \;
\left\{
\begin{array}{c}
\text{\  Conjugacy classes \ } \\
\text{ of subgroups } \\
R \leq \pi_1(B^*, b_0)
\end{array}
\right\} \hspace{0.6 cm}
\]

\vspace{0.3 cm}
For finite covers, using the action of $\pi_1(B^*,b_0)$ on the fiber of $b_0$, the bijection extends to the monodromy representations. We have the following bijection:

\[
\left\{
\begin{array}{c}
\text{Isomorphism classes}  \\
\text{of connected coverings} \\ 
f \colon X^\circ \rightarrow B^* \\
\text{of degree } d
\end{array}
\right\}
\; \longleftrightarrow \;
\left\{
\begin{array}{c}
\text{Homomorphisms} \\
\text{with transitive image} \\
\Phi \colon \pi_1(B^*,b_0) \rightarrow S_d \\
\text{(up to conjugacy in } S_d \text{)}
\end{array}
\right\}
\]

\vspace{0.3 cm}
If $X\rightarrow B$ is a branched cover between compact Riemann surfaces with ramification divisor $D$, such that $B^* = B - D$, then, under the requirement that the maps are holomorphic and with suitable treatment at the cusps, the bijection becomes:

\[
\left\{
\begin{array}{c}
\text{Isomorphism classes} \\
\text{of holomorphic maps} \\
f \colon X \rightarrow B \text{ of degree } d, \\
\text{with branch points in } D
\end{array}
\right\}
\; \longleftrightarrow \;
\left\{
\begin{array}{c}
\text{Homomorphisms} \\
\text{with transitive image} \\
\Phi \colon \pi_1(B^*,b_0) \rightarrow S_d \\
\text{(up to conjugacy in } S_d \text{)}
\end{array}
\right\}
\]

\vspace{0.3 cm}
See \cite[Ch. 3.4]{MR1326604} for the precise constructions.

In contrast to the arithmetic case, where such a description arises naturally, there is no analogous way to describe a topological action on the algebraic curve via the coefficients. We therefore introduce the following definition. 

\begin{definition} \label{def:topological_defined}
Let $X\rightarrow B$ be a branched cover with associated topological cover $X^\circ \rightarrow B^*$, corresponding to $R\leq \pi_1(B^*,b_0)$. Let $A$ be a subgroup of $\mathrm{Mod}(B^*)$ such that, for every $\alpha $ in $A$, the induced class of automorphisms $\alpha_*$ preserves the conjugacy class of $R$ in $\pi_1(B^*,b_0)$. For a subgroup $H \leq \Pi_A$ consider the following conditions:
\begin{enumerate}
\item $H\cap \pi_1(B^*,b_0) = R$.
\item $H \leq N_{\Pi_A}(R)$.
\item If $H$ is normal in $\Pi_A$, then $R$ is normal in $\pi_1(B^*,b_0)$.
\item $H/R \cong A$ (regularity condition).
\end{enumerate}
If a subgroup $H$ satisfying conditions $(1)-(3)$ exists, we say the cover $X\rightarrow B$ is {\em defined} over $A$, with corresponding group $H\leq \Pi_A$, and we denote this by $X\rightarrow (B,A)$. 
If, in addition, condition $(4)$ holds, we say the cover is regular with respect to $A$, or simply $A$-regular.

\end{definition}

\begin{remark}
These conditions are designed to mirror the arithmetic setting, where a cover $X\rightarrow B$ is defined over $K$, giving rise to the subgroups $H=\Gal(\Omega_D,K(X))$ and $R=\Gal(\Omega_D,K_s(X))$. 
In particular, $(1)$ is about the fact that we can view a $K$-variety $X$ as a $K_s$-variety, $(2)$ is that the extension $K_s(X)/K(X)$ is Galois, $(3)$ mimics the fact that if $K(X)/K(B)$ is a Galois extension, then so is $K_s(X)/K_s(B)$, and $(4)$ is whether the group of this extension is the full group $\Gal(K_s/K)$.

Note that, under the regularity assumption for a Galois cover $X\rightarrow (B,A)$ with deck group $G = \pi_1(B^*,b_0)/R$, holds that $G\cong \Pi_A/H$, which parallels the arithmetic situation $\Gal(K_s(X)/K_s(B)) \cong \Gal(K(X)/K(B))$.
\end{remark}

Therefore, we can extend the notion of $(G)$-covers to include the additional structure given by the group $A$. That is, instead of simply writing $ X\rightarrow B$, we use the notation $X\rightarrow (B,A)$, which reflects the role of $A$ in the topological setting—regardless of whether $R$ and $H$ are normal in $\pi_1(B^*,b_0)$ and $\Pi_A$, respectively. To remain consistent with the framework of \cite{DebesDouai97}, we require that
$(G)$-covers be regular.

Now, we discuss the notion of $G$-cover attached to the mere cover $X\rightarrow (B,A)$. The Galois closure of $X\rightarrow (B,A)$ is well-defined. Indeed, let $H$ and $R$ be the subgroups corresponding to the cover, as before. Denote by $\widehat{H} = \cap_{g \in \Pi_A} gHg^{-1}$ the normal closure of $H$ in $\Pi_A$ and, similarly, denote by $\widehat{R}$ the normal closure of $R$ in $\pi_1(B^*,b_0)$. Then,
\begin{align*}
\widehat{H}\cap \pi_1(B^*,b_0) &= \bigcap\limits_{g \in \Pi_A} g Hg^{-1} \cap \pi_1(B^*,b_0) \\
& = \bigcap\limits_{g\in \Pi_A } g R g^{-1} \\
& = \bigcap\limits_{g \in \pi_1(B^*,b_0)} gRg^{-1} \\ 
&= \widehat{R}
\end{align*} by the normality of $\pi_1(B^*,b_0)$ and the fact that the outer action of $A$ preserves the conjugacy class of $R$ for any lift. The above also proves that $N_{\Pi_A}(\widehat{R}) = \Pi_A$, thus $(2),(3)$ of Definition \ref{def:topological_defined} are automatically satisfied by $\widehat{H}$. Note that the cover corresponding to $\widehat{H}$ is not necessarily regular, as we only get an injection $\widehat{H}/\widehat{R} \xhookrightarrow{} A$. Indeed, observe that $R\cap \widehat{H} = \widehat{R}$, so $\widehat{H}/\widehat{R} = \widehat{H}/\widehat{H}\cap R \cong \widehat{H}R/R$ which embeds into $H/R \cong A$. Thus, the cover is a $G$-cover, i.e. is regular with respect to $A$, if and only if $\widehat{H}R = H$. In the arithmetic setting, this amounts to the Galois closure $\widehat{K(X)}/K(B)$ being a regular extension of $K(X)/K(B)$.

Now we have the necessary tools to extend the monodromy representations with respect to the subgroups of $\mathrm{Mod}(B^*)$, thus the analogue of Proposition \ref{prop:merecovertrasnArith} is the following.

\begin{proposition}{\label{prop:mereCoverstransTop}}
  Degree $d$ mere covers $X \rightarrow (B,A)$, $A \leq \mathrm{Mod}(B^*)$ with ramification locus in $D$ correspond to transitive representations 
  \[
    \Psi \colon \Pi_A \longrightarrow S_d,
  \]
  such that the restriction to $\pi_1(B^*,b_0)$ is also transitive.
\end{proposition}
\begin{proof}
Let $R \leq \pi_1(B^*,b_0)$ and $H\leq \Pi_A$ be the subgroups corresponding to the cover $X\rightarrow (B,A)$ and let $\widehat{H}$ be the normal closure of $H$ as above. Then, analogously to the arithmetic construction in the proof of Proposition \ref{prop:merecovertrasnArith}, we have $G = \Pi_A/\widehat{H}$ and we get the transitive representation by the action of $G$ on the cosets of $H/\widehat{H}$, after fixing a labeling such that $1$ corresponds to $H/\widehat{H}$. The restriction to $\pi_1(B^*,b_0)$ factors through the quotient map $\pi_1(B^*,b_0) \rightarrow \pi_1(B^*,b_0)/\widehat{R} \cong G$ and the transitivity follows similarly, from classical monodromy theory.

Conversely, given such a homomorphism $\Psi \colon \Pi_A \to S_d$, denote by $\Pi_A(i)$ the stabilizer of $i$ in $S_d$, and by $\pi_1(i)$ the stabilizer of $\Psi|_{\pi_1(B^*,b_0)}$. Let $H=\Pi_A(1)$ and set $R=H\cap \pi_1(B^*,b_0) = \pi_1(1)$. The goal is to prove that conditions (2)–(4) of Definition~\ref{def:topological_defined} hold for $H$ with respect to $R$. 

Condition (2):  
For any $h$ in $H$, it is clear that $hRh^{-1}$ stabilizes $1$ via its image under $\Psi$. Moreover, $hRh^{-1} \subset \pi_1(B^*,b_0)$ since the outer action of $A$ preserves the conjugacy class of $R$ inside $\pi_1(B^*,b_0)$. It follows that $hRh^{-1} = R$. 

Condition (3):
Assume that $H$ is normal in $\Pi_A$. Observe that for any $g$ in $\Pi_A$ the group $gHg^{-1}$ stabilizes
$\Psi(g)(1)$. Since $H$ is normal, this implies $\Psi(g)(1)=1$ and, in particular, this holds for all $g$ in $\pi_1(B^*,b_0)$, which implies that $gRg^{-1} = R$ for all $g$ in $\pi_1(B^*,b_0)$. Hence, $R$ is normal if $H$ is normal.

Condition (4):
Proving regularity amounts to showing that $H \pi_1(B^*,b_0) = \Pi_A$, since 
$$H/R = H/H\cap \pi_1(B^*,b_0) \cong H\pi_1(B^*,b_0)/\pi_1(B^*,b_0)$$ 
embeds into $\Pi_A/\pi_1(B^*,b_0) \cong A$. By transitivity of $\Psi$ and its restriction to $\pi_1(B^*,b_0)$, we have that $[\Pi_A:H] = [\pi_1(B^*,b_0):R] = d$.
Let $\pi_1(B^*,b_0) = \sqcup_{i=1}^d g_iR$ be a decomposition into cosets with representatives $g_i$ in $\pi_1(B^*,b_0)$ for all $i=1,\ldots,d$. Then, $g_i H = H$ if and only if $g_i \in H\cap \pi_1 = R$, which implies that $g_i=1$. By a similar argument, $g_iH = g_j H$ if and only if $g_i=g_j$. Therefore, $g_i$ are coset representatives of $\Pi_A/H$ even though they lie in $\pi_1(B^*,b_0)$.
\end{proof}
Similarly to the arithmetic case, the $(G)$-cover $X\rightarrow B$ is defined over $A \leq \mathrm{Mod}(B^*)$ if and only if the associated representation $\Psi \colon \pi_1(B^*,b_0) \rightarrow S_d$ extends to a representation $\Pi_A \rightarrow S_d$, since this implies the existence of the pair $R,H$ described in Definition \ref{def:topological_defined}. As definability regarding $A$ implies definability regarding each subgroup $A^\prime \leq A$, we state the above statement formally as a proposition analogous to Proposition \ref{prop:DebesArith}.

\begin{proposition}
  \label{prop:DebesTop}
Let the $(G)$-cover $X \rightarrow (B,A^\prime)$ correspond to a homomorphism $\Psi: \Pi_{A^\prime}\twoheadrightarrow G \leq N$. The $(G)$-cover is defined over the group $A$, where $A^\prime \leq A \leq \mathrm{Mod}(B^*)$ if and only if the homomorphism $\Psi$ extends to a homomorphism $\Pi_{A} \rightarrow N$.
\end{proposition}

Correspondingly with the arithmetic case, we also have the action of $\mathrm{Mod}(B^*)$ to these representations in the following way. Given an element $h\in \mathrm{Mod}(B^*)$, choose a lift $\hat{h}\in \Pi_{\mathrm{Mod}(B^*)}$. Let $\rho_{\hat{h}}$ denote the inner automorphism of $\Pi_{\mathrm{Mod}(B^*)}$ determined by the conjugation by $\hat{h}$. The new representation $h\cdot \Psi$ is defined to be the composition $\Psi \circ \rho_{\hat{h}}$ as shown in the following diagram.

\[
\begin{tikzcd}
\pi_1(B^*, b_0) \arrow[d, "\rho_{\hat{h}}"'] \arrow[rr, two heads, "\Psi \circ \rho_{\hat{h}}"] & & G \\
\pi_1(B^*, b_0) \arrow[rru, two heads, bend right, "\Psi"'] & & 
\end{tikzcd}
\]

Explicitly, this means:
\begin{equation}\label{eq:top_monodromy}
^h \! \Psi \coloneqq h\cdot \Psi(x) = \Psi( \hat{h} x \hat{h}^{-1}),
\end{equation}
This construction is independent of the choice of lift $\hat{h}$, since the image of the representation in $S_d$ is well-defined up to conjugation. More precicely, for any $\tau \in \pi_1(B^*,b_0)$ we have $\Psi(\hat h \tau x \tau^{-1} \hat{h}^{-1}) = \Psi( \hat{h} x \hat{h}^{-1}) $ up to conjugation by elements in the normalizer of $G$ in $S_d$.

In order to compare the representations $\Psi, ^h \! \Psi$, we first need to extend the notion of isomorphism between covers $X\to B$ to the setting of $(G)$-covers defined over $A \leq \mathrm{Mod}(B^*)$. In the arithmetic setting, if two covers $X\to B$ and $X^\prime \rightarrow B$ are isomorphic over $K$ by an isomorphism $\chi \colon X \to X'$, then the extensions $K(X)/K(B)$ and $K(X^\prime)/K(B)$ are conjugate by an element $g_\chi$ in $\Pi_K(B^*)$. That is, $g_\chi\Gal(\Omega_D/K(X))g_\chi^{-1} = \Gal(\Omega_D/K(X^\prime))$.

This leads us to the following definition:
\begin{definition}
Let $X \to (B,A)$ and $X^\prime\to (B,A)$ be two covers defined over $A\leq \mathrm{Mod}(B^*)$. We say that the covers are isomorphic over $A$, or simply $A$-isomorphic, if there exists a homeomorphism $\chi \colon X \to X^\prime$ such that the corresponding subgroups $H, H^\prime \leq \Pi_A$ are conjugate; that is $H^\prime = g_\chi H g_\chi^{-1}$ for an element $g_\chi \in \Pi_A$.
\end{definition}

Note that the homeomorphism $\chi$ induces an element $g_\chi \in \Pi_A \leq \Pi_{\mathrm{Mod}(B^*)}$ by considering its isotopy class in $\mathrm{Mod}(B^*,b_0)$, the basepoint $b_0$ is now considered marked, and the isomorphism of $\Pi_{\mathrm{Mod}(B^*)} \cong \mathrm{Mod}(B^*,b_0)$ arising from the Birman exact sequence (\ref{ses:Birman}). Additionally, the definition of $A$-isomorphic covers is compatible with isomorphic topological covers; that is if $\chi$ is an $A$-isomorphism, then it is an isomorphism of topological covers.

The topological analogue of Proposition \ref{prop:isomCovers} is the following: 
\begin{proposition}{\label{prop:isomCoversTop}}
  Two mere covers $X \rightarrow (B,A)$, $X^\prime \rightarrow (B,A)$ are isomorphic over $A$ if and only if the corresponding representations $\Psi,\Psi': \Pi_A \rightarrow S_d$ are conjugate by an element $\varphi$ in the normalizer $N_{S_d}(G)$ of the image group $G=\Psi(\Pi_A) = \Psi'(\Pi_A)$, that is
  \[
  \Psi'(x) = \varphi \Psi(x) \varphi^{-1} \text{ for all } x \in \Pi_A. 
\]

\end{proposition}
\begin{proof}
  Identical to the proof of Proposition \ref{prop:isomCovers}, under proper adjustments for the subgroups $R,H$.
\end{proof}
Therefore, for a $(G)$-cover $X\rightarrow (B,A)$, corresponding to the representation $\Psi:\Pi_A\rightarrow G\leq N$, the action of and element $h \in A$ on $\Psi$ yields the representation:
\[ 
^h  \Psi (x) = \Psi(\hat{h}) \Psi(x) \Psi(\hat{h}^{-1}) = \Psi_h \cdot \Psi(x) \cdot \Psi_h^{-1},  \quad  \Psi_h \in N_{S_d}(G)
\] 
where $\hat{h} \in \Pi_A$ is a lift of $h$. This shows that the cover corresponding to $^h \Psi$ is isomorphic over $A$ to the cover $X\rightarrow (B,A)$.  

\begin{remark} Notice that for a $(G)$-cover $X\rightarrow (B,A)$, with $A\leq \mathrm{Mod}(B^*)$ the corresponding exact sequence
\[1 \rightarrow R \rightarrow H \rightarrow A \rightarrow 1,\] yields a homomorphism $ \phi: A\rightarrow \mathrm{Out}(R)$, by the outer conjugation action of $A$ on $R$. Since, in the punctured case, we require the mapping class group to preserve the conjugacy classes of the loops corresponding to punctures of $B^*$, the action of $A$ on $R$ will have to preserve the loops corresponding to the ramified points of $X$, which are exactly the punctures of $X^\circ$. This implies that the image of the previous homomorphism actually lies in $\mathrm{Out}^*(R)$, that is
\[ \phi: A \rightarrow \mathrm{Out}^*(R) \cong \mathrm{Mod}(X^\circ) \leq \mathrm{Out}(R).\] 
This is interesting in its own right, as it provides a map from subgroup of $\mathrm{Mod}(B^*)$ into the mapping class group $\mathrm{Mod}(X^\circ)$. This naturally raises questions such as: 
\begin{itemize}
\item What are the kernel and image of this homomorphism? How do they vary over different choices of the cover $X\rightarrow (B,A)$?  These questions resemble the arithmetic questions regarding $\phi_K$ in Chapter \ref{ch:profinite_braids} for various $B^*$. 
\item Is this related to the Birman-Hilden property? See the survey \cite{MR4275077} and the later Chapter \ref{ch:Hilden}.  Specifically, since $A/\ker\phi$ is isomorphic to a subgroup $S \leq \mathrm{Mod}(X^\circ)$, can we think of this isomorphism as dual to the Birman–Hilden property? While the latter addresses when mapping classes on the cover descend to the base, our perspective here addresses when mapping classes on the base can be mapped into the mapping class group of the cover.
\end{itemize}
\end{remark}

\subsection{Non-isomorphic covers}
So far in this section, we have considered the case where the action of the mapping class group preserves the isomorphism class of an initial cover $X\rightarrow B$. To obtain a cover non-isomorphic to $X$, it suffices to require that $h_* [R] \neq [R]$, for some element $h$ in $\mathrm{Mod}(B^*)$. This means for any lift $\hat{h}$ of $h$, the image $\hat{h}_*(R)$ is not conjugate to $R$, and thus does not correspond to an isomorphic topological cover.

More precisely, any lift $\hat{h}$ of $h \in \mathrm{Mod}(B^*)$ is a homeomorphism $\hat{h} \colon B^* \to B^*$ which, by the uniqueness of path-lifting, see \cite[Prop. 13.5]{MR1343250}, induces another homeomorphism $\tilde{h} \colon X^\circ \rightarrow Y^\circ$, where $Y^\circ = \widetilde{B^*}/\hat{h}(R)$, if and only if $\hat h_*(R)$ is conjugate to $R$. Note that the underlying topological space remains the same, but there are different covering maps corresponding to $R$ and $\hat{h}_*(R)$.

In the case of algebraic curves, once we complete $B^*$ to $B$ by adding a holomorphic structure, there is a unique way to compactify the topological cover $X^\circ\rightarrow B^*$, e.g. by \cite[Prop. 19.9, Ex. 19.10]{MR1343250}, to a holomorphic branched cover $X\rightarrow B$. Furthermore, each subgroup $R$ and $\hat{h}_*(R)$ gets a unique compactification to a curve $X$ and ${^{\hat{h}} \! X}$ respectively, by assuming that the underlying base $B^*$ has the same holomomorphic structure in both cases. This implies that in the latter case we don not choose the holomorphic structure translated by $\hat{h}$ but rather the initial one. Compare this with the arithmetic setting; recall the same structure dilemma in eq. (\ref{eq:schemeActi}).

Moreover, in the hyperbolic case we are interested in, $\pi_1(B^*,b_0)$ is non-abelian and finitely generated. It is well-known, see for instance \cite[Thm. 29.10]{bogoGrp}, that any finitely generated group has only a finite number of subgroups of a given index $d$. Consequently, there are only finitely many coverings of degree $d$ corresponding to subgroups $R$, and in particular, the conjugacy class $[R]$ is also finite.

We would like to be able to define an action of $h$ in $\mathrm{Mod}(B^*)$ on the set of finite branched covers $X\rightarrow B$. However, choosing different representatives in the class of $h$, which differ by inner automorphisms under the Dehn-Nielsen-Baer, complicates this in general, since the subgroup $\hat{h}(R)$ varies over the elements in $[R]$, corresponding to different covers. By the previous discussion, this orbit is finite. In fact, $ {^h \! X}$ is well-defined if the cover $X\rightarrow B$ is Galois. However, for a mere cover $X\rightarrow B$, we have that ${^h \! X}$ is not well-defined, as there is a finite choice of curves it can land on for every element in $[R]$. Thus it could become well-defined but non-canonically.

Nevertheless, the covering isomorphism class of ${^h \! X}$, for any non-canonical choice, is well-defined, consisting of all such possible choices corresponding to elements in $[\hat{h}(R)]$, for any lift $\hat{h} $ of $h$.
We can avoid the ambiguity by using the action of $h$ in $\mathrm{Mod}(B^*)$ on the monodromy representation $\Psi\colon \pi_1(B^*,b_0)\rightarrow S_d$, yielding ${^h\Psi}$ as we defined it previously. Even though $^hX$ is not well-defined, as there is a finite orbit set $\{ ^{\hat{h}\! }X  :  \hat{h} \textrm{ lift of } h \in \mathrm{Mod}(B^*)\}$, the cover corresponding to the representation $^h \Psi$, will be inside this orbit and thus isomorphic to $^{\hat{h}}\! X$ for all possible lifts. This is as close as saying $^hX$ and the cover corresponding to $^h\Psi$ are isomorphic as $(G)$-covers.

We have already discussed the action of a group $A\leq \mathrm{Mod}(B^*)$ on the representation $\Psi\colon \Pi_A\rightarrow S_d$ of a $(G)$-cover $X\rightarrow (B,A)$, we can however act on this by the whole mapping class group, yielding a new representation 
\[^h\Psi \colon \Pi_{hAh^{-1}} \rightarrow S_d,\] defined by
\[^h\Psi(x) \coloneqq \Psi(\hat{h}x\hat{h}^{-1}),\] corresponding to a new cover $X^\prime\rightarrow (B,hAh^{-1})$.

\subsection{Comparison}
We have discussed a common framework for both the group actions of $\Gal(K_s/K)$ and $\mathrm{Mod}(B^*)$ on the $(G)$-covers $X\rightarrow B$, realized via outer automorphisms of the (geometric) fundamental groups. The ambiguity of inner automorphisms is dealt with by considering the actions on the monodromy representations, which we summarize in the table at the end of this section.

In the arithmetic setting, the situation can be somewhat simpler, as the group $\Gal(K_s/K)\cong \Gal(K_s(X)/K(X))$ acts naturally on $K_s(X)$ via field automorphisms. In the next step, this situation aligns with the topological setting when inner automorphisms of $\Pi_{K_s}$ are taken into account. These correspond to possible lifts of field automorphisms from $K_s(X)$ to $\Omega_D$ when an element $\sigma \in \Gal(K_s/K)$ is lifted to $\hat{\sigma} \in \Pi_K$.

Using the more concrete language of function fields has lot of advantages in the arithmetic setting. However, this language does not 
allow us to handle the topological setting, where there is no analogue of the notion of field. 
To address this, we adopt the viewpoint that in the absence of fields, we can work with the corresponding group action data instead. This serves as the common ground between the arithmetic and topological contexts.
In the following we wish to impose the idea, 
that in order to handle the absence of fields in the topological setting, we will refer to their corresponding group action data. As the common ground 

Suppose we have a $K_s$-variety $X \rightarrow \Spe K_s$, which is defined over $K$, that is $X={^\sigma \! X}$ for all $\sigma \in \Gal(K_s/K)$, see Proposition \ref{prop:invArith}. For any intermediate field $K \subset L \subset K_s$, we can consider the base changes:
 $X \times _{K}  K_s \rightarrow  K_s$, $X \times _{ K}  L \rightarrow  L$.

In the diagram of eq. (\ref{diag:fields}), the tower of fields on the left hand side corresponds to group actions on the right hand side of the diagram. By the notation $(X,G)$, we mean that the variety $X$ is equipped with a $G$-action. In terms of $(G)$-covers, a $(G)$-cover $X\rightarrow B$ over $K$ is simply denoted as $X\rightarrow (B,\Gal(K_s/K))$.

We also denote by $\widetilde{B^*}$ we denote the universal 
covering of $B^*$, which corresponds to $\Omega_D$ as an analogue for the universal cover of fields. In this analogy, $\Omega_D$ plays the role of a "universal cover" of fields, in the sense that $\Pi_K(B^*)$ acts transitively on the fibers of every unramified place in the extension $\Omega_D/K(B)$ - similarly to how $\pi_1(X)$ acts on the fibers in the covering $X^\circ \to B^*$.

\newsavebox{\ssss}
\sbox{\ssss}
{
\tiny{
  \xymatrix@R=.7pc{
  X\times_{K}\Spe(K_s) \ar[d] 
  \\ 
  \Spe(K_s)
  }}
}
\newsavebox{\ssssa}
\sbox{\ssssa}
{
\tiny{
  \xymatrix@R=.7pc{
  X\times_{K}\Spe(L) \ar[d] 
  \\ 
  \Spe(L)
  }}
}
\newsavebox{\ssssb}
\sbox{\ssssb}
{
\tiny{
  \xymatrix@R=.7pc{
  X\times_{K}\Spe(K) \ar[d] 
  \\ 
  \Spe(K)
  }}
}
\begin{equation}
  \label{diag:fields}
  \xymatrix{
  & \Omega_D \ar@{-}[d]      &  &  & \widetilde{B^*} \ar@{-}[d]
  \\ 
  *!<0pt,-3pt>{
  \usebox{\ssss} }   & K_s(X)  \ar@{-}[d] &  K_s \ar@{-}[d]& \{1\} \ar@{-}[d] & X\rightarrow (B,\Gal(K_s /K))\ar@{-}[d] 
  \\
  *!<0pt,-3pt>{
 \usebox{\ssssa}} & L(X)    \ar@{-}[d]&  L \ar@{-}[d] & \Gal(K_s /L) \ar@{-}[d] & X\rightarrow (B,  \Gal(Ks / L) )\ar@{-}[d] 
 \\ 
 *!<0pt,-3pt>{
 \usebox{\ssssb} }&  K(X)      &  K &   \Gal(K_s /K) & X \rightarrow (B,\{1\})   
  \\
      & K(B) \ar@{-}[u]  & &  & B\ar@{-}[u] }
\end{equation}
In the diagram above, the second column consists of a sequence of function fields, each invariant under the action of the groups in the forth column. These are subgroups of  $\Gal ( \Omega_D/K(B) )$. 

\vspace{0.25 cm}

\textbf{Moto: Descent on Fields - Ascent on Groups} \newline
This language can be used for both the arithmetic and the topological setting. Indeed, a {\em descent datum} of the variety $X$ in terms of fields 
$K_s \supseteq L \supseteq K$ corresponds to an {\em ascent datum} $\{(X,\Gal(K_s/L))\}_{K_s \supseteq L \supseteq K}$. 
This perspective transfers naturally to the topological setting (without involving fields) by considering an ascent datum $\{(X,H)\}_{G\geq H}$, where $G$ is a group of isotopy classes of homeomorphisms. This justifies the notation $X\rightarrow (B,A)$, where $A\leq \mathrm{Mod}(B^*)$.

{\small
\begin{longtable}[c]{>{\raggedright\arraybackslash}p{4.9cm}  >{\raggedright\arraybackslash}p{4.3cm}  >{\raggedright\arraybackslash}p{3.1cm}}
\rowcolor{headerblue}
\textbf{Number Theory} & \textbf{Topology} & \textbf{Common Notation} \\
\endfirsthead

\rowcolor{headerblue}
\textbf{Number Theory} & \textbf{Topology} & \textbf{Common Notation} \\
\endhead

$K_s(X)/K_s(B)$ & $X\rightarrow B$ & $X\rightarrow B$  \\

\rowcolor{rowgray}
$\Psi \colon \Pi_{K_s} \rightarrow S_d$ 
  & $\Psi \colon \pi_1(B^*,b_0) \rightarrow S_d$ 
  & \\
  
$K(X)/K(B)$ 
  & $X\rightarrow (B,A)$ 
  & $X\rightarrow (B,A), \ A\leq \mathfrak{\Gamma}$
  \\

\rowcolor{rowgray}
$\Psi\colon \Pi_K\rightarrow S_d$
& $\Psi \colon \Pi_A \rightarrow S_d$ 
& $\Psi \colon \Pi_A \rightarrow S_d, \ A\leq \mathfrak{\Gamma}$
\\

$^\sigma \! X$, $\sigma \in \Gal(K_s/K)$,
  & $^h \! X$, $h \in \mathrm{Mod}(B^*)$ 
  & \\
  
where $X$ is a $K_s$-variety 
& where $X\rightarrow B$ is Galois 
&
\\

\rowcolor{rowgray}
 $^\sigma\Psi \colon \Pi_{K_s}\rightarrow S_d$
  & $^h\Psi \colon \pi_1(B^*,b_0) \rightarrow S_d$
  & \\
\end{longtable}
}

\section{Field \& Group of Moduli - Definition - Invariance}\label{sec:field of moduli, definition, invariance}
In this section, we introduce several refinements of classical notions, namely the concepts of fields of moduli, definition and invariance, which correspond respectively to the notions of groups of moduli, definition and invariance. We begin by formulating these ideas in the arithmetic setting in order to illustrate the fundamental principle that descent on fields corresponds to ascent on groups. This perspective is the core motivation for our definitions and extends naturally to the topological setting, where no underlying fields are present. We then recall the classical descent theorem of Weil (Theorem~\ref{thm:weil's descent}) and establish an analogue in the topological context (Theorem~\ref{thm:Weil topological}). Finally, after comparing the arithmetic and topological notions, we present examples illustrating both situations: cases in which the field of moduli coincides with the field of definition, and cases in which the two differ. These are discussed in Subsections~\ref{exm: moduli = definition} and~\ref{exm:field of moduli vs field of definition}, respectively.

\subsection{Arithmetic Field of Moduli / Definition} \label{sec:FieldMODDefinition}

Recall from Remark \ref{rem:non isomorphic K_s-schemes} that, as $K_s$-schemes, $X$ and $^\sigma \! X$ might not be isomorphic. Thus, it is worth exploring conditions under which they become isomorphic. In this section, we explore the different ways that $X$ and its images via the $\Gal(K_s /K)$-action are isomorphic and use the extra structure that appears, in order to descend to the field over which $X$ (or something isomorphic to $X$) is defined. We define everything in the greatest possible generality, i.e. in terms of any subgroup $A$ of $\mathfrak{\Gamma} = \Gal(K_s/K)$. When $A$ is closed, it corresponds to an intermediate field extension
$K= K_s^{\mathfrak{\Gamma}} \subseteq K_s^A \subseteq K_s$. We then specialize to the case $A=\mathfrak{\Gamma}$, which is the most interesting case. Recall that, by Remark \ref{rem:generalizedgroupaction}, the closed subgroup $A = \Gal(K_s / K_s^A)$ also acts on $X$.

\begin{definition}
  \label{def:groupModuliArAA}
For an algebraic $K_s$-variety $X$ and a closed subgroup $A \leq \mathfrak{\Gamma}$ define the {\em relative group of moduli of $X$ over $A$}
\[
  A_X^{\mathrm{mod}} = \{ 
\sigma \in A : \exists f_{\sigma} \colon X \xrightarrow[]{\simeq}   {^\sigma \! X} \} ,
\]
where the isomorphism is defined over  $K_s$, and is an isomorphism of $K_s$-schemes. 
The {\em absolute group of moduli of $X$ (over $\mathfrak{\Gamma}$)} is the group $\mathfrak{\Gamma}_X^{\mathrm{mod}}$. \newline
The {\em relative field of moduli of $X$ over $A$} is defined to be the fixed field $F=K_s^{A_X^{\mathrm{mod}}}$. 
The {\em absolute field of moduli of $X$ (over $\mathfrak{\Gamma}$)} is defined to be the fixed field $F=K_s^{\mathfrak{\Gamma}_X^{\mathrm{mod}}}$.
\end{definition}

Note that, by \cite[Proposition 2.1]{DebesEmsalem} and the discussion below that, the field of moduli is a closed subgroup of $A$ and thus the fixed field $K_s^{A_X^{\mathrm{mod}}}$ makes sense.

\begin{remark}[Moduli interpretation]
 Note that the isomorphism between $X$ and ${^\sigma \! X}$ in the definition of the group of moduli need not be induced by the projection map $\tilde{\sigma}$ of Definition \ref{def:schemeactedupon}.
 It is known that the 
 field of moduli of $F=K_s^{A_X^{\mathrm{mod}}}$ (resp. $F=K_s^{\mathfrak{\Gamma}_X^{\mathrm{mod}}}$) is a finite extension of $K_s^A$. (resp. of $K$).
The field of moduli $F$ is the smallest field $F \subseteq K_s$ such that each $\sigma \in A_X^{\mathrm{mod}}$ (resp. $\sigma \in \mathfrak{\Gamma}_X^{\mathrm{mod}}$) carries $X$ to an isomorphic copy of itself.  
In some sense, if there is a space $\mathcal{M}$, such that the points of $\mathcal{M}$ correspond to isomorphism classes of varieties, then the point $[X]$ corresponding to the isomorphism class of $X$ is $A$-invariant (resp. $\mathfrak{\Gamma}$-invariant).
Moreover, if $[X]$ is seen as a closed point in $\mathcal{M}$, then the residue field $K([X])$ corresponding to the point $[X]$ is equal to the field of moduli \cite{MR0162799}.
\end{remark}

We remark that the family of isomorphisms $f_\sigma \colon  X \stackrel{\cong}{\longrightarrow}  {^\sigma \! X}$ in the definition of the group of moduli has no extra structure; that is, we do not require any relation between the isomorphisms 
$f_\sigma\colon X  \to {^\sigma \! X}$, 
$f_\tau \colon X \to  {^\tau \! X}$ 
and 
$f_{\sigma\tau} \colon X \to { ^{\sigma\tau} \!X}$.
Later, we will need this extra structure to define the group of definition.

\begin{remark}
  \label{rem:finiteExtArith}
Recall that a $K_s$-variety is defined over $L$ if its corresponding ideal $I(X)$ can be generated by a finite collection of polynomials with coefficients in $L$. 
Every variety is defined over a finite extension of $K$, since it is generated as the zero locus of a finite set of polynomials, and each one has a finite number of coefficients. The field generated by all these finite coefficients in $K_s$ is a finite extension of $K$.
\end{remark}

Recall from Proposition \ref{prop:invArith} that, a $K_s$-variety is defined over $L$ if and only if it is $\Gal(K_s/L)$-invariant.
Motivated by this, we give the following  definition: 
\begin{definition}
  For a $K_s$-variety $X$ and a closed subgroup $A \leq \mathfrak{\Gamma}$, we define the {\em relative group of invariance of $X$ over $A$} 
\[
  A_X^{\mathrm{inv}}= \{\sigma \in A : {^\sigma \! X} =X\}.
\]
and {\em the absolute group of invariance of $X$ (over $\mathfrak{\Gamma}$)} to be the group $\mathfrak{\Gamma}_X^{\mathrm{inv}}$.
\end{definition}

Equivalently, the group of invariance over $A$ is the largest subgroup $Z \leq A$ such that $X$ is defined over its fixed field $K_s ^{Z}$. It is closed by Remark \ref{rem:finiteExtArith}, hence $A_X^{\mathrm{inv}}$ corresponds to the minimal field extension of $K_s^{A}$ over which $X$ is defined, that is, $K_s ^{A_X^{\mathrm{inv}}}$. We have that $K \subseteq K_s^A \subseteq K_s ^{A_X^{\mathrm{inv}}} \subseteq K_s$. Moreover, assuming $A = \mathfrak{\Gamma}$, we can always think of a variety defined over the minimal field over $K$, in the above sense.

Obviously $A_X^{\mathrm{inv}} \leq A_X^\mathrm{mod}$ as subgroups of $A$ since we require that the isomorphisms $f\colon X \to {^\sigma \!X}$ are identities. Now we work our way towards defining what the group of definition (of $X$ over $A $) is, and which actually sits in between these two groups. We begin with the following definition to which the field of definition corresponds.

\begin{definition}{\label{def:fieldofdef}}
Let $X$ be an $L$-variety and let there be the following field extensions $K \subseteq L_0 \subseteq L  \subseteq K_s$. 
We say that the variety $X$ is
{\em definable} over $L_0$ with respect to the Galois extension $L/L_0$, if there is an algebraic variety $Y_0$ defined over $L_0$ and a birational isomorphism $R\colon X \rightarrow Y_0 
\times_{\Spe (L_0)} \Spe (L)$ defined over $L$. 

The smallest field $L_0$ such that $X$ is definable over $L_0$ is called the {\em absolute field of definition} of X.
\end{definition}

Note that here we assume that $X$ is defined over $L$ and not over $K_s$ as usual. We emphasize this point, because we will need the fact that the field extension $L/K$ has to be finite for the following construction to work - and this is holds always by Remark \ref{rem:finiteExtArith}.

Assume that $X$ is a variety defined over $L$, and that $X$ is also  definable over $L_0 \subseteq L$, then by definition  there is a birational  isomorphism $R\colon X \rightarrow Y $
defined over $L$, where $Y \coloneqq Y_0\times_{\Spe (L_0)} \Spe (L)$ and $Y_0$ is defined over $L_0$. There are birational isomorphisms 
${^\sigma\! R} \colon {^\sigma \!Y} \xrightarrow[]{} {^\sigma \! X} $,
such that the maps
\[
  \{ f_\sigma=({^\sigma\! R})^{-1}\circ
    (\mathrm{Id}_Y \times \sigma) \circ
  R \colon X \rightarrow {^\sigma \! X}\}_{ \sigma \in \Gal(L/L_0)}
\]
satisfy the condition: 
\[
  f_{ \sigma \tau }= { ^\sigma \!f}_\tau \circ 
  f_\sigma \text{ for all } \sigma, \tau \in \Gal(L/L_0). 
\]
Indeed, we have the diagram:

\begin{equation*}
\begin{tikzcd}
  && {^\tau \! Y}  \arrow[rr, "(^{\tau}R)^{-1}"] 
  & & {^{\tau}\! X} 
  \\
 {^{\sigma \tau} \! Y } \arrow[rrrrr,  bend left = 60, "({^{ \sigma \tau} \! R})^{-1}" description]
 &&  Y 
  \arrow[u,"\mathrm{Id}_Y \times \tau"] 
  \arrow[d,"\mathrm{Id}_Y \times \sigma"']
  \arrow[ll,"\mathrm{Id}_Y \times  \sigma \tau"']
  & X \arrow[l, "R" description] \arrow[rr, "f{\sigma \tau}" description] \arrow[ur, "f{\tau}" description] \arrow[dr, "f{\sigma}" description] & &  {^{ \sigma \tau}\! X} 
  \\
  && {^\sigma \! Y} \arrow[ull, bend left, "^{\sigma}(\mathrm{Id}_Y\times \tau)"] \arrow[rr, "(^{\sigma}R)^{-1}"]
  & & {^{\sigma}\! X} \arrow[ur, bend right =  20, "{ ^\sigma \!f}\tau"']
\end{tikzcd}
\end{equation*}

which is commutative by the following observation:
\begin{align*}
  f_{\sigma \tau }
  &= ({^{ \sigma \tau} \! R})^{-1} \circ
  (\mathrm{Id}_Y \times \sigma\tau) \circ  R
   \\ 
   &= ( ^{\sigma}\! ( {^\tau \!  R}) )^{-1} \circ (\mathrm{Id}_Y \times \sigma\tau) \circ R \\
   &= \ ^{\sigma} \! \left( (^{\tau}R)^{-1} \right) \circ(\mathrm{Id}_Y \times \sigma\tau) \circ R 
   \\   
   &=  
  \ ^{\sigma}\!\!\left( ( ^\tau \!R)^{-1})  \circ (\mathrm{Id}_Y \times \tau) \circ R \circ
   \big( (\mathrm{Id}_Y \times \tau) \circ R\big)^{-1} \right) \circ
   (\mathrm{Id}_Y \times \sigma\tau) \circ
   R 
  \\
   &= \  ^{\sigma} \!\big(f_{\tau} \circ R^{-1} \circ (\mathrm{Id}_Y \times \tau)^{-1} \big) \circ (\mathrm{Id}_Y \times \sigma\tau) \circ R 
   \\
   &=   {^{\sigma} \!f}_\tau  \circ  ({^{\sigma}}\!R)^{-1} \circ ({^{\sigma}}\!(\mathrm{Id}_Y \times \tau))^{-1} \circ (\mathrm{Id}_Y \times \sigma\tau) \circ R 
   \\
   & ={^{\sigma} \!f}_\tau f_\sigma. 
\end{align*} 
Notice that in the last equality we have used the identity 
\[
  (\mathrm{Id}_Y \times \sigma\tau)= {^\sigma \!(\mathrm{Id}_Y \times \tau) } \circ (\mathrm{Id}_Y \times \sigma).
\]
The family $\{f_\sigma\}_{\sigma \in \mathrm{Gal(L/L_0)}}$ is called a {\em Galois descent datum of $X$ with respect to} (the finite extension) $L/L_0$.
A. Weil in \cite{Weil56} proved the Galois Descent Theorem, which is inverse to the above construction.
\begin{theorem}[Weil's Descent Theorem]\label{thm:weil's descent}

Let $L/L_0$ be a Galois extension, with $K \subseteq L_0 \subseteq L \subseteq K_s$, and let $X$ be a variety defined over $L$.
\begin{enumerate}
  \item 
 If $X$ admits a Galois descent datum $\{f_\sigma\}_{\sigma \in \Gal(L/L_0)}$ with respect to $L/L_0$, then there exist an algebraic variety $Y$ defined over $L_0$ and an isomorphism $R:X \rightarrow Y$ defined over $L$, such that $R= {^\sigma \! R}\circ f_\sigma$, for every $\sigma \in \Gal(L/L_0)$. 
 \item If there exists another variety $\widehat{Y}$ defined over $L_0$ together with an isomorphism $\widehat{R}\colon X \rightarrow \widehat{Y}$ defined over $L$, such that $\widehat{R}= {^\sigma\!\widehat{R}} \circ f_\sigma$, for every $\sigma \in \Gal(L/L_0)$, then there exist an isomorphism $J:Y \rightarrow \widehat{Y}$ defined over $L_0$, such that $\widehat{R}=J \circ R$.
 \end{enumerate} 
\end{theorem}

We are ready to define the group of definition.

\begin{definition}
  \label{def:groupdefArith}
For a $K_s$-variety $X$ and a closed subgroup $A\leq \mathfrak{\Gamma}$, we define {\em the relative group of definition of $X$ over $A$} as 
\[
A_X^{\mathrm{def}}=\{\sigma \in A_X^{\mathrm{mod}}: \ \exists f_\sigma\colon X \xrightarrow{\simeq} {^\sigma\! X} 
\text{ with } f_{\sigma\tau}={^\sigma \! f_{\tau} \circ f_\sigma} \text{ for all } \sigma,\tau \in A_X^{\mathrm{mod}} \}.
\]
and {\em the absolute group of definition of $X$ (over $\mathfrak{\Gamma}$)} to be the group $\mathfrak{\Gamma}_X^{\mathrm{def}}$.
\end{definition}

The group of definition of $X$ over $A \leq \mathfrak{\Gamma}$ is the largest subgroup $Z\leq A_X^{\mathrm{mod}}$ such that there exists a family $\{ f_\sigma \}_{\sigma \in Z}$ of isomorphisms $f_\sigma\colon X \xrightarrow{\simeq} {^\sigma\! X} $, which satisfies the Weil descent condition. It is a closed subgroup by light adjustment of the proof \cite[Proposition 2.1]{DebesEmsalem}, where one would have to additionally consider that $f_{\sigma} = f_{\tilde{\sigma}}$ in their notation.

Notice that the group satisfying Weil's Galois descent datum is the quotient group $A_X^{\mathrm{def}} / A_X^{\mathrm{inv}}$, which is the Galois group of the extension $L/L_0$, for $L=K_s^{A_X^{\mathrm{inv}}}$ and $L_0=K_s^{A_X^{\mathrm{def}}}$, as in the notation of Weil's descent Theorem \ref{thm:weil's descent}. The quotient is well-defined since $A_X^{\mathrm{inv}}$ is always normal in $A$.

\begin{remark}
Closed subgroups of $A_X^{\mathrm{inv}}$ correspond to various fields $K_s^A \subseteq L $ over which $X$ is definable, while $A_X^{\mathrm{def}}$ corresponds to the minimal such field. For $A=\mathfrak{\Gamma}$, that is the absolute field  of definition of $X$, see Definition \ref{def:fieldofdef}.
\end{remark}

Observe that the groups we introduced fit in the following {\em ascending} tower of subgroups 
\[
\{1\} \leq A_X^{\mathrm{inv}} \leq A_X^{\mathrm{def}} \leq A_X^{\mathrm{mod}} \leq A \leq \mathfrak{\Gamma} = \Gal(K_s/K),
\]
while their respective fields of invariants satisfy the normal inclusion-reversing relation:
\[
K_s \supseteq K_s^{A_X^{\mathrm{inv}}} \supseteq K_s^{A_X^{\mathrm{def}}} \supseteq K_s^{A_X^{\mathrm{mod}}} \supseteq K_s^A \supseteq K
\]
It is worth noticing that, by Remark \ref{rem:finiteExtArith}, the extension $ K_s^{A_X^{\mathrm{inv}}} / K$ is finite and consequently every other intermediate extension. So Weil's theorem always applies the for extension $ K_s^{A_X^{\mathrm{inv}}} / K_s^{A_X^{\mathrm{def}}}$, i.e. for a variety defined over $K_s^{A_X^{\mathrm{inv}}}$ which is definable over $K_s^{A_X^{\mathrm{def}}}$.

Later in this section, we shall present two well-known examples: one in which the field of moduli $K_s^{\mathfrak{\Gamma}_X^{\mathrm{mod}}}$ coincides with the field of definition $K_s^{\mathfrak{\Gamma}_X^{\mathrm{def}}}$ (see Subsection~\ref{exm: moduli = definition}), and another in which the two fields differ (see Subsection~\ref{exm:field of moduli vs field of definition}).

\subsection{Topological}\label{sec:Top_Weil}

As we mentioned in the prologue of this section, we seek to introduce a notion in the topological setting that is analogous to the arithmetic ``group of moduli'', ``group of invariance'' and ``group of definition''. A fundamental difficulty that arises is the absence of a canonical analogue for the notion of a field and especially its separable closure. In Section \ref{sec:covers} we remedied this by thinking of the extension $K_s/K$ as the extension $K_s(B)/K(B)$, so that $\Gal(K_s/K)$ was replaced by the mapping class group $\mathrm{Mod}(B^*)$ of a base space $B$ minus the branch points. Our aim with this approach is to provide a reasonable topological analogue for curves of Weil's descent Theorem \ref{thm:weil's descent} from the arithmetic setting. We propose that this is Theorem \ref{thm:Weil topological} and we argue that it correctly matches our intuition of Weil's theorem on the moduli space, as it can be realized inside the broader setting of the Teichmüller space.

As our setting for this section, we restrict attention to covers of hyperbolic curves and we work under one of the following assumptions: either the cover $X\rightarrow (B,A)$ is Galois, or there is a section in the exact sequence (\ref{ses:Pi_Mod}). Both assumptions ensure that ${^\sigma \! X}$ is well-defined for all $\sigma \in \mathrm{Mod}(B^*)$. That is, ${^\sigma \! X}$ denotes the compactified algebraic curve over $\mathbb{C}$ corresponding to $\sigma_*(R)$, where $X\rightarrow B$ corresponds to the subgroup $R$. 

Unless stated otherwise, we restrict to the Galois case, i.e. $R$ is a normal subgroup. Throughout this section, we mainly deal with the case $\sigma_*(R) \neq R$ as subgroups of $\pi_1(B^*,b_0)$, even though they can be isomorphic as abstract groups. The study of group ascent on the mere subcovers of $X$ which satisfy $\sigma_*(R) = R$ will be undertaken in Chapter \ref{ch:descend_on_monodromy}. 

We briefly recall some details about the topological constructions. The subgroup $R$ of $\pi_1(B^*,b_0)$ corresponds to a cover $p\colon X^\circ \rightarrow B^*$ together with a choice of a basepoint $x_0$ of $X^\circ$ such that $p(x_0) = b_0$ and $R=p_*(\pi_1(X^\circ,x_0))$. 
By the Galois correspondence, $\sigma_*(R)$ corresponds to a cover ${^\sigma p} \colon {^\sigma \! X}^\circ \rightarrow B^*$ and a choice of a basepoint ${^\sigma \! x_0}$, such that ${^\sigma \! p}({^\sigma \! x_0}) = b_0$. 
Furthermore, by the lifting property \cite[13.5]{MR1343250}, for any lift $\hat{\sigma}$ of $\sigma \in \mathrm{Mod}(B^*)$ such that $\hat{\sigma}(b_0) = b_0$, there exists a unique lift $\tilde{\sigma} \colon X^\circ \rightarrow {^\sigma \! X}^\circ$ that preserves the basepoint choices and makes the following diagram commute:

\begin{equation} \label{eq:sigma_tilde_top}
    \begin{tikzcd}
X^\circ \arrow[r, "\tilde{\sigma}", dashed] \arrow[d, "p"'] & {^\sigma \! X}^\circ \arrow[d, "{^\sigma \! p}"] \\
B^* \arrow[r, "\hat{\sigma}"] & B^*  
\end{tikzcd}
\end{equation}
That is, $\tilde{\sigma}(x_0) = {^\sigma \! x_0}$. The lift $\tilde{\sigma}$ is unique up to the choice of basepoint: there may exist homeomorphisms $f_\sigma \colon X^\circ \rightarrow {^\sigma \! X}^\circ $ for which the above diagram commutes, but with $f_\sigma(x_0) \neq {^\sigma x_0}$.

To define an action of $\mathrm{Mod}(B^*)$ on Galois covers and their morphisms, we assume that exists a set-theoretic section $ \hat{} : \mathrm{Mod}(B^*) \rightarrow \mathrm{Homeo}_+(B^*)$ that preserves the basepoint $b_0$. Each element $\sigma\mapsto \hat\sigma$ admits a unique lift to a homeomorphism $\tilde\sigma \colon X^\circ \rightarrow {^\sigma \! X}^\circ$. We formulate this set-theoretic section as generally as possible so as to accommodate various geometric settings: for instance, these homeomorphisms could be considered to be geodesic loops of $b_0$ around handles, or have some specific interactions with the markings inside a Teichmüller space \cite[Ch.1]{MR2006093}.

An action of $\sigma \in \mathrm{Mod}(B^*)$ means that we have the following commutative diagrams for a cover $X \to (B,A)$ and corresponding group $R$:
\begin{equation*}
    \begin{tikzcd}
        X^\circ  \arrow[d]\arrow[r, "\tilde{\sigma}"] & ^\sigma \! X^{\circ} \arrow[d] & &
        X  \arrow[d]\arrow[r, "\tilde{\sigma}"] & ^\sigma \! X \arrow[d]\\
        (B^*,A) \arrow[r, "\hat{\sigma}"] & (B^*, {\sigma  A \sigma^{-1}}) & & (B,A) \arrow[r, "\hat{\sigma}"] & (B, {\sigma  A \sigma^{-1}})
    \end{tikzcd}
\end{equation*}
where the left diagram corresponds to the open cover and the right to the compactified one. 
Moreover, $X^\circ  = \widetilde{B^*}/R$ and $ Y^{\circ} = \widetilde{B^*}/\sigma_*(R) $. 

We define the category of covers over $B$ defined over $A$ and work over this category to distinguish the new structure of morphisms of covers needed in this section. To that end we provide the following definition:

\begin{definition}\label{def:category of covers}
Let $\mathsf{Cov}_{(B,A)}$ denote the category whose objects are Galois covers $X\rightarrow B$ with common branched divisor in $B$, that is the underlying topological space of the associated topological cover is fixed, namely it is $B^*$, and the cover is defined over $A$ (recall Definition~\ref{def:topological_defined}). The morpshisms of $\mathsf{Cov}_{(B,A)}$ are of the form $(f,g) \colon \big( X \to (B,A) \big) \xrightarrow[]{} \big( Y \to (B,A) \big)$ where $f$ and $g$ are orientation-preserving continuous maps rendering the following diagram commutative:
\begin{equation*}
    \begin{tikzcd}
        X \arrow[r, "f"] \arrow[d] & Y \arrow[d] \\
        (B,A) \arrow[r, "g"'] & (B,A)
    \end{tikzcd}
\end{equation*}

In particular for $A=1$, we denote this category by $\mathsf{Cov}_{B}$ and its morphisms are pairs of continuous maps $(f,g)\colon (X \to B) \to (Y \to B)$ that make the following diagram commute:
\begin{equation*}
    \begin{tikzcd}
        X \arrow[r, "f"] \arrow[d] & Y \arrow[d] \\
        B \arrow[r, "g"'] & B
    \end{tikzcd}
\end{equation*}
\end{definition}

The idea is that we do not require the bottom arrow to be identity, as required for the deck transformations of a cover. Two covers are {\em isomorphic} if both $f$ and $g$ are homeomorphisms for which we use the symbol “$\simeq $".

We note that $\mathsf{Cov}_{(B,A)}$ is a full subcategory of $\mathsf{Cov}_B$ since the morphisms between objects of the former are precisely the morphisms of the latter between the same objects, by forgetting the “defined over $A$" structure. In particular, the only difference of between these two categories is that we require objects to be defined over a non trivial subgroup of $\mathrm{Mod}(B^*)$.

To stay in parallel with the arithmetic case of descent as in Section \ref{sec:FieldMODDefinition}, where we work over $\Spe(K_s)$, we assume for the rest of the section that $A=1$, which is the correct analogue. However, we will revisit the category $\mathsf{Cov}_{(B,A)}$ in Section \ref{sec:ActionOnCatArith}.

An element of $\sigma \in \mathrm{Mod}(B^*)$ yields the following diagram:
\begin{equation*}
    \begin{tikzcd}
        X \arrow[r, "f"] \arrow[rrr, bend left=40, "\tilde{\sigma}"] \arrow[d] & Y \arrow[d] \arrow[rrr, bend left=40, "\tilde{\sigma}"] & & {^\sigma \! X} \arrow[r, "{^\sigma \!f}"] \arrow[d] & {^\sigma \! Y} \arrow[d]\\
        B \arrow[r, "g"'] \arrow[rrr, bend right=40, "\hat{\sigma}"']& B \arrow[rrr, bend right=40, "\hat{\sigma}"'] & & B \arrow[r, "{^\sigma \!g}"'] & B
    \end{tikzcd}
\end{equation*}
where ${^\sigma \! f} = \tilde{\sigma} f \tilde{\sigma}^{-1}$ and ${^\sigma \! g} = \hat{\sigma} g \hat{\sigma}^{-1}$ which render the right square commutative. In particular, the pair $({^\sigma \! f} , {^\sigma \! g})$ is a morphism of covers in $\mathsf{Cov}_B$. Moreover, this yields a bijection of morphisms $(f,g)$ and twisted morphisms $({^\sigma \! f} , {^\sigma \! g})$.
Observe that for an isomorphism $(f,g) \colon (X \to B) \to ({^\sigma \! X} \to B)$, i.e. a pair of homeomorphisms $(f,g)$ then $g = \hat{\sigma}$ since the section $ \hat{} : \mathrm{Mod}(B^*) \rightarrow \mathrm{Homeo}_+(B^*)$ is an epimoprhism. 

We now formulate the topological analogue of Weil's descent datum and the corresponding topological groups.

\begin{definition}
Let $A \leq \mathrm{Mod}(B^*) = \mathfrak{\Gamma}$ be a subgroup and let $(f_\sigma , \hat{\sigma})_{\sigma \in A}$ be a family of pairs of homeomorphisms $(f_\sigma, \hat{\sigma}) \colon (X \to B) \to ( {^\sigma \! X} \to B)$ is a {\em Mapping class group ascent datum relative to $A$} if the following diagram commutes: 

\begin{equation*}
    \begin{tikzcd}
        X \arrow[rr, bend left = 30, "f_{\sigma \tau}"] \arrow[r, "f_{\sigma}"] \arrow[d] & {^\sigma \! X} \arrow[r, "{^\sigma \! f_{\tau}}"] \arrow[d] & {^\tau \! ({^\sigma \! X})} \arrow[d] \\
        B \arrow[rr, bend right = 30, "\widehat{\sigma \tau}"']\arrow[r, "\hat{\sigma}"] & B \arrow[r, "{^\sigma \! \hat{\tau}}"] & B
    \end{tikzcd}
\end{equation*}
We write $(f_{\sigma \tau} , \widehat{\sigma \tau}) = ({^\sigma f_\tau} , {^\sigma \hat{\tau}}) \circ (f_\sigma , \hat{\sigma}) $ for $\sigma, \tau \in A$ when the above diagram commutes, which we call the cocycle condition for $\sigma$ and $\tau$. 
We also define the following groups:
\begin{align*}
    A^{\mathrm{mod}}_X &= \{ \sigma \in A \  \colon \exists (f_{\sigma}, \hat{\sigma}) \colon (X \to B) \xrightarrow[]{\simeq}   ({^\sigma \! X}  \to B) \}, \\
    A_X^{\mathrm{def}} &= \{\sigma \in A \ : \ \exists (f_{\sigma}, \hat{\sigma}) \colon (X \to B) \xrightarrow[]{\simeq}   ({^\sigma \! X}  \to B) , \text{ with } (f_{\sigma \tau} , \widehat{\sigma \tau}) = ({^\sigma f_\tau} , {^\sigma \hat{\tau}}) \circ (f_\sigma , \hat{\sigma}) \text{ for all } \sigma,\tau \in A \}, \\
    A^{\mathrm{inv}}_X &= \{ \sigma \in A  \ : \  {^\sigma \! X}= X \}    
\end{align*} corresponding respectively to the relative groups of moduli, definition and invariance. 
\end{definition}

\begin{remark}
It is obvious that $A^{\mathrm{inv}}_X \leq A_X^{\mathrm{def}}\leq A^{\mathrm{mod}}_X \leq A$. Note also that $\mathfrak{\Gamma}^{\mathrm{inv}}_X$ is the largest subgroup of $\mathfrak{\Gamma}$ over which the cover $X\to B$ is defined. By the previous discussion, we also observe the following equalities:
\begin{align*}
    A^{\mathrm{mod}}_X &= \{ \sigma \in A \  : \ \sigma_*(R) \simeq R \  \textrm{ (as abstract groups)}\} = A , \\
    A^{\mathrm{inv}}_X &= \{ \sigma \in A \ : \  \sigma_*(R) = R \}    
\end{align*}
where the first equality always holds by definition of the action since the group isomorphism $\sigma_*(R) \simeq R$ is equivalent to $\hat{\sigma}$ being a homeomoprhism and thus the lift $\tilde{\sigma}$ is also a homeomorphism. In other words, in the topological setting, the group of moduli $A^{\mathrm{mod}}_X$ is the whole group and thus not very interesting. 
However, this is not true for the group of definition $A_X^{\mathrm{def}}$. Indeed, the condition $\widehat{\sigma \tau} =  {^\sigma \! \hat{\tau}} \hat{\sigma}$ does not a priori hold for any choice of section $ \hat{} : \mathrm{Mod}(B^*) \rightarrow \mathrm{Homeo}_+(B^*)$. From our conventions, this condition amounts to the nmap $\hat{}$\ being a homomorphism on $\sigma,\tau$. Recall also that the Birman exact sequence (\ref{ses:Birman}) does not split in general, thus we are making “fair" assumptions so far. Therefore, a Weil descent datum need not (1) necessarily exist and (2) have $f_\sigma = \tilde{\sigma}$ as the only choice, since there might exist different pairs of homeomorphisms $(f_\sigma, \hat{\sigma} ) \colon (X \to B) \to ({^\sigma\! X} \to B)$, as we explained in diagram (\ref{eq:sigma_tilde_top}), satisfying the ascent/descent relation without preserving the basepoints.
\end{remark}

We give the following definition, which is the topological analogue of Definition \ref{def:fieldofdef}. Recall that $\mathrm{N}_{\mathrm{Mod}(B^*)}A$ denotes the normalizer of $A$ in $\mathrm{Mod}(B^*)$.

\begin{definition}\label{def:definable_top}
    Let $X \to (B,A)$ be a cover defined over $A$ and let there be a subgroup $A' \leq \mathrm{Mod}(B^*)$ such that $A \lhd A'$. We say that the cover $X \to (B,A)$ is {\em definable} over $A'$ with respect to $A'/A$ if there is a cover $Y \to (B,A')$ and an element $\sigma \in \mathrm{N}_{\mathrm{Mod}(B^*)}A$ such that $\sigma_*(R_X) = R_Y$, where $R_X,R_Y$ are the subgroups of $\pi_1(B^*,b_0)$ corresponding to the topological covers $X^\circ,Y^\circ$ respectively.

    The largest such group $A'$ is called the {\em absolute group of definition of $X$}. 
\end{definition}

Note that, the normalizer condition is the topological equivalent of being $A$-invariant. Then, $\sigma_*(R_X) = R_Y$ implies that $\sigma $ induces an isomorphism between the covers $X\to B$ and $Y\rightarrow B$.
We are ready to formulate Weil's ascent theorem for the topological case:

\begin{theorem} \label{thm:Weil topological}
    Let $A \lhd_{\mathrm{f}} A^{\prime}  \leq \mathrm{Mod}(B^*) $ and let $X \to (B,A)$ be a cover defined over $A$. Then $X$ is definable over $A^\prime$ if and only if $X\to B$ admits a mapping class group ascent datum $\{ f_\sigma \} _{\sigma \in A^\prime /A }$ with respect to the finite group quotient $A^\prime /A $.
\end{theorem}

\begin{proof}
    First, note that the cocycle condition $(f_{\sigma \tau} , \widehat{\sigma \tau}) = ({^\sigma f_\tau} , {^\sigma \hat{\tau}}) \circ (f_\sigma , \hat{\sigma})$ is equivalent to $\widehat{\sigma\tau} = \hat \sigma \hat \tau$. Indeed, this follows from the definitions on the twist by $\sigma$ using conjugation; this is exactly what the second coordinate of the cocycle implies. Then, there are many compatible choices of $f_\sigma$ and $f_\tau$ as we explained in the discussion after diagram \ref{eq:sigma_tilde_top}. 
    Thus, for a suitable choice of (finite) elements indexed by $A^\prime/A$, we have that the section $ \hat{} : \mathrm{Mod}(B^*) \rightarrow \mathrm{Homeo}_+(B^*)$ behaves like) a homomorphism. Therefore, we treat the mapping class group ascent datum accordingly.

    Given a mapping class group ascent datum indexed by the finite $A^\prime/A$, write $A^\prime = \coprod_{i=1}^n \sigma_i A$. Then the cocycle condition on the elements $\hat\sigma_i$ implies that $\coprod_{i=1}^n \hat{\sigma_i}H$ is a well-defined group and that each $\hat{\sigma_i}$ normalizes $H$. There is a natural projection $\coprod_{i=1}^n \hat{\sigma}_i H \rightarrow  \coprod_{i=1}^n \sigma_i A$ and assume $R^\prime$ to be its kernel. We have the following commutative diagram of exact sequences of groups:
\begin{equation*}
    \begin{tikzcd}
1 \arrow[r] & R \arrow[r] \arrow[d] & H \arrow[r] \arrow[d, hook]                       & A \arrow[r] \arrow[d, hook]                 & 1 \\
1 \arrow[r] & R^\prime \arrow[r]    & \coprod\limits_{i=1}^n \hat{\sigma_i} H \arrow[r] & \coprod\limits_{i=1}^n \sigma_i A \arrow[r] & 1,
\end{tikzcd} 
\end{equation*} which implies that $R=R^\prime$. This topological construction asserts that $Y = X$ in the definition \ref{def:definable_top}. 

Conversely, given any exact sequence $1\rightarrow \sigma_*(R) \rightarrow H^\prime \rightarrow A^\prime \rightarrow 1$, write $A^\prime = \coprod_{i=1}^ng_i A$ and since $R$ is contained in $H$, $\sigma_*(R)$ is contained in $\hat{\sigma}H\hat{\sigma}^{-1}$. We can thus write $H^\prime = \coprod_{i=1}^n \hat{g_i} \hat{\sigma} H \hat{\sigma}^{-1} =  \coprod_{i=1}^n \hat{g_i} H$ since $\sigma$ normalizes $A$ and thus $\hat{\sigma}$ normalizes $H$. We will prove now that the finite elements $g_i$ provide a proper mapping class group ascent datum with respect to $A^\prime/A$. Observe that $\hat{g_i}\hat{g_j}\widehat{g_ig_j}^{-1}$ is in $H$ since its image is in $A$ regarding the map $H^\prime/H \rightarrow A^\prime/A$. However, this map is injective from the exact sequence below
\[1\rightarrow \sigma_*(R)/\hat{\sigma}R\hat{\sigma}^{-1} \rightarrow H^\prime/\hat{\sigma}H\hat{\sigma}^{-1} \rightarrow A^\prime/ \sigma A \sigma^{-1}\rightarrow 1, \]
since $\sigma_*(R) = \hat{\sigma}R\hat{\sigma}^{-1}, \hat{\sigma}H\hat{\sigma}^{-1} = H, \sigma A \sigma^{-1} = A$. Thus $\hat{g_i}\hat{g_j}=\widehat{g_ig_j}$. 
\end{proof}

\begin{remark}
In the topological analogue of Weil’s descent theorem, the curve $X$ itself ascends to a curve defined over $A^\prime$. In contrast with the arithmetic situation, the underlying model does not change: the cover $X \to (B,A)$ ascends directly to $X \to (B,A^\prime)$. This phenomenon may be traced to Definition~\ref{def:topological_defined}, where we require each class of automorphism preserves the conjugacy class of $R$ in $\pi(B^*, b_0)$. If this condition could be relaxed in a meaningful way, one might instead obtain a distinct cover $Y \to (B, A^\prime)$ which in some sense, after restriction of groups, is isomorphic to the cover $X \to (B,A)$.
\end{remark}

We now reinterpret the preceding result in terms of Teichmüller theory in order to provide additional intuition. For background on the material reviewed below, we refer to \cite[Ch.~1]{MR2006093} and \cite[Pt.~2]{FarbMagalit}. Throughout, we fix an integer $g \geq 2$. Let $S_g$ be a reference surface of genus $g$ and let $m_X \colon S_g \rightarrow X$ denote denote a marking of the cover $X \to B$, that is, a diffeomorphism.
The Teichmüller space $T(S_g)$ is the set of marked Riemann surfaces $(X,m_X)$ modulo the equivalence relation that identifies two marked surfaces whenever there exists a biholomorphism between them that respects the markings.

A classical result in the theory of Teichmüller spaces asserts that $\mathrm{Mod}(S_g)$ acts on $T(S_g)$ via 
\[
\sigma \cdot   [(X,m_X)]  = [(X,m_X \tilde{\sigma}^{-1})], 
\] 
for any homeomorphisms $\tilde{\sigma}$ that lifts $\sigma$. The equivalence relation ensures that the action well-defined and independent of $\tilde{\sigma}$. Moreover, a second classical result identifies the quotient
\[
T(S_g)/\mathrm{Mod}(S_g) \cong M_g,
\]

where $M_g$ denotes the moduli space of curves of genus $g$. Observe that this action on $T(S_g)$ is always trivially fiber-preserving, that is the curve $X$ is never replaced by some conjugate curve $Y$.

However, that is not the case for the action via $\mathrm{Mod}(B^*)$ for which the presence of the Galois cover $X\rightarrow B$ plays a central role. On the one hand, $\sigma \in\mathrm{Mod}(B^*)$ maps $X$ to ${^\sigma \! X}$. On the other hand, a mapping class $\sigma$ of $B^*$ lifts to a mapping class of $\mathrm{Mod}(S_g)$ via the Birman-Hilden property (see Chapter \ref{ch:Hilden}) up to isotopy classes of deck transformations of the cover. Conveniently, this deals with the distinction between mapping classes of the punctured surface $B^*$ and the closed unpunctured surface $S_g$. 

The ambiguity regarding classes of deck transformations is then resolved by our choice of section \linebreak 
$ \hat{} : \mathrm{Mod}(B^*) \rightarrow \mathrm{Homeo}_+(B^*)$ which allows one to lift the quotient $\mathrm{LMod}(B^*)$ to $\mathrm{SMod}(S_g)$, as discussed in Chapter~\ref{ch:Hilden}.
Consequently, a mapping class $\sigma = [\hat{\sigma}]$ of $\mathrm{Mod}(B^*)$ lifts to a mapping class $[\tilde{\sigma}]$ of $\mathrm{Mod}(S_g)$. Recall also the induced homeomorphism $\tilde{\sigma}_X \colon X\rightarrow {^\sigma X}$. Then, the action of $\mathrm{Mod}(B^*)$ on $T(S_g) $ is given by
\[
\sigma \cdot   [(X,m_X)]  = [({^\sigma \! X}, \tilde{\sigma}_X m_X \tilde{\sigma}^{-1})]. 
\]

This defines a genuine group action, even though it need not satisfy $\widetilde{\sigma \tau} = \tilde\sigma \tilde\tau$ strictly; the relevant maps are homotopic, and thus no ambiguity arises at the level of Teichmüller space.

In this scenario the curve $X$ does not necessarily remain fixed under the action of $\mathrm{Mod}(B^*)$ as a point in the moduli space $M_g$. However, given that this point remains fixed under the action of a subgroup $A\subseteq \mathrm{Mod}(B^*)$, then the mapping class group ascent datum with respect to $A^\prime/A$ precisely guarantees that the point represented by $X$ is also fixed under the action of $A^\prime$. This interpretation aligns closely with the arithmetic intuition underlying fields of moduli and fields of definition.

\subsection{Common Notation}

Note the following similarity. In the arithmetic case, given a variety $X$ defined over $L$, then for $A= \Gal(K_s/L)$ we have that $A_X^{\mathrm{inv}} = A$. Similarly, for a cover $X \to (B,A) $, we have that $A_X^{\mathrm{inv}}=A$. In both variants, it is meaningful to pick a group greater than $A$ in order to descent/ascent, otherwise we would obtain trivial results. For instance, in arithmetic case, if $X$ is defined over $L$ and if $A \leq \Gal(K_s/L)$ closed subgroup, then by definitions we view $X$ as $K_s$-variety and then can always descent to a subfield $K_s^A$ over which $X$ is defined. In particular, this variety is the extension of scalars of $X$, i.e. $X \times _{K_s^A} L  $. Of course, if we use $A=\mathfrak{\Gamma}$ then this always makes sense.

\begin{remark}
    One difference is that, for the arithmetic case, we require that the group $A$, and consequently its subgroups $A_X^{\mathrm{mod}}, A_X^{\mathrm{def}}, \text{\ and \ } A_X^{\mathrm{inv}}$, be closed, since we needed that each corresponds to an appropriate fixed field. In general these definitions make sense for any subgroup, but then one would have to take its closure to obtain some fixed field.
    
\end{remark}

\pagebreak

{\small
\begin{longtable}[c]{>{\raggedright\arraybackslash}p{5.5cm}  >{\raggedright\arraybackslash}p{4cm}  >{\raggedright\arraybackslash}p{4.5cm}}
\rowcolor{headerblue}
\textbf{Number Theory} & \textbf{Topology} & \textbf{Common Notation} \\
\endfirsthead

\rowcolor{headerblue}
\textbf{Number Theory} & \textbf{Topology} & \textbf{Common Notation} \\
\endhead

$\Gal(K_s/K)$ & $\mathrm{Mod}(B^*)$ 
    & $\mathfrak{\Gamma}$  \\

\rowcolor{rowgray}
\begin{tabular}{l}
      closed subgroup $A \leq \Gal(K_s/K)$ \\ 
      fixed field $K_s^A$
\end{tabular}  & $A \leq \mathrm{Mod}(B^*)$ & $ A \leq \mathfrak{\Gamma}$  \\
acts on $X \to \Spe(K_s) $ & acts on $X \to (B,\{ 1 \} )$ & \\

\rowcolor{rowgray}
$A_X^{\mathrm{mod}} \leftrightsquigarrow K_s^{A} \subseteq K_s^{A_X^{\mathrm{mod}}} $   & $A_X^{\mathrm{mod}}$ & $\{\sigma \in A : \exists f_{\sigma} \colon X \xrightarrow[]{\simeq}{^{\sigma}\!X} \}$
  \\ 

$A_X^{\mathrm{def}} \leftrightsquigarrow K_s^{A} \subseteq K_s^{A_X^{\mathrm{def}}}$ minimal \linebreak over which $X$ is definable &  $A_X^{\mathrm{def}}$
&
$\{ \sigma \in A_X^{\mathrm{mod}} :$ the family  $f_\sigma$ is a  descent datum$\}$
\\

\rowcolor{rowgray}
$A_X^{\mathrm{inv}} \leftrightsquigarrow K_s^{A} \subseteq K_s^{A_X^{\mathrm{inv}}}$ minimal \linebreak over which $X$ is defined & $A_X^{\mathrm{inv}}$ & $\{ \sigma \in A : {^\sigma \! X} = X \}$ \\

\end{longtable}
}

\subsection{Example of Field of Moduli $=$ Field of Definition}\label{exm: moduli = definition}

Let us consider the case of elliptic curves. We know from \cite[III.1]{MR2514094} that over the algebraically closed field $\bar{k}$ two elliptic curves 
\[
  E:y^2 =x^3+ Ax+ B, \text{ and } E':y^2= x^3 + A'x +B'
\]
are isomorphic if and only if their $j$-invariants are the same that is
\[
  \frac{1728}{16} \frac{ (4A)^3}{4A^3+27 B^2} =  \frac{1728}{16}\frac{ (4A')^3}{4A'^3+27 B'^2}.
\]
The above equation implies that if $AB\neq 0 \neq A' B'$ then the  following map is an isomorphism 
\begin{align*}
f \colon E' & \longrightarrow E \\
(x',y') & \longmapsto (x,y)=(u^2x', u^3 y')
\end{align*}
where $u^4 A=A'$, $u^6 B=B'$ (assuming that $j \neq 0, 1728$).

Let us now consider the elliptic curve  
\[
  E: y^2 = x^3 + \sqrt[3]{2} x + i,
\]
which is initially defined over $K=\mathbb{Q}(\sqrt[3]2,i)$.

The Galois group $\Gal(K/\mathbb{Q})$ is isomorphic to a dihedral group $D_6$ of order $6$ generated by the elements $\sigma$ of order $3$ and $\tau$ of order $2$
so that 
\[
  \sigma(\sqrt[3]{2}) = \zeta_3 \sqrt[3]{2}, \sigma(i)=i, \tau(\sqrt[3]{2})=\sqrt[3]{2}, \tau(i)=-i.
\]
We have the following $6$ isomorphic models of the elliptic curve $E$:
\begin{align*}
E: \ y^2 &= x^3 + (\sqrt[3]{2}) \cdot x + i,\\
^{\sigma}\!E: \ y^2 &=   x^3 + (\zeta_3\sqrt[3]{2}) \cdot x + i,\\
^{\sigma^2}\!E: \  y^2 &= x^3 + (\zeta_3^2 \sqrt[3]{2}) \cdot x + i,\\
^{\tau}\!E: \ y^2 &= x^3 + (\sqrt[3]{2}) \cdot x - i,\\
^{\tau\sigma}\!E: \ y^2 &= x^3 +  (\zeta_3^2 \sqrt[3]{2}) \cdot x - i,\\
^{\tau\sigma^2}\!E: \ y^2 &= x^3 +  (\zeta_3 \sqrt[3]{2}) \cdot x - i,
\end{align*}
The isomorphism $f_g \colon E \rightarrow ^{g}\!\!E$ for $g\in D_6$ are given by:
\begin{align*}
  E & \longrightarrow ^{g}\!\!E \\
  (x,y) & \longmapsto (u_g^3 x, u_g^2 y),
\end{align*}
where
\[
u_e=1, \quad
  u_\sigma=\zeta_6^5, \quad 
  u_{\sigma^2}=\zeta_6, \quad 
  u_\tau=\zeta_{12}^3, \quad 
  u_{\tau\sigma}=\zeta_{12}^5,  \quad 
  u_{\tau\sigma^2}=\zeta_{12}
\]
We can define an action of $\Gal(\bar{\mathbb{Q}}/ \mathbb{Q})$ on $E$. Notice, that the group 
$H=\Gal(\bar{\mathbb{Q}}/K)$ acts trivially on the model of $E$, that is $^h \! E=E$ for all $h\in H$. Therefore, there is a well-defined action of $h$ on points of $E$, i.e. 
\begin{align*}
\tilde{h}\colon E & \longrightarrow E \\
P & \longmapsto \tilde{h}P
\end{align*}
A general element $\gamma \in \Gal(\bar{\mathbb{Q}}/\mathbb{Q})$, falls in one of the six cosets
\[
  H, \sigma H, \sigma^2 H, \tau H, \tau \sigma H, \tau \sigma^2 H
\]
and defines an action on $E$ as follows: Write $\gamma= g \cdot h$ for some $h\in H$ and some $g\in D_6$ and define the action on $E$ by letting $g$ acting on $E$ by $\tilde{g} f_g$ so the action of $\gamma$
is given by  
\[
\xymatrix@R6\partial {
E \ar[r]^{f_\gamma}    &
^{\gamma} \! E  
\ar[r]^{\tilde{\gamma}}
&
E \\
(x,y) \ar@{|->}[r] & 
(u_\gamma^3 x, u_\gamma^2 y)
\ar@{|->}[r] 
&
(
  \gamma(u_\gamma^3 x), \gamma(u_\gamma^2 y)
)
}
\]
\begin{remark}
Notice that, since $E$ has $j$-invariant $j(E)=-13824/19$, the curve is isomorphic over $\bar{\mathbb{Q}}$ to the curve with equation
\[
 E'\colon \ 
 y^2 + xy =x^3 - \frac{36}{j(E)-1728} \cdot x - \frac{1}{j(E) - 1728} 
\]
which has short Weierstrass form
\[
 E''\colon  \ y^2 = x^3 -1934917632/361 \cdot x + 60183678025728/6859 
\]
The isomorphism between $E,E''$ is given by
\[
(x,y) \longmapsto 
  \left(\frac{x}{-\frac{19}{62208} i \cdot 2^{\frac{2}{3}}}, \frac{y}{-\frac{1}{1492992} \, \left(-\frac{19}{3} i\right)^{\frac{3}{2}}}\right)
\]
\end{remark}

\subsection{Example of Field of Moduli $\neq$ Field of Definition}\label{exm:field of moduli vs field of definition}

This example is due to the third author, see \cite{Kont-Bord}. Let $K=\mathbb R$ and let $c$ denote complex conjugation in $\Gal(\mathbb C/ \mathbb R)$. Assume $n,m,q\in \Z_{>1}$ with $m$ odd, $2q<2mn$ and $q\mid 2m$. Let $X$ be the curve defined by
\[
X/\mathbb{C}\colon \quad y^q = \prod\limits_{1\leq i \leq m} \left(x^n-a_i\right)\left(x^n+\frac{1}{c(a_i)}\right),
\]
where $a_i = (1+i)\zeta_m^i$ and $\zeta_m$ denotes a primitive $m$-th root of unity. The curve $X$ is a branched cover of the projective line, by projection to the $x$-coordinate, with $\Gal(\mathbb{C}(X)/\mathbb{R}(\mathbb{P}^1))$ being cyclic of order $q$. The conjugate curve 
\[
{^c \! X}/\mathbb{C}\colon \quad y^q = \prod\limits_{1\leq i \leq m} \left(x^n-c(a_i)\right)\left(x^n+\frac{1}{a_i}\right),
\]
is isomorphic to $X$, via $\mu : X\rightarrow {^c \! X}$ which is defined by
\[
\mu (x,y) = \left(\frac{1}{\omega x}, \frac{\omega^\prime y}{x^{2mn/q}}\right),
\]
where $\omega^n = -1$ and $(\omega^\prime)^q = -1$. This implies that the field of moduli of $X$ is $\mathbb{R}$. 

The automorphism group $\mathrm{Aut}(X)$ of the curve $X$ is the product of cyclic groups $C_q\times C_n$ generated respectively by $\gamma(x,y) = (x,\zeta_q y)$ and $\nu(x,y)=(\zeta_nx,y)$, where $C_q \cong \Gal(\mathbb{C}(X)/\mathbb{R}(\mathbb{P}^1))$. 
Therefore, any isomorphism $f_c \colon X \rightarrow {^c \! X}$ is given by $f_c = \mu \gamma^i \nu^j$ for some $0 \leq i < q$ and $0\leq j < n$. For any such $f_c$, we have ${^c \!f}_c \circ f_c \neq 1$, by \cite[Prop. 5.1]{Kont-Bord}. Thus, Weil's cocycle condition fails and consequently $X$ does not admit an $\mathbb{R}$-model.

\section{Regularity}\label{sec:regularity}
In this section, we provide a group-theoretic description of the regularity condition on the covers. We will work using the common notation of the previous sections. In particular, let $\mathfrak{\Gamma}$ denote either $\Gal(K_s/K)$ or $\mathrm{Mod}(B^*)$. 
Let $X\rightarrow (B,A)$ denote a cover either of varieties or of curves (i.e. such that $X^\circ \rightarrow B^*$ are hyperbolic) defined over the (closed) subgroup $A \leq \mathfrak{\Gamma}$. Denote also by $\Pi_1$ either of the groups $\Pi_{K_s}$ or $\pi_1(B^*,b_0)$. In the arithmetic case, $\Pi_A$ represents $\Pi_{K_s^A}(B^*)$. We have the following proposition which is formulated for both cases.

\begin{proposition}
\label{prop:equalityofindices}
If the cover $X \to (B,A)$ is regular, then \[[\Pi_A:H] = [\Pi_1:R].\]
\end{proposition}

\begin{proof}

We consider the following short exact sequences of groups

\[
\begin{tikzcd}[column sep=normal]
1 \ar[r] & \Pi_1 \ar[r,"i"] & \Pi_A \ar[r,"p_1"] & A \ar[r] & 1 \\
1 \ar[r] & R \ar[r,"i"] \ar[hookrightarrow,u] & H \ar[r,"p_2"] \ar[hookrightarrow,u,"\phi"] & A \ar[r] \ar[u, equal] & 1
\end{tikzcd}
\]

Let $aR$ be a coset in $ \Pi_1/R$. The coset $i(a)H$ is a well-defined coset in $\Pi_A/H$. Two cosets $aR$ and $bR$ give rise to equal cosets $i(a)H$, $i(b)H$ if and only if $i(a)^{-1}i(b) \in H$. We will prove that $i(\Pi_1) \cap H=i(R)$ and therefore $i(a)H=i(b)H$ if and only if $aR=bR$.

It is clear that $i(R) \subseteq H \cap i(\Pi_1)$. Let $h\in H \cap i(\Pi_1)$. If $h \not \in i(R)$ then $p_2(h)\neq 0$. On the other hand $\phi(h) \in  i(\Pi_1)$ and $p_1(\phi(h))=0$, which means that $\phi(h) \in i(R)$, a contradiction. Thus $H \cap i(\Pi_1)=i(R)$.
This proves that the map $aR \mapsto i(a)H$ is 1-1. 

To finish the proof we need the following lemma which we interject here:

\begin{lemma}
  \label{lemma:coset-map}
  If the cover $X \rightarrow (B,A)$ is regular, then for every coset $xH$  in $\Pi_A/H$ we can select the representative $x \in \Pi_A$.
\end{lemma}
\begin{proof}
We have $p_1(x)=\gamma \in A$. There is an element $h \in H$ such that $p_2(h)=\gamma$. Then $xH= x h^{-1}H$ and $p_1(xh^{-1})=1$. This means that $xh^{-1} \in i(\Pi_1)$.
\end{proof}

The lemma \ref{lemma:coset-map} implies that $xH=i(a)H$ for some $aR \in \Pi_1/R$. This proves that the indices $[\Pi_A:H]$ and $[\Pi_1:R]$ are equal.
\end{proof}

We now relate the group of invariance $A_X^{\mathrm{inv}}$ with the regularity condition. Consider the following commutative diagram.

\begin{equation}
\label{eq:defH-R}
\begin{tikzcd}[column sep=normal,row sep=normal]
1 \ar[r] & R \ar[r] & N_{\Pi_A}R \ar[r,"\pi"] & \dfrac{N_{\Pi_A}R}{R} \ar[r] & 1 \\
1 \ar[r] & R \ar[u, equal] \ar[r] \ar[hookrightarrow,d] & H \ar[r] \ar[hookrightarrow,u] \ar[hookrightarrow,d] & A_X^{\mathrm{inv}} \ar[r] \ar[hookrightarrow,u] \ar[hookrightarrow,d] & 1 \\[7.5pt]
1 \ar[r] & \Pi_1 \ar[r] & \Pi_A \ar[r] & A \ar[r] & 1
\end{tikzcd}
\end{equation}

The lowest row in the above diagram is the short exact sequence of eq. (\ref{eq:septp}). 
Moreover, the injection $A_X^{\mathrm{inv}} \hookrightarrow N_{\Pi_A}R/R$ exists, simply by the usual group isomorphism theorems, however, in a broader sense the group $N_{\Pi_A}R/R$ contains all automorphisms; number-theoretic (resp. topological) and geometric ones, implying the injection arises naturally.  The cover $X\rightarrow (B,A)$ is regular if and only if $A_X^{\mathrm{inv}} = A$.

\section{Descent on Monodromy}\label{ch:descend_on_monodromy}

This section concerns the theory of descent—or equivalently, group ascent—for covers and their monodromy representations, as developed in the arithmetic setting by Dèbes and Douai in \cite{DebesDouai97}. In particular, they provide cohomological criteria guaranteeing when the group of moduli {\em with respect to a cover} is actually a field of definition {\em for the cover}. By adopting a unified language and notation for both the arithmetic and topological settings, we emphasize that their results, formulated in terms of group cohomology, carry over verbatim to the topological context.

Recall that in the arithmetic setting $\mathfrak{\Gamma}$ is $\Gal(K_s/K)$ and by $A^\prime/A$ we denote the group $\Gal(L/K)\cong \Gal(K_s/K)/\Gal(K_s/L)$ of a finite Galois extension $L/K$. The objects are varieties over a field $K\subseteq L \subseteq K_s$ and a $(G)$-cover $X\rightarrow (B,A)$ of degree $d$ corresponds to the regular function field extension $K(X)/K(B)$ determined by a monodromy representation $\Psi\colon \Pi_A\rightarrow G\leq N$, where $N$ is the normalizer of $G$ in $S_d$. 
In the topological setting, $\mathfrak{\Gamma}$ is $\mathrm{Mod}(B^*)$, the objects are algebraic curves over $\mathbb{C}$, such that for each branched cover $X\rightarrow B$, $X$ is of genus $ g\geq 2$ and the corresponding topological cover $X^\circ \rightarrow B^*$ is between hyperbolic surfaces. Everything else in the topological setting is denoted analogously, with the Birman exact sequence (\ref{ses:Birman}) playing the role of the arithmetic exact sequence (\ref{ses}).

In particular, we will need the following hypothesis which we adopt from \cite{DebesDouai97}:

\begin{align*}
    \textbf{(Seq/Split)} \ \ \ \ & 1\to\Pi_{A} \to \Pi_{A^\prime}\to A^\prime/A\to 1 \ \  \textrm{ is split.}
\end{align*}
\vspace{0.01cm}

As a special case for $A=1, A^\prime = \mathfrak{\Gamma}$ this is the condition of splitting of both (\ref{ses}) and (\ref{ses:Birman}) in the arithmetic and topological settings, respectively.

The direct way to attack the problem of descent is to work with the possible group extensions $H$ which keep track of the definability of the cover with (geometric) fundamental group $R$ and are classified by some pointed set $H^2(\mathfrak{\Gamma},R)$ in non-abelian cohomology, as seen on the top row of the diagram below
\[\begin{tikzcd}
1 \arrow[r] & R \arrow[r] \arrow[d, hook] & H \arrow[r] \arrow[d, hook]       & \mathfrak{\Gamma} \arrow[r] \arrow[d, no head, Rightarrow] & 1 \\
1 \arrow[r] & \Pi_1 \arrow[r]             & \Pi_{\mathfrak{\Gamma}} \arrow[r] & \mathfrak{\Gamma} \arrow[r]                                & 1.
\end{tikzcd}\]  
The brilliant idea used in \cite{DebesDouai97} is that rather than doing that, 
they exploited the finite group “$G$" in the $(G)$-cover, as well as the exact sequences
\[1\rightarrow Z(G)\rightarrow C\rightarrow CG/G\rightarrow 1 \ \text{ and } \  1 \rightarrow CG/G\rightarrow N/G \rightarrow N/CG\rightarrow 1,\]
involving the centralizer $C$ and normalizer $N$ of $G$ in $S_d$ and the center $Z(G)$. These allow the descent problem to be transferred to (abelian) group cohomology, since the outer action of $\mathfrak{\Gamma}$ on $R$ corresponds to a conjugation action on $G$ through the monodromy representation of the cover.

It is noteworthy that the problem was later revisited in \cite{MR1671938} using non-abelian cohomology tools and the theory of gerbes developed by Giraud in \cite{MR344253}.

\begin{remark}
    The “of the cover" part is important here; in this section the group of moduli is defined with respect to isomorphisms of covers. In previous sections, the group of moduli is defined with respect to isomorphisms as $K_s$-schemes. Here we simply demand that the (geometric) fundamental groups are conjugate and equal in the case of Galois covers. In the previous topological setting, we allowed $\sigma_*(R)$ not to be necessarily equal to $R$, but the present framework demands equality. 
    Furthermore, in the arithmetic case of curves and schemes of finite-type, the group of moduli and the group of moduli of the covers are connected by the canonical quotient map $X\rightarrow X/\mathrm{Aut}(X)$, which is also a covering, as shown in \cite[Thm. 3.1]{DebesEmsalem}. The reduction to covers in that theorem is based on Weil’s cocycle condition, further highlighting the connection between the present topological framework and Section \ref{sec:Top_Weil}.
\end{remark}

We now develop the theory of group ascent on covers. Assume that for a (closed) subgroup $A\leq \mathfrak{\Gamma}$ we have the homomorphism
\[
  \phi\colon \Pi_A \rightarrow  G < N,
\]
corresponding to the $(G)$-cover $f\colon X \rightarrow (B,A)$.

\begin{definition}\label{def:groupofModuliCovers}
Let $1 \leq A \leq A' \leq \mathfrak{\Gamma}$ be groups with $A \lhd A'$. 
The group $A'$ is called {\em the group of moduli of the $(G)$-cover $f$ relative to $A \leq A'$} if and only if the ramification divisor is $A'$-invariant and the following {\em group of moduli of the cover condition} \textbf{(CMod)} holds: 

\textbf{(CMod)}: For every $U \in \Pi_{A'}$ there exists $\varphi_U\in N$ such that
\[
\phi(x^U) = \phi(x)^{\varphi_U} \text{ for all } x \in \Pi_A.  
\]
\end{definition}
Note that the condition about the ramification divisor affects only the arithmetic setting. It is automatically satisfied in the topological setting as the mapping class group fixes the marked ramified branched points by definition.

The elements $\varphi_U$ are well defined up to elements in $C=\mathrm{Cen}_NG$.
This way, we obtain a map $\bar{\varphi}\colon \Pi_{A'} \rightarrow  N/C$ defined by $U \mapsto \varphi_U \mod C$. 
Note that, for $U \in \Pi_A \leq \Pi_{A'}$, we have $\varphi_U= \phi(U) \mod C$, since
\[
  \phi(x^U)= \phi(U x U^{-1}) = \phi(U) \phi(x) \phi(U)^{-1} = \phi(x)^{\phi(U)} \mod C.
\] 
The map $\phi_U$ is not necessarily a homomorphism, but satisfies
\[
  \varphi_{\eta_1} \varphi_{\eta_2} \phi(x) \varphi_{\eta_2}^{-1} \varphi_{\eta_1}^{-1}  = \varphi_{\eta_1\eta_2} \phi(x) \varphi_{\eta_1 \eta_2}^{-1}, \ \text{for all } x \in \Pi_{A},
\]
which implies
\[
  \varphi_{\eta_1} \varphi_{\eta_2}  \varphi_{\eta_1 \eta_2}^{-1} =c(\eta_1, \eta_2) \in C. 
\]
This map is a $2$-cocycle. Moreover, it is trivial if the sequence 
\[
  1 \rightarrow C \rightarrow N \rightarrow N/C \rightarrow 1
\]
splits.
In this case, the homomorphism  $\bar{\varphi}$ can be extended to a homomorphism $\Pi_{A'} \rightarrow N$. 
If such a lift exists, then the cover can be defined over its group of moduli $A^\prime$ and we say that an $A^\prime$-model exists. This terminology is well-understood in the arithmetic case, but we make formal in the next definition for both settings.

\begin{definition}{\label{def:Amodel}}
Let  $f\colon X \rightarrow (B,A)$ be $(G)$-cover with corresponding monodormy $\phi \colon \Pi_A \rightarrow G < N$.
  An {\em $A^\prime$-model of the $(G)$-cover} $f$ is a cover $f_A^{\prime} \colon X_{A^\prime}\rightarrow (B,A^\prime)$, corresponding to a homomorphism $\phi_{A^\prime}\colon\Pi_{A'} \rightarrow N$, such that $\phi$ and $\phi_{A^\prime}$ are conjugate when restricted to $\Pi_1$.
\end{definition}

Assume that there is such a homomorphism  $\phi_{A'}\colon \Pi_{A'} \rightarrow N$
that extends $\phi$, and consider the following diagram 
\[
\begin{tikzcd}
& 1 \arrow[d] & 1 \arrow[d] & 1 \arrow[d] \\
1 \arrow[r] 
  & \ker(\phi) \arrow[d] \arrow[r] 
  & \ker(\phi_{A'}) \arrow[d] \arrow[r] 
  & \ker(\Lambda) \arrow[d] \arrow[r] 
  & 1 \\
1 \arrow[r] 
  & \Pi_A \arrow[d, "\phi"'] \arrow[r] 
  & \Pi_{A'} \arrow[d, "\phi_{A'}"'] \arrow[r] 
  & A'/A \arrow[d, "\Lambda"] \arrow[r] 
  & 1 \\
1 \arrow[r] 
  & G \arrow[r] 
  & N \arrow[r] 
  & N/G \arrow[r] 
  & 1
\end{tikzcd}
\]

The existence of the homomorphism $\Lambda$ follows using $\phi(\Pi_A)= \phi_{A'}(\Pi_{A}^\prime) = G$, and the uniqueness follows becasuse the map $\Pi_{A'} \rightarrow  {A'}/A$ is surjective.

For this paragraph, we shift solely to the arithmetic setting, to gain some insight applicable for both contexts. Suppose that $A^\prime/A = \Gal(L/K)$, then the group $\mathrm{ker}(\Lambda)$ is the Galois group $\Gal(L/\hat{K})$, where $\hat{K}=\widehat{K(X_{A^\prime})} \cap L$ is the extension of constants in the Galois closure. Moreover, $\Gal(\hat{K}/K)=\mathrm{Im}(\Lambda)$, see \cite[sec. 2.8]{DebesDouai97}. 
The tower of field extensions $K = K_s^{A'} \subseteq  \hat{K} = K_s^{\hat{A}} \subseteq L=K_s^A$ corresponds to the chain of closed Galois groups $ A \leq   \hat{A} \leq A' \leq \Gal(K_s/K)$, which is a construction we can carry over on the topological setting.

Returning to the common notation and treating both settings simultaneously, the kernels $\ker(\phi)$ and $\ker(\phi_{A^\prime})$ correspond to Galois closures of the respective covers. Their quotient yields a subgroup $\ker (\Lambda) \cong \hat{A}/A \leq A^\prime/A$, and a sequence of (closed) subgroups $A\leq \hat{A} \leq A^\prime \leq \mathfrak{\Gamma}$. The subgroup $\hat{A}$ encapsulates a notion of “extension of constants in the Galois closure" that is fully analogous to the arithmetic one.

\begin{remark}{\label{rem:GcoversL}}
For a $G$-cover $f\colon X\rightarrow (B,A)$, we have $N=G$ and therefore $\hat{A}=A$.
\end{remark}

\subsection{Mere cover extensions}
\label{sec:mercovers}
 
We now turn to the case of mere covers. We recall the two theorems introduced in \cite{DebesDouai97} and consider their topological interpretation. The original proofs are of group theoretic nature and carries over verbatim to the topological case.  

Assume that $A'$ is the group of moduli of the $(G)$-cover $f\colon X \rightarrow (B,A)$, where $A \lhd A'$. Suppose that the group of moduli condition \textbf{(CMod)}, given in definition \ref{def:groupofModuliCovers} holds. 
Let $\bar{\phi}\colon \Pi_{A'}  \rightarrow N/C$ be the representation of $\phi$ modulo $C$, which exists by the group of moduli condition. 

As in the arithmetic setting, studied by \cite{DebesDouai97}, there might exist several $A'$-models of the $(G)$-cover \linebreak $f\colon X \rightarrow~(B,A)$. An invariant of such cover is the group $\widehat{A}$, interpreted earlier as the extension of constants in the Galois closure. The following theorem is the analog of ``first obstruction'' of \cite[Main theorem I]{DebesDouai97}.

\begin{theorem}{\label{thm:Mere1}}
Let $f\colon X\rightarrow (B,A)$ be a $(G)$-cover with group of moduli $A^\prime$. Let $\bar\phi \colon  \Pi_{A^\prime} \rightarrow N/C$ be the representation of $\Pi_A$ modulo $C$ given by the group of moduli condition.

  \begin{enumerate}
    \item
There is a unique homomorphism $\lambda$ making the  diagram commute
\[
  \xymatrix{
    \Pi_{A'} \ar[r] \ar[d]_{\bar{\phi}} & A'/A \ar[d]^{\lambda} \\ 
    N/C \ar[r] & N/CG
  }
\]
\item For each $A^\prime$-model, the associated map $\Lambda\colon A'/A \rightarrow N/G$ is a lift of $\lambda$. 
\item ($\lambda$/lift)-condition: The existence of a lift $\Lambda$ making the following diagram
commute
\[
  \xymatrix{
   && A'/A \ar[d]_{\Lambda} \ar [dr]^{\lambda} & \\
   1 \ar[r]& CG/G \ar[r] &  N/G \ar[r] & N/CG \ar[r] & 1
  }
\]
is necessary for the group $A'$ to be a group  of definition of the $(G)$-cover.  
  \end{enumerate}
\end{theorem}
The map $\lambda\colon \Pi_{A'} \rightarrow N/C$ is uniquely determined by the representation $\phi_{A} \colon \Pi_{A} \rightarrow G \leq N$ associated to the $(G)$-cover $f\colon X  \rightarrow (B,A)$. If $\phi_{A}$ is replaced by its conjugate $\phi_{A}^a$ for some $a\in N$, then $\lambda$ must be replaced by $\lambda^{\bar{a}}$, where $\bar{a}$ is the coset of $a$ modulo $CG$.  

In the special case of $G$-covers, we have that $N/CG=N/G=\{1\}$, so $\lambda$ is the trivial map modulo $C$. Thus, the trivial map $A^\prime/A \rightarrow N/C$ is the only possible constant extension lifting $\lambda$. Therefore, the ($\lambda$/lift)-condition holds automatically.    

Following \cite[3.2]{DebesDouai97}, we define the actions $L,L_\Lambda,L_\Lambda^*$ and the group cohomology operator $\delta^1$. The notation $L$ no longer denotes a field, as in the arithmetic setting; instead denotes the action defined as follows:
\begin{itemize}
    \item $L$ is the action of $A^\prime/A$ on $Z(G)$ obtained by composing $\lambda\colon A'/A \rightarrow  N/CG$ with the conjugation action of $N/CG$ on $Z(G)$. 
    \item $L_\Lambda$ is the action of $A^\prime/A$ on $C$ obtained by composing a lift $\Lambda\colon A'/A \rightarrow N/G$ of $\lambda$ with the conjugation action of $N/G$ on $C$.
    \item  $L_\Lambda^*$ is the action $A^\prime/A$ on $CG/G$ by composing $\Lambda \colon A'/A \rightarrow N/G$ with the conjugation action of $N/G$ on $CG/G$
\end{itemize}

For a fixed lift $\Lambda\colon A'/A \rightarrow N/G$, we consider the short exact sequence:
\[
  1 \rightarrow Z(G) \rightarrow  C \rightarrow CG/G \rightarrow 1.
\] 
It has an abelian kernel $Z(G) \leq Z(C)$, therefore, the group cohomology operator 
\[
\delta_1\colon H^1(A'/A,CG/G, L_\Lambda^*) \rightarrow 
H^2(A'/A,Z(G),L)
\]
is defined. 
\begin{theorem}{\label{thm:DebesDouaiThB}}
Let $A$ be the group of moduli of the $(G)$-cover $f\colon X \rightarrow (B,A)$ relative to $A'/A$ and assume that the condition $(\lambda/\textrm{lift})$ of Theorem \ref{thm:Mere1} holds. Fix a lift $\Lambda \colon A'/A \rightarrow N/G$ of $\lambda$. 
\begin{enumerate}
    \item Let $s\colon A^\prime/A \rightarrow \Pi_{A'}$ be a set-theoretic section of the map $\Pi_{A'} \rightarrow A'/A$, and let $\bar{\phi} \colon \Pi_{A'} \rightarrow N/C$ be the representation of $\Pi_{A'}$ modulo $C$ given by the group of moduli condition. For each $u \in A'/A$, there exists an element $\phi_u \in N$, unique modulo $Z(G)$, such that
    \[
      \phi_u=\Lambda(u) \mod G \ \text{and} \  \phi_u= \bar{\phi} \circ s(u) \mod C.
    \]
    \item Consider the $2$-cochain $\Omega_{u,v}$, where $u,v \in A^\prime/A$, defined by 
    \[
      \Omega_{u,v} = \phi_u \phi_v \phi_{uv}^{-1}
      \big(
        \phi_{A}(s(u),s(v),s(uv)^{-1})
         \big)^{-1}.
    \] The $2$-cochain $\left(\Omega_{u,v}\right)_{u,v\in A^\prime/A}$ induces a $2$-cocycle $\Omega_\Lambda \in H^2(A'/A,Z(G),L)$, independent of the choice of $\phi_u \in N$ modulo $Z(G)$ and of the set-theoretic section $s$.
    
    \item Every lift $\Lambda'$  of $\lambda\colon A'/A \rightarrow N/CG$ is of the form $\Lambda'= \tilde{\theta} \cdot \Lambda$, where $\tilde{\theta}\colon A'/A \rightarrow Z(G)$ is any $1$-cochain in $Z^1(A'/A,CG/G,L^*_\Lambda)$.
    
    \item Let $\tilde{\theta} \in Z^1(A'/A,CG/G,L^*_\Lambda)$ and let $\theta$ be the corresponding $1$-cocycle in the cohomology group $H^1(A'/A,CG/G,L^*_\Lambda)$. The following are equivalent: 
    \begin{enumerate}
        \item $\Omega_\Lambda^{-1}= \delta_1(\theta)$
        \item There exists an $A'$-model of the $(G)$-cover $f$ with $\lambda\colon A'/A \rightarrow N/CG$ equal to the map $\tilde{\theta} \cdot \Lambda \colon A'/A \rightarrow N/G$.
    \end{enumerate}
    \item The following conditions are equivalent 
    \begin{enumerate}
        \item $\Omega_\Lambda^{-1} \in \delta_1(H^1(A'/A,CG/G,L^*_\Lambda))$
        \item The group of moduli $A'$ is a group of definition for the $(G)$-cover $f$.
    \end{enumerate}
    
\end{enumerate}
\end{theorem}

If, moreover, the \textbf{(Seq/Split)} condition holds, then the above theorem simplifies as follows. Let $s\colon A^\prime/A \rightarrow \Pi_{A^\prime}$ be a group-theoretic section of the map $\Pi_{A^\prime} \rightarrow A^\prime/A$. Then, following Dèbes and Douai \cite{DebesDouai97}, ascending the group of definition from $A$ to $A^\prime$ consists of constructing an action $A^\prime/A\rightarrow N\leq S_d$ that is compatible with the given $A$-model $\phi_A\colon \Pi_A\rightarrow G\leq N$, in the sense that it respects the semi-direct product $\Pi_{A^\prime} \cong \Pi_A \rtimes_s A^\prime/A$ induced by $s$. Under these hypotheses, the theorem takes the following simpler form.

\begin{theorem}{\label{thm:DebesDouaiTHC}}
  Assume that $A'$ is the group of moduli of the $(G)$-cover $f \colon X \rightarrow (B,A)$ relative to $A'/A$ and that \textbf{(Seq/Split)} condition holds. Fix a section $s\colon A'/A \rightarrow \Pi_{A'}$, and let $\bar{\phi} \colon \Pi_{A'} \rightarrow N/C$ be the representation of $\Pi_{A'}$ modulo $C$ given by the group of moduli condition.
  \begin{enumerate}
      \item For each lift $\Lambda \colon A'/A \rightarrow N/G$ of $\lambda \colon A'/A \rightarrow N/CG$, the $2$-cocycle $\Omega_\Lambda \in H^2(A'/A,Z(G),L)$ is trivial if and only if there exists a homomorphism $\phi\colon A'/A \rightarrow N$ that lifts the homomorphism $\bar{\phi} \circ s$ and induces $\Lambda$ modulo $G$. 
      
      \item In particular, the group of moduli $A'$ is a group of definition of the $(G)$-cover if and only if the homomorphism $\bar{\phi} \circ s \colon A'/A \rightarrow N/C$ admits a lift $\phi\colon A'/A \rightarrow N$. 
      
      \item To each lift $\phi \colon A'/A \rightarrow N$ of the map $\bar{\phi} \circ s$ corresponds an $A^\prime$-model of the $(G)$-cover $f$. The corresponding $A^\prime$-model is precisely the one associated with the extension of $\phi_A \colon \Pi_A\rightarrow G\leq N$ to $\Pi_{A^\prime} \cong \Pi_A \rtimes_s A^\prime/A$ and this extension coincides with $\phi$ on the $A^\prime/A$ component, respecting the semi-direct product structure.

  \end{enumerate}
  
\end{theorem}

\section{The Birman-Hilden Property}\label{ch:Hilden}
 
In this section, we discuss the Birman-Hilden property and its connections with our work. The articles of Birman-Hilden 
\cite{Birman1971-qj},
\cite{Birman1972-pg},
\cite{Birman1973-ea}  represented a major breakthrough in the study of mapping class groups and the Birman-Hilden property (BH for short) is a central result in their theory.
For a recent comprehensive survey of the subject and its developments, we refer to \cite{MR4275077}, which we will mainly follow in this section.

The starting point of Birman and Hilden’s work was the relationship between mapping classes under branched coverings.
In particular, they considered the two-fold branched cover $S_2\rightarrow S_{0,6}$, where the six branch points are regarded as marked. 
Their goal was to compute $\mathrm{Mod}(S_2)$ by mapping it to $\mathrm{Mod}(S_{0,6})$, the latter being well understood because of its connection with the braid group, see equation~(\ref{eq:sphMod}). 
The deck transformation of the cover is the hyperelliptic involution, denoted by $\iota$, and Birman and Hilden proved that $\mathrm{Mod}(S_2)/\langle \iota \rangle \cong \mathrm{Mod}(S_{0,6})$. The main difficulty they faced is that elements of $\mathrm{Mod}(S_2)$ are defined only up to isotopy, and isotopies need not respect the involution $\iota$.Therefore, there is no reason a priori to expect that the surjection $\mathrm{Mod}(S_2) \rightarrow \mathrm{Mod}(S_{0,6})$ is well-defined. However, Birman and Hilden showed that one can always pick isotopies that respect the involution, thereby resolving this issue.

In general, for a (possibly branched) finite cover $p\colon X\rightarrow B$ of hyperbolic surfaces, the BH-property aims to relate $\mathrm{Mod}(X)$ with $\mathrm{Mod}(B^*)$. Note that the branched points of $B$ are considered marked, while the ramified points of $X$ are neither removed nor marked.
Let $\mathrm{SMod}(X)$ be the elements $\sigma$ of $\mathrm{Mod}(X)$ that are fiber-preserving, i.e. elements $\sigma$ satisfying $\sigma(p^{-1}(b)) = p^{-1}(b^\prime)$ for $b,b^\prime$ in $B^*$. 
These are precicely the lifts of liftable homeomorphisms of $B$ that preserve the branched points. To regard $\mathrm{SMod}(X)$ as a subgroup of $\mathrm{Mod}(X)$, we have to consider its elements up to isotopies, not necessarily fiber-preserving. 
Let $\mathrm{LMod}(B^*)$ be the subgroup of $\mathrm{Mod}(B^*)$ defined by the property that each element has a representative that lifts to a homeomorphism of $X$.
Let $G$ denote the group of deck transformations of the cover and $G^*$ its faithful embedding in $\mathrm{SMod(X)}$; deck transformation are trivially fiber-preserving and represent non-trivial mapping classes \cite{MR4275077}\cite[Prop. 7.7]{FarbMagalit}.
  
The Birman-Hilden holds if for any two fiber-preserving homeomorphisms of $X$ that are isotopic, a fiber-preserving isotopy can be chosen. In that case, by \cite[Prop 3.1]{MR4275077}, the BH-property is equivalent to the following sequence being well-defined and exact:
  \begin{equation}\label{eq:BH-property}
    1 \rightarrow G^* \rightarrow \mathrm{SMod}(X) \rightarrow \mathrm{LMod}(B^*) \rightarrow 1.
  \end{equation}
We now state some notable results of the Birman-Hilden theory for a finite branched cover $X\rightarrow B$ of hyperbolic surfaces:
\begin{itemize}
    \item (Birman-Hilden \cite[Prop. 3.1]{MR4275077}) For Galois branched covers the BH-property holds if and only if $N_{\mathrm{Mod}(X)}(G^*)/G^* \cong \mathrm{LMod}(B^*)$.
    \item (Winarski \cite{MR3370031}) Fully ramified covers satisfy the BH-property.
    \item (Birman-Hilden \cite{Birman1972-pg}, \cite{Birman1973-ea}) Unbranched covers and covers that have solvable deck transformation group satisfy the BH-property.
    \item (Maclachlan-Harvey \cite{MR374414}) Any Galois cover satisfies the BH-property.
\end{itemize}
  Note that Winarski's result implies the last two results, although they were proved independently and earlier. Moreover, M. Sato \cite[Prop. 1.2]{SATO_2009} characterized 
  $\mathrm{LMod}(B^*)$ as the stabilizer of the action of $\mathrm{Mod}(B^*)$ on the monodromy representation $\Psi \colon \pi_1(B^*,b_0)\rightarrow S_d$ corresponding to $X\rightarrow B$.
  
  We now turn our attention to covers $X\rightarrow B$ of algebraic curves over $\mathbb{C}$ such that $X^\circ\rightarrow B^*$ are hyperbolic. Recall that in the case of curves, the corresponding function field extension $\mathbb{C}(X)/\mathbb{C}(B)$ does not detect the punctures; therefore $X\rightarrow B$ is Galois if and only if $X^\circ \rightarrow B^*$ is Galois, if and only if the extension $\mathbb{C}(X)/\mathbb{C}(B)$ is Galois.

The Birman-Hilden property can be interpreted in terms of group theory as follows. Let $R < \pi_1(B^*,b_0)$ be the fundamental group of the open curve $ X^{\circ}$, which is generated by $2g+s-1$ elements, freely if $s\geq 1$, admitting a presentation
\[
    R= \langle \alpha_1, \beta_1, \ldots, \alpha_g, \beta_g, \gamma_1, \ldots, \gamma_s | \gamma_1 \gamma_2 \cdots \gamma_s[\alpha_1,\beta_1]\ldots [\alpha_g,\beta_g]=1 \rangle.
\]
Since the cover $X^\circ \rightarrow B^*$ is unbranched, it satisfies the BH-property. The groups in the exact sequence (\ref{eq:BH-property}) are the following:
\begin{align*}
    G^* &= \mathrm{N}_{\pi_1(B^*,b_0)(R)}/R, \\
    \mathrm{SMod}(X^\circ) &= \{ [f] \in \mathrm{Out}_+^*(R) : \exists\, \hat f \in \mathrm{Aut}_+^*(\pi_1(B^*,b_0)),\ \hat f|_R = f \text{ modulo } \mathrm{Inn}(R) \}, \\
    \mathrm{LMod}(B^*) &= \{\sigma \in  \mathrm{Out}_+^*(\pi_1(B^*,b_0)) : \ \sigma([R]) = [R]\}.
\end{align*}

Here the symbols “$+$" and “$*$" denote, respectively, the orientation-preserving automorphisms and those preserving
the conjugacy classes of the loops around the punctures, as in Theorem \ref{thm:Dehn-Nielsen-Baer}. 
In the above presentation of $R$, these are the parabolic elements $\gamma_i$. 
The group $G^*$ embeds into $\mathrm{SMod}(X^\circ)$ via the outer conjugation action on $R$. The characterization of $\mathrm{SMod}(X^\circ)$ follows from the fact that each fiber-preserving self-homeomorphism of $X^\circ$ yields a unique self-homeomorphism of $B^*$ such that the following diagram commutes: 
\[\begin{tikzcd}
X^\circ \arrow[r] \arrow[d] & X^\circ \arrow[d] \\
B^* \arrow[r]               & B^*  .            
\end{tikzcd}\]

This argument, used previously in Section~\ref{sec:Top_Weil}, also appears in \cite{SATO_2009}.
In the Galois case, where we do not need to keep track of inner automorphisms of $\pi_1(B^*,b_0)$ preserving $R$, we obtain
\[
\mathrm{SMod}(X^\circ) = \{ \sigma \in \mathrm{Out}_+^*(R) : \ \exists \ \hat \sigma \in \mathrm{Out}_+^*(\pi_1(B^*,b_0)), \ \hat \sigma|_R = \sigma \}.
\]
    
Assume now that the compact curve $X$ has genus $g\geq 2$, therefore it is hyperbolic and it can be uniformized as a quotient $\mathbb{H}/\Gamma$, where $\Gamma$ admits a presentation 
\begin{equation} \label{eq:Gam}
    \Gamma= \langle \alpha_1, \beta_1, \ldots, \alpha_g, \beta_g  : [\alpha_1,\beta_1]\ldots [\alpha_g,\beta_g]=1 \rangle.
\end{equation}

Let $\langle \gamma_1, \ldots, \gamma_s \rangle$ be the normal closure in $\pi_1(B^*,b_0)$ of the subgroup generated by the homotopy classes of loops around the punctures $\gamma_1, \ldots, \gamma_s$. This means that the group $\Gamma$ can be recovered from $R$, as the quotient 
\begin{equation} 
\label{eq:RG}
\Gamma=
\frac{
R\langle \gamma_1, \ldots, \gamma_s\rangle}
{ \langle \gamma_1, \ldots, \gamma_s \rangle},
\end{equation}
by compactification from $X^\circ$ to $X$, see \cite[sec. 5]{MR4117575} and \cite{MR4186523}. 
The mapping class group of $X$ can be realized as $\mathrm{Out}_+(\Gamma)$, by the Dehn–Nielsen–Baer theorem \ref{thm:Dehn-Nielsen-Baer}. 

In particular, if the branched cover $X\rightarrow B$ is Galois, so that it also satisfies the BH-property, then we can replace $\mathrm{SMod}(X^\circ)$ by $\mathrm{SMod}(X)$ which consists of the elements $\sigma \in \mathrm{Out}_+(\Gamma)$ that arise in the following way: if $\hat \sigma \in \mathrm{Out}_+^*(\pi_1(B^*,b_0))$ preserves $R$, then it induces an outer automorphism $\tilde \sigma$ on $\pi_1(B^*,b_0)/\langle \gamma_1,\cdots,\gamma_s\rangle$ and we set $\sigma \coloneqq \tilde \sigma|_R$.

Indeed, this substitution can also be done without invoking the BH-property in the following elementary way.

\begin{lemma}{\label{lemma:extRout}}
  Let $h \in \mathrm{Aut}(\pi_1(B^*,b_0))$ and let $\Gamma$ be given by eq. (\ref{eq:Gam}). 
Assume that $h \langle \gamma_1, \ldots, \gamma_s \rangle = \langle \gamma_1, \ldots, \gamma_s\rangle$ and 
$h(R)=R$. 
Then, 
the automorphism $h$ induces a map $\bar{h} \in \mathrm{Aut}(\Gamma)$.
\end{lemma}
\begin{proof}
Consider the following diagram of groups, which has exact rows and $h$ is an automorphism that keeps  $\langle \gamma_1, \ldots, \gamma_s \rangle$ and $R$ invariant: 
  \[
    \xymatrix{
      1 \ar[r] & \langle \gamma_1, \ldots, \gamma_s \rangle \ar[r] \ar[d]^h & R \langle \gamma_1, \ldots, \gamma_s \rangle \ar[r] \ar[d]^{h} & 
      \frac{ R \langle \gamma_1, \ldots, \gamma_s \rangle}{\langle \gamma_1, \ldots, \gamma_s \rangle} \ar[r] \ar[d]_{\bar{h}}& 1
      \\
      1 \ar[r] & \langle \gamma_1, \ldots, \gamma_s \rangle \ar[r] & R \langle \gamma_1, \ldots, \gamma_s \rangle \ar[r] & 
      \frac{ R \langle \gamma_1, \ldots, \gamma_s \rangle}{\langle \gamma_1, \ldots, \gamma_s \rangle} \ar[r] & 1
    }
  \]
It is clear that $h$ induces a group homomorphism $\bar{h}\colon \Gamma \rightarrow  \Gamma$. By repeating the same construction for $h^{-1}$, we obtain that $\bar{h}^{-1}$ is the inverse of $\bar{h}$, thus $\bar{h}$ is an automorphism.
\end{proof}

Note that, the reverse construction follows from short $5$-lemma. Since $R$ is normal, $R\langle\gamma_1,\ldots,\gamma_s\rangle = R$ and both the previous lemma and its reverse construction, can be adjusted to map $\mathrm{Out}_+^*(R)$ to $\mathrm{Out}_+(\Gamma)$. 
Additionally, $G^*$ can be interpreted as subgroup of $\mathrm{Out}(\Gamma)$: we have that
\[
G^* = \frac{\pi_1(B^*,b_0)}{R}= 
\frac
{
  \frac{\pi_1(B^*,b_0)}{  \langle \gamma_1, \ldots, \gamma_s \rangle} 
}
{ 
  \frac{R}
{ \langle \gamma_1, \ldots, \gamma_s \rangle}
}
=
\frac{\frac{\pi_1(B^*,b_0)}{  \langle \gamma_1, \ldots, \gamma_s \rangle}}{\Gamma}.
\]
This implies that $G^*$ acts by outer conjugation on $\frac{\pi_1(B^*,b_0)}{  \langle \gamma_1, \ldots, \gamma_s \rangle}/\Gamma$ inducing an injection $G^*\hookrightarrow \mathrm{Out}_+(\Gamma)$.

To link the BH-property with the previous sections, notice that a $(G)$-cover $X\rightarrow (B,A)$ defined over a subgroup $A\leq \mathrm{Mod} (B^*)$ induces a homomorphism $A\rightarrow \mathrm{Out}_+^*(R)$, which can be viewed as the dual operation to the BH-property of lifting mapping classes. To conclude this section and connect it with the arithmetic, we pose the following question.

\textbf{Question.}
Does there exist a profinite or pro-$\ell$ analogue of the Birman-Hilden property? That is, upon replacing $R$ by its profinite or pro-$\ell$ completion, $\langle \gamma_1,\ldots,\gamma_s\rangle$ by its topological closure, and fundamental groups of curves by their étale fundamental groups, can one formulate a relation between the corresponding outer automorphism groups? This would have to carefully account for the inertia subgroups at the branch points.
Moreover, can such a relation be made functorial with respect to the classical BH-problem?

\section{Finite Group Actions on Categories}
\label{sec:equivCat}

In this section, we shift gears and recall a few notions regarding group actions on categories and equivariant categories. We use this language for this and the final chapter. The reader is referred to \cite{elagin2015equivarianttriangulatedcategories}, \cite{MR3979084}, \cite{MR2353249}, \cite{MR4589277}, \cite{MR3334458}  and \cite{karakikes2025equivariantrecollementssingularequivalences}
for a more extensive and complete exposition. 

As a prologue to the content of the final two chapters of this paper and a link with the previous chapters, we mention that a finite subgroup $G$ of automorphisms of a variety $X$ yields the quotient variety $X/G$ (under certain assumptions).
This group action on the variety gives rise to a group action on the category of sheaves of the variety. In particular, there is an induced group action on the category $\mathsf{Coh}(X)$ of coherent sheaves of the variety $X$ which has been extensively studied. Note that $\mathsf{Coh}(X)$ is known to be an invariant of the algebraic variety, a result known as Gabriel's Reconstruction Theorem, see \cite{MR232821}. It is natural to consider a ``categorical quotient action'' analogue for the category $\mathsf{Coh}(X/G)$. This is exactly the equivariant category $\mathsf{Coh}^G(X)$ and it is canonically equivalent to $\mathsf{Coh}(X/G)$. We also upgrade this equivalence in the derived category setting that is our main object of interest in the final chapter.

We observed that the theory of equivariant categories is also related to Weil's descent datum, both arithmetic and topological (see section \ref{sec:field of moduli, definition, invariance}) and the goal of this section is to clarify how these objects are related. In fact, in Subsections \ref{sec:ActionOnCatArith}, we show that, for finite Galois extension $L_0/L$, then the $\Gal (L_0 /L) $-equivariant category of varieties defined over $L_0$ is canonically equivalent to the category of varieties defined over $L$. Similarly, using the topological analogue of Weil's Theorem \ref{thm:Weil topological}, in Subsection \ref{sec:ActionOnCatTop}, we show the topological analogue for suitable categories of covers defined over $A \lhd_{\mathrm{f}} A^\prime$.

\subsection{Actions on Additive \& Abelian Categories}

\begin{definition}{\label{def:catactions}}
Let $G$ be a finite group and $\mathcal{A}$ an additive category.
A right group action of $G$ on $\mathcal{A}$ consists of the following data:
\begin{itemize}
\item[(i)] every $g \in G$ corresponds to an auto-equivalence also denoted by $g \colon \mathcal{A}~\to~\mathcal{A}$, for the sake of simplicity,
\item[(ii)] a family of natural isomorphisms $\theta_{g,h} \colon g \circ h \xrightarrow{\simeq} (gh)$ to compose the auto-equivalences coherently,
\item[(iii)] given three auto-equivalences $g,h,k$ the following diagram commutes:
\begin{equation*}\label{cocycle}
\begin{tikzcd}
g \circ h \circ k \arrow[r] \arrow[d] & g\circ (hk) \arrow[d] \\
(gh)\circ k \arrow[r] & (ghk)
\end{tikzcd}
 \end{equation*}
\end{itemize}
Condition (iii) is called the \textit{2-cocycle condition}.
\end{definition}

Note that similarly we can define a left group action with compositions reversed, i.e. $\theta_{g,h} \colon g \circ h \xrightarrow{\simeq} (hg)$. A categorical group action is called {\em strict} if all $\theta_{g,h}$ are identities. Then notice that the 2-cocycle condition holds automatically and the identity element of $e \in G$ corresponds to $\mathrm{Id}_{\mathcal{A}}$.

\begin{example}\label{actiononcoh1}
 Let $G$ be a finite subgroup of automorphisms of a variety (or scheme) $X$. Then $G$ induces a left action on $\mathsf{Coh}(X)$ by pullbacks. Indeed, every automorphism $g$ induces an auto-equivalence $g^*\colon \mathsf{Coh}(X) \to \mathsf{Coh}(X)$ and moreover there is a canonical isomorphism $\theta_{g,h} \colon g^* \circ h^* \xrightarrow{\simeq} (hg)^*$. The 2-cocycle condition also holds. This is a left group action. Similarly, by using pushforwards $g_*$ instead of pullbacks, one can define a right group action.
\end{example}

One can consider the ``naive'' category of invariants of the action, but this category does not behave well. Namely, given an abelian category $\mathcal{A}$ the category of invariants is not always abelian. The reason for that is that (co-)kernels may themselves not be invariant. The correct notion is that of the equivariant category:

\begin{definition}{\label{def:equivcats}}
    Given a finite group $G$ acting on an additive category $\mathcal{A}$, the equivariant category $\mathcal{A}^G$ has objects $(E, \phi)$ where $E$ is an object of $\mathcal{A}$ and $\phi$ is a family of isomorphisms $\{ \phi_g : E \xrightarrow[]{\sim} gE \}_{g\in G}$, called the linearization of $E$, satisfying the following commutative diagram $\forall g,h \in G$: 
    \begin{equation*}
        \begin{tikzcd}
            E \arrow[r, "\phi_g"] \arrow[rrr, bend right=20, "\phi_{gh}"'] & gE \arrow[r, "g\phi_h"] & ghE \arrow[r, "\theta_{g,h}"] & (gh)E
        \end{tikzcd}
    \end{equation*}
   and morphisms $f\colon (E,\phi) \rightarrow (E', \phi')$ are morphisms of $\mathcal{A}$ making the following diagram commute for all $ g \in G$:
    \begin{equation*}
        \begin{tikzcd}
            E \arrow[r, "f"] \arrow[d,"\phi_g"'] & E' \arrow[d, "\phi'_g"] \\
            gE \arrow[r, "g f"] & gE'
        \end{tikzcd}
    \end{equation*}    
\end{definition}
Note that equivariant objects $(E,\phi)$ are clearly $G$-invariant objects of $\mathcal{A}$ but with extra structure. This idea of using linearizations comes from Geometric Invariant Theory, see 
\cite[Ch.1]{MR1304906}. Note also that the hom-sets of the equivariant category $\mathcal{A}$ are subsets of the hom-sets of $\mathcal{A}$. Especially, the equivariant hom-sets are the fixed points of the $G$-action on hom-sets induced by $f \mapsto (\phi'_g)^{-1} \circ(gf) \circ \phi_g$ for $f \colon E \to E' $ when $E$ and $E$ admit linearizations $\phi$ and $\phi'$, respectively. Finally, the equivariant category of an abelian (resp. additive) category is abelian (resp. additive).

Given a finite subgroup of automorphisms of a variety $X$ (or a scheme), the quotient variety $X/G$ makes sense if and only if each $G$-orbit is contained in some affine open, see \cite[Expos\'{e} V, Proposition 1.8]{SGA1}.
Note that the condition that each $G$-orbit is contained in some affine open holds automatically when $X$ is quasi-projective, see \cite[Chapter II, Paragraph 7, Remark]{MumfordAbelian}. This is the case for all our algebraic varieties which we'll deal with in the following chapters.

\begin{example}\label{actiononcoh2}
Building on the previous Example \ref{actiononcoh1} we have that there exists an equivariant category $\mathsf{Coh}^G(X)$. This is indeed the correct notion of categorical quotient since, when each $G$-orbit is contained in some affine open and $G$ acts freely, there is an equivalence of categories
\[
\mathsf{Coh}^G(X) \xrightarrow{\simeq} \mathsf{Coh}(X/G)
\]
given by pullbacks of sheaves along the quotient map $\pi \colon X \to X/G$. This equivalence, in the case of free group actions, was first established in \cite[Proposition 2, Chapter II, Paragraph 7]{MumfordAbelian}. Note that the condition that each $G$-orbit is contained in some affine open holds automatically when $X$ is quasi-projective, see \cite[Chapter II, Paragraph 7, Remark]{MumfordAbelian}. In case the action is not free or some $G$-orbit is not contained in an affine open, the same thing holds for the stacky quotient $[X/G]$ and thus the equivariant category can be thought of as a non commutative quotient category.
\end{example}

\begin{example}\label{exm:quotaffine}
    The affine case of the previous examples is a particularly well behaved scenario which is very well known. Namely, given an affine scheme $\Spe R$ and a finite group action $G$ on $R$ we have an induced action on $\Spe R$ such that
\begin{equation*}
(\Spe R)/G  = \Spe (R^G).
\end{equation*}
For a proof see \cite[Proposition. 4.8]{EdixhovenGeerMoonen_AV} for the general case and \cite[Chapter II, Paragraph 7]{MumfordAbelian} for the case of varieties. 
Moreover, if the action is free, we have the following equivalences of categories:
$$ (\mathsf{mod-}R)^G \simeq \mathsf{Coh}^G(\Spe R) \simeq  \mathsf{Coh}(\Spe R /G) \simeq \mathsf{Coh}(\Spe (R^G)) \simeq \mathsf{mod-}R^G $$
Note that there is also another way to interpret this action. Namely a group action on a ring $R$ induces a group action on its category of modules such that $(\mathsf{mod-}R)^G \simeq \mathsf{mod}-RG$, where $RG$ is the skew group algebra of $R$ by $G$ (it is the free $R$-module with basis the elements of $G$ but multiplication defined by the rule $rg = gr^g$ where $r^g$ denotes the action of $g$ on $r$). Hence, we can extend the above row of equivalences by adding this one.  

This tells us that equivariant coherent sheaves can be thought of as sheafified modules over the skew group algebra. This interpretation is also well known, see for example Section 2 in \cite{MR4130074}.
 For a complete and recent exposition on this we refer the reader to  Section 2.4 of 
\cite{karakikes2025equivariantrecollementssingularequivalences}.
\end{example}

\subsection{Actions on Derived Categories}\label{subsec:action on derived cats}
Another interesting and well studied invariant of an algebraic variety $X$ is the bounded derived category of coherent sheaves on $X$ denoted by $\mathsf{D^b}(X)$. For the subject of bounded derived categories of coherent sheaves on $X$ we refer  to \cite{MR2244106}. Thus, it is reasonable to want to extend the group action to the derived category level and investigate the relations that arise.

Let $G$ be a finite group acting on an abelian category $\mathcal{A}$. There is a well defined action on $\mathsf{D^b}(\mathcal{A})$.
Indeed, since every $g$ (seen as an auto-equivalence of $\mathcal{A}$) is an exact functor, it induces a derived functor $\mathsf{D^b}(g)$ which acts on complexes componentwise, i.e. given a complex $E^{\bullet}$ of $\mathsf{D^b}(\mathcal{A})$, then $\mathsf{D^b}(g) (E^{\bullet}) = (gE)^{\bullet}$ (of course $g$ acts also on differentials of the complex componentwise). The composition natural isomorphisms $\theta_{g,h}$ lift to compositions of the derived functors and the 2-cocycle condition also holds. This gives rise to the data of a group action on the additive category $\mathsf{D^b}(\mathcal{A})$. Hence one can consider the equivariant category $\mathsf{D^b}(\mathcal{A})^G$.

Moreover, since the equivariant category $\mathcal{A}^G$ is abelian, there exists its derived category $\mathsf{D^b}(\mathcal{A}^G)$.
As it turns out (see  \cite[Th. 7.1]{elagin2015equivarianttriangulatedcategories} or \cite[Prop. 13]{MR3334458}) there is an equivalence of categories $\mathsf{D^b}(\mathcal{A}^G) \xrightarrow{\simeq} \mathsf{D^b}(\mathcal{A})^G$ realised by the functor which sends a complex $(E, \phi)^{\bullet}$ to $(E^{\bullet} , \phi^{\bullet})$. Note that, for this equivalence, we need to assume a non-modularity condition, namely that $|G|$ is invertible in the category $\mathcal{A}$, which means that for each morphism $f$ is uniquely divisible by $|G|$. Note also, that although there is no reason for the equivariant category $\mathsf{D^b}(\mathcal{A})^G$ to be a priori triangulated this equivalence yields its triangulated structure.

\begin{example}\label{actiononcoh3}
Building more on Examples \ref{actiononcoh1} and \ref{actiononcoh2} we have the following. Assume that $X$ is a variety over a field of characteristic prime to the order of the group $G$ (this implies the non-modularity condition) of automorphisms of $X$. Then we have the following equivalences:
\[
\mathsf{D^b}(X)^G \simeq \mathsf{D^b}(\mathsf{Coh}^G(X)) \simeq \mathsf{D^b}([X/G]).
\]
We used the stacky quotient $[X/G]$ for the general case. Recall that when $X$ is quasi-projective then we have that the quotient variety $X/G$ exists and if we assume that the action is free, then we can use this instead of the stacky quotient in the above equivalence.
\end{example}

\subsection{Action on Categories of Varieties and Weil's Descent Theorem}\label{sec:ActionOnCatArith}
In the following example, we canonically construct a $\Gal(L/K)$-equivariant category of $L$-varieties. Then we show that the said equivariant category captures precisely the information of the varieties that satisfy the Weil descent condition.

\begin{example}\label{Galoisactiononcategoryofvarieties}
Let $\mathsf{Var_{L}}$ denote the category of varieties defined over $L$. For field extensions $K \subseteq L_0 \subseteq L \subseteq K_s$ we obtain natural inclusions of full subcategories $\mathsf{Var}_{L_0} \hookrightarrow \mathsf{Var}_L \hookrightarrow \mathsf{Var}_{K_s}$, by extension of scalars as in Section \ref{subsec:aractions}. 
Given a finite Galois field extension $L/L_0$ and its corresponding Galois group $\Gal(L/L_0)$, we obtain a natural $\Gal(L/L_0)$-action on $\mathsf{Var}_L$. Indeed, the action described in Section~\ref{subsec:aractions}, using eq.~\ref{eq:strictness of action} and Remarks \ref{rem:invariantactionofsubgroup} and \ref{rem:generalizedgroupaction}, yields a strict $\Gal(L/L_0)$-action on the category of $L$-varieties. Simply observe that each $\sigma \in \Gal(L/L_0)$ induces an autoequivalence $\widetilde{\sigma}$ of $\mathsf{Var}_L$.

Therefore, we can study the equivariant category $\mathsf{Var}_L^{\Gal(L/L_0)}$. Its objects are of the form 
$\big(X, \{f_\sigma\}_{\sigma \in \Gal(L/L_0)} \big)$, where $X$ is an $L$-variety and $\{f_{\sigma}\}_{\sigma \in \Gal(L/L_0)}$ is a linearization. 

Notice that $X$ admits a descent datum if and only if $X$ admits a linearization. Morphisms of equivariant objects are morphisms that descend to morphisms of $L_0$-varieties.
This means that the equivariant category $\mathsf{Var}_L^{\Gal(L/L_0)}$ is precisely the category of $L$-varieties that are definable over $L_0$.

In fact, we have realized an equivalence of categories:
\[
\mathsf{Var}_L^{\Gal(L/L_0)} \simeq \mathsf{Var}_{L_0}
\]
The functor from $L_0$-algebraic varieties maps $Y_0$ to $(Y_0\times \Spe(L) , \{\mathrm{Id}\times \sigma \}_{\sigma \in \Gal (L/L_0)})$ while the inverse functor is induced by Weil's Theorem \ref{thm:weil's descent}.
\end{example}

Note that by Remark \ref{rem:finiteExtArith} the group $\Gal(L/L_0)$ is finite and, even if the variety is defined over $K_s$, we can find a finite field extension $L_0 \subseteq L$ such that $X$ is defined over it and apply Weil's theorem for the quotient group $\Gal(K_s /L_0) / \Gal(K_s / L) \simeq \Gal(L/L_0)$, by Remarks \ref{rem:invariantactionofsubgroup} and \ref{rem:generalizedgroupaction}. 

We can also do this more abstractly using the language we introduced in Section \ref{sec:FieldMODDefinition}, for any closed subgroup $A \leq \mathfrak{\Gamma} = \Gal(K_s/L_0)$. Namely, that of the group of invariance and the corresponding field over which $X$ (a fixed variety we are interested in for some reason) is defined and the group of definition the corresponding field of definition, i.e. using the action of the quotient group 
\[
\Gal(K_s ^{A_X^{\mathrm{inv}}}/K_s^{A_X^{\mathrm{def}}}) \simeq \Gal(K_s/K_s^{A_X^{\mathrm{def}}}) / \Gal(K_s / K_s ^{A_X^{\mathrm{inv}}}) = A_X^{\mathrm{def}} / A_X^{\mathrm{inv}}
\]
and obtain that $\mathsf{Var}_{K_s^{A_X^{\mathrm{inv}}}}^{A_X^{\mathrm{def}} / A_X^{\mathrm{inv}}} \simeq \mathsf{Var}_{A_X^{\mathrm{def}}}$.
We have actually proved the following:
\begin{proposition}\label{prop:equivariantweildescent}
Let $X$ be an $L$-variety. It is definable over $L_0$ if and only if it admits a linearization with respect to the natural $\Gal(L/L_0)$-action on the category $\mathsf{Var}_{L}$. 
    For any closed subgroup $ \Gal(K_s/L) \leq A \leq \Gal(K_s/ K)$, the group of definition $A_X^{\mathrm{def}}$ of $X$ over $A$ is the largest subgroup of $\Gal(K_s/K_s^A)$ such that $X$ admits a linearization with respect to the natural action of the quotient group $A_X^{\mathrm{def}}/A_X^{\mathrm{inv}}$.
    The group of definition $\mathfrak{\Gamma}^{\mathrm{def}}_X$ of $X$ is the largest subgroup of $\Gal(K_s/K)$ such that $X$ admits a linearization with respect to the natural action of the quotient group $\mathfrak{\Gamma}_X^{\mathrm{def}} / \mathfrak{\Gamma}_X^{\mathrm{inv}}$.
\end{proposition}

\subsection{Action on Categories of Covers and Ascent Theorem}\label{sec:ActionOnCatTop}
In the following example, we canonically construct a $A'/A$-equivariant category of covers. 

Recall from Definition \ref{def:category of covers} that $\mathsf{Cov}_{(B,A)}$ denotes the category of covers $X \to (B,A)$ with morphisms of the form $(f,g)\colon \big(X\to (B,A)\big) \to \big(Y \to (B,A)\big)$. For groups $\{1\} \leq A \leq A' \leq \mathrm{Mod}(B^*)$ we obtain natural inclusions of full subcategories $\mathsf{Cov}_{(B,A)} \hookrightarrow \mathsf{Cov}_{(B,A ^\prime)} \hookrightarrow \mathsf{Cov}_{(B,\{1\})}$, as explained below Definition \ref{def:category of covers}, thus we can simplify $(B,A)$ to $B$ if it makes the notation less heavy. Given a finite group quotient $A^\prime/A$ we obtain a left non-strict $A'/A $-action on $\mathsf{Cov}_{(B,A)}$ induced by the action of $\mathrm{Mod}(B^*)$ as described in Subsection \ref{sec:Top_Weil}.

Indeed, notice that for $\sigma \in A'/A$ we obtain a functor, also denoted $\sigma$ by abuse of notation, by mapping covers $X \to (B,A)$ to ${^\sigma \! X} \to (B,A)$ (induced by the action of $\mathrm{Mod}(B^*)$) and morphisms of covers $(f,g)$ to $({^\sigma \! f} , {^\sigma \! g})$ as described in Subsection \ref{sec:Top_Weil}. Note that for $\sigma \in A$ the cover ${^\sigma \! X} \to (B,A)$ is exactly $X \to (B,A)$ since the cover is defined over $A$ (also discussed in the said subsection) and thus the action is well-defined for elements $\sigma \in A' /A$ (compare also with Remark \ref{rem:invariantactionofsubgroup}). Moreover, we have that $\sigma $ preserves the identity morphism $(\mathrm{Id}_X , \mathrm{Id}_B)$ of any cover $X \to (B,A)$ and preserves the composition of morphisms $\big( {^\sigma (f f')} , \ {^\sigma (gg')} \big) = ({^\sigma \! f} , {^\sigma \! g}) \circ ({^\sigma \! f'} , {^\sigma \! g'})$, since, for $ff'$ and similarly for $gg'$ we have that
\[
{^\sigma \! (f f')} = \tilde{\sigma} ff' \tilde{\sigma}^{-1} = \tilde{\sigma} f \tilde{\sigma}^{-1} \tilde{\sigma} f' \tilde{\sigma}^{-1} = {^\sigma \! f} {^\sigma \! f'}
\]
Note that ${^\sigma \! ({^{\sigma^{-1}} \! X})} = X $, which is induced by $\sigma_*(\sigma^{-1}_*(R))=R$, hence $\sigma$ is surjective. Also $\sigma$ induces a bijection on the Hom-sets also discussed in Subsection \ref{sec:Top_Weil}. Hence, the functor $\sigma$ is an equivalence of categories.

Moreover, $\sigma \circ \tau \simeq (\sigma \tau)$, by which we mean that the composition of the induced functors $\sigma$ and $\tau$ is naturally isomorphic to the composition of the functor induced by the group element $\sigma\tau$. Indeed, we have a family of isomorphisms $\eta^X_{\sigma, \tau} \colon ({^\sigma \! ({^\tau \! X})} \to B) \to  ({^{\sigma\tau} \! X} \to B) $ by defining $\eta^X_{\sigma, \tau} = \widetilde{\sigma\tau} \tilde{\tau}^{-1} \tilde{\sigma}^{-1}$, which for any $(f,g) \colon (X \to B) \to (Y \to B)$ yields the following commutative diagram:

\begin{equation*}
    \begin{tikzcd}
        ({^\sigma \! ({^\tau \! X})} \to B) \arrow[d, "{^\sigma ({^\tau \! (f,g)})}"'] \arrow[r, "\eta^X_{\sigma, \tau}"] & ({^{\sigma\tau} \! X} \to B) \arrow[d, "{^{\sigma\tau} \! (f,g)}"]\\
        ({^\sigma \! ({^\tau \! Y})} \to B) \arrow[r, "\eta^Y_{\sigma, \tau}"'] & ({^{\sigma\tau} \! Y} \to B)
    \end{tikzcd}
\end{equation*}
The commutativity is checked in the following equality:
\[
{^{\sigma\tau} \! (f,g)} \eta^X_{\sigma,\tau } = \widetilde{\sigma \tau} (f,g) \widetilde{\sigma \tau}^{-1} \widetilde{\sigma \tau} \tilde{\tau} \tilde{\sigma} = \widetilde{\sigma \tau} \tilde{\tau}^{-1} \tilde{\sigma}^{-1} \tilde{\sigma} \tilde{\tau} (f,g) \tilde{\tau}^{-1} \tilde {\sigma}^{-1} = \eta^Y_{\sigma,\tau}{^\sigma ({^\tau \! (f,g)})}
\]
Thus $\eta_{\sigma,\tau} \colon \sigma \circ \tau \to \sigma\tau$ is a natural isomorphism. In addition we have the 2-cocycle condition for the natural isomorphism $\eta$ (see Definition \ref{def:catactions}) which is easily checked be the definition of $\eta$.

We have showed that there is non-strict $A'/A$-action on $\mathsf{Cov}_{(B,A)}$ and hence we can consider the equivariant category $\mathsf{Cov}_{(B,A)}^{A'/A}$.
As the arithmetic case suggests, we have an equivalence of categories $\mathsf{Cov}_{(B,A)}^{A'/A} \simeq \mathrm{Cov}_{(B,A')}$. Thus the equivariant category carries the data of the covers $X \to (B,A)$ that ascend to covers $X \to (B,A')$ and morphisms between such covers that accordingly ascend. In more detail, an equivariant object which consists of a cover $X \to (B,A)$ along with its linearization $(f_\sigma, \hat{\sigma})_{\sigma \in A'/A}$ which a Mapping class group ascent datum relative to $A'/A$ corresponds to a cover $X \to (B,A')$ by Theorem \ref{thm:Weil topological}. The inverse functor is given by extending the cover $X \to(B,A')$ to $X \to (B,A)$ with linearizations naturally given by the construction in the proof of the said theorem.

We can thus formulate the analogue of Proposition \ref{prop:equivariantweildescent}.

\begin{proposition}\label{prop:equivariantweildescent-top}
Let $X \to (B,A)$ be an cover defined over $A$ and $A \lhd_{\mathrm{f}} A^{\prime}$. It is definable over $A'$ if and only if it admits a linearization with respect to the natural $A'/A$-action on the category $\mathsf{Cov}_{(B,A)}$. 
For any subgroup $A \lhd_{\mathrm{f}} A' \leq \mathrm{Mod}(B^*)$, the group of definition ${A'}_X^{\mathrm{def}}$ of $X \to (B,A)$ over $A'$ is the largest subgroup of $A'$ such that $X$ admits a linearization with respect to the natural action of the quotient group $A_X^{\mathrm{def}}/A_X^{\mathrm{inv}}$. 
The group of definition $\mathfrak{\Gamma}^{\mathrm{def}}_X$ of $X$ is the largest subgroup of $\mathrm{Mod}(B^*)$ such that $X$ admits a linearization with respect to the natural action of the quotient group $\mathfrak{\Gamma}_X^{\mathrm{def}} / \mathfrak{\Gamma}_X^{\mathrm{inv}}$.
\end{proposition}

\section{Actions by Isomorphisms and Descent using Derived Categories}\label{sec: descent using derived categories}

In this section we work with an $L$-variety $X$ where $K \subseteq L_0 \subseteq L$ is a finite extension. We aim to construct the Weil descent variety $Y$ defined over $L_0$ using a suitable $\Gal(L /L_0)$-action on $X$ which depends on the Weil-descent datum. This provides a modern proof on Weil's theorem. Moreover, we extend this action and the respective quotient to bounded derived categories of coherent sheaves of $X$.

\subsection{Weil's Descent Variety}
Recall by Section \ref{subsec:aractions} and in particular Remark \ref{rem:non isomorphic K_s-schemes}, that each $\sigma \in G = \Gal(L/L_0)$ induces an isomorphism of schemes 
$\tilde{\sigma} \colon {^\sigma \! X} \rightarrow X$ which does not preserve the underlying $L$-scheme structure, i.e. it is not compatible with the local $L$-algebra structure of the respective structure sheaves.

\begin{proposition}\label{prop: group action by weil descent}
Let $X$ be an algebraic variety defined over $L$. If there exists a family $\{f_\sigma\}_{\sigma \in G} $, 
which  satisfies the Weil cocycle/descent condition 
\[
  f_{\sigma\tau}={^\sigma \! f_{\tau} \circ f_\sigma} \text{ for all } \sigma,\tau \in G
\]
then the map
\begin{align*}
X & \longrightarrow X  
\\
P & \longmapsto \tilde{\sigma} f_\sigma(P)
\end{align*}
is a right group action of $G$ on $X$ by automorphisms of the underlying variety $X$.
\end{proposition}

\begin{proof}
By considering the following commutative  diagram 
\[
  \xymatrix{
  X   \ar[r]^{f_\tau}   &  {^\tau \! X}  
  \\
  {^\sigma \!X} \ar[r]^{ {^\sigma \! f}_\tau}
  \ar[u]^{\tilde{\sigma}}  &
   {^{ \sigma \tau} \! X}
   \ar[u]_{\tilde{\sigma}}
  }
\]
we observe that $ {^\sigma \! f}_\tau = \tilde{\sigma}^{-1} f_\tau \tilde{\sigma}$.

Recall the relation $\widetilde{\sigma\tau} = \widetilde{\tau}\cdot \widetilde{\sigma}$ from Section \ref{subsec:aractions} and  compute:
\[
  ( \widetilde{\sigma \tau } ) f_{  \sigma \tau}
  = \widetilde{\tau} \widetilde{\sigma} \ {^\sigma \! f}_{\tau} f_\sigma 
  = 
  \widetilde{\tau}\widetilde{\sigma} \widetilde{\sigma}^{-1} f_\tau \widetilde{\sigma} f_\sigma 
  = (\widetilde{\tau} f_\tau)(\widetilde{\sigma} f_\sigma).
\]
which is the group law for the right action.
Fianlly,  $\tilde{\sigma}f_\sigma$ is an automorphism of $X$ because 
\[ 
(\widetilde{\sigma^{-1}}f_\sigma)( \tilde{\sigma}f_\sigma) = \widetilde{\sigma \sigma^{-1}}f_{\sigma \sigma^{-1}} = \tilde{e}f_e = \mathrm{id}_X
\]
using the cocycle condition and the identity $\tilde{e} = f_e = \mathrm{Id}_X$, where $e \in G$ is the identity element of the group.
\end{proof}

\begin{remark}
We record two technical observations.
\begin{itemize}
\item[\textnormal{(i)}] The computation of $ ( \widetilde{\sigma \tau } ) f_{  \sigma \tau} = (\widetilde{\tau} f_\tau)(\widetilde{\sigma} f_\sigma)$ is further reflected on the diagram: 
    \[
\begin{tikzcd}[row sep=2.5em, column sep=0.3em]
&&
{^\sigma\!X} 
 \arrow[ddd,swap,"{^\sigma \!p}" description,near start ] 
\arrow[dr] 
\arrow[dll,"\tilde{\sigma }",swap]
&&&
X   
 \arrow[lll,teal,swap,"f_\sigma"]  
 \arrow[dr,"f_\sigma  ", teal] 
 \arrow[ddd,"p"]
 \arrow[dll,"f_{\sigma \tau }" description, teal]
\\
 X 
 \arrow[rrr,<-,crossing over,"\widetilde{\sigma \tau }" description, near start] 
 \arrow[rd,<-,crossing over,"\tilde{\tau }"' ] 
 &&&
{^{\sigma\tau} \!X }  \ar[dd]
&&&  
{^\sigma \!X} 
  \arrow[ddd,"{^\sigma \!p}"] 
 \arrow[lll,teal, crossing over,"{^\sigma \!f_{\tau}}" description, swap]
 \arrow[dll,crossing over, "\tilde{\sigma}" description, near end]
\\
& {^\tau \! X} 
\arrow[rrr,<-,teal,"f_\tau" description,crossing over] 
\arrow[urr, <-, "\tilde{\sigma }" description, crossing over, near start]
& &  \;&  
X
\\ 
&& \Spe(K_s ) 
 \arrow[dr] 
 \arrow[dll,"\tilde{\sigma }" description,near start]
& \; && 
\Spe(K_s ) 
 \arrow[lll,"\mathrm{Id} ",teal] 
 \arrow[dr,"\mathrm{Id}",teal] 
 \arrow[dll,"\mathrm{Id}",teal,near start ]
\\
 \Spe(K_s  ) 
  \arrow[rrr,<-,"\widetilde{\sigma \tau}" description] 
  \arrow[uuu,<-,crossing over,"{p}" ]
  \arrow[dr,<-,"\tilde{\tau }"']   
&&& 
\Spe(K_s ) 
 \arrow[uu,<-,"{^{\sigma\tau} \!p} ",crossing over,near end] 
&&& 
 \Spe(K_s)
   \arrow[lll,"\mathrm{Id}",teal]
\\
&
\Spe(K_s ) 
 \arrow[uuu,<-,"{^{\tau} \!p} ",crossing over,near end]
 \arrow[urr,"\tilde{\sigma }" description,<-]
&&&
\Spe(K_s ) 
 \arrow[uuu,<-,"{p} ",crossing over,near end,swap]
 \arrow[urr,<-,"\tilde{\sigma }"']
 \arrow[lll,swap,teal,"\mathrm{Id}"']
\end{tikzcd}
\]
where the teal arrows represent morphisms of $\Spe(K_s)$-schemes and $\tilde{\sigma }$,  $\tilde{\tau}$ and ${\widetilde{\sigma \tau}}$ are merely isomorphisms of schemes. If they were isomorphisms of $K_s$-schemes for example, all the compositions of bottom arrows would have to be identities.

\item[\textnormal{(ii)}] The composition $\tilde{\sigma}f_\sigma$ is used to define a non trivial automorphism of $X$. The non trivial part comes from $f_\sigma$ being isomorphisms of $X$. Thus, ${^\sigma \!X}$ are, in general, different from $\tilde{\sigma}^{-1}$. Of course, it may be the case that $f_\sigma=\tilde{\sigma}^{-1}$, for all $\sigma$, but this would imply that the chosen descent datum is trivial, i.e. $X$ is actually defined over $L_0$, hence the descent is trivial. 
\end{itemize}
\end{remark}

In the following Remark we wish to clear some ambiguity (at least the authors found this ambiguous) regarding the automorphism group of $X$ in the celebrated reconstruction theorem of Bondal and Orlov.

\begin{remark}\label{rem:bondal orlov remark}
  A. Bondal  and D. Orlov, in their seminal article \cite{MR1818984}, studied the group of exact autoequivalences of a smooth projective variety $X$ defined over a field $k$ with (anti-)ample canonical sheaf, and proved that this group includes the group 
  \[
    \mathrm{Aut}(X) \ltimes (\mathrm{Pic}(X) \oplus \mathbb{Z} ),  
  \]
where $\mathrm{Pic}(X)$ is the Picard group of the variety $X$, the group $\mathbb{Z}$ corresponds to the shift functor, and $\mathrm{Aut}(X)$ is the automorphism group of $X$. 
If moreover the canonical or the anticanonical sheaf is ample, then this is the whole group of exact autoequivalences.  

Notice, that there is some ambiguity here of what the authors mean by the group $\mathrm{Aut}(X)$. When we are dealing with varieties defined over an algebraically closed field $K$, by automorphisms of $X$, we mean automorphisms in the category of $L$-schemes, that is automorphisms $f$ that make the following diagram commutative:
\[
  \xymatrix{
 X \ar[rr]^{f} \ar[dr] &  & X \ar[dl]
 \\
 & \mathrm{Spec} k & 
  }
\]
However in our setting we allow maps that are not $k$-automorphisms. 

Let us repeat the proof of \cite[Th. 3.1]{MR1818984} supposing that the canonical sheaf is ample. In the case of curves this means that the genus $g$ of the curve $X$ is $g \geq 2$. First an exact autoequivalence should preserve the invertible objects, that is invertible scheaves up to translations. Therefore, after composing with an element in $\mathrm{Pic}(X) \oplus \mathbb{Z}$, we can assume that the autoequivalence $F$ preserves $\mathcal{O}_X$ and also any tensor power $\omega_X^{\otimes i}$, since it commutes with the Serre functor. This means that $F$ induces a (twisted) automorphism of the graded coordinate algebra $\bigoplus_{i=0}^\infty H^0(X, \omega_X^{\otimes i})$ of the canonical sheaf. So in terms of the short exact sequence of the canonical embedding 
\[
0 \rightarrow  I_X \longrightarrow  \mathrm{Sym} H^0(X, \omega_X ) \longrightarrow 
\bigoplus_{i=0}^\infty H^0(X, \omega_X^{\otimes i}) 
\rightarrow 0  
\]  
$F$ acts as a twisted automorphism on all three terms involved, 
that is the algebras in the second and third position are twisted $\Gamma$-algebras, and the ideal $I_X$ has a compatible twisted action. 
Therefore the group $\mathrm{Aut}(X)$ is defined to be the elements

that keep the ideal $I_X$ invariant. 

We have already explained in Example \ref{actiononcoh3} how an isomorphism induces an action on $\mathsf{D^b}(X)$.
Briefly, an isomorphism $\sigma \colon X \rightarrow X$ induces, by functoriality, a map $\sigma^* \colon\mathsf{Coh}(X) \rightarrow \mathsf{Coh}(X)$, such that  $\sigma^*\mathcal{O}_X = \mathcal{O}_X$. This autoequivalence  $\sigma^*$ in the category $\mathsf{Coh}(X)$ induces an autoequivalence on $\mathsf{D^b}(X)$. Moreover, it keeps $I_X$ invariant and thus belongs in $\A (X)$.
\end{remark}

Observe that the quotient $X/G$, if it exists as a scheme, for the action of Proposition \ref{prop: group action by weil descent}, it is a variety defined over $L_0$. Indeed, the group $G=\Gal(L/L_0)$ acts trivially on $X/G$ and, by Proposition \ref{prop:invArith}, we reach this conclusion. In the following proposition we show that this quotient is a way to construct $Y$ - which exists since $X$ admits a Weil descent datum.
Note that the action of Proposition \ref{prop: group action by weil descent} does not affect Mumford's quotient $X/G$ although he assumes he is working over $K$-algebraic varieties (see his proof). Basically Mumford's claim is that $A^G$ is a $k$-algebra, thus $\Spe A^G = \Spe A /G$ is a $k$-scheme but in our work we do not need this extra structure.

\begin{proposition}\label{prop:actiononvariety}
    Let $X$ be an algebraic variety defined over $L$, $\{ f_\sigma \}_{\sigma \in G}$ is a Weil descent datum on $X$, where $G=\Gal(L/L_0)$, and $Y$ is the $L_0$-variety over which $X$ is definable. Consider also the action of isomorphism of $G$ on $X$ of Prosposition \ref{prop: group action by weil descent}. This action is free and $X/G \simeq Y$.  
\end{proposition}
\begin{proof}
It is known that for product quotient stacks the following holds
\begin{equation}
\label{eq:fibreProdStacks}
\left[
\frac{X \times_S Y}{G\times H} 
\right]
 \simeq 
 \left[
 \frac{X}{G} 
 \right] \times_S 
 \left[
 \frac{Y}{H} 
 \right],
\end{equation}
for some groups $G$ acting on $X$ and $H$ acting on $Y$ over some base scheme $S$,
see for example \cite{152738}.

In this case, we have $X \simeq Y \times_{L_0} \mathrm{Spec}(L)$ (where the isomorphism is over $L$) and the group $G = \{1\} \times G$ which acts trivially on $Y$, so we have:
\begin{align*}
\left[ \frac{X}{G} \right] &\simeq 
\left[  \frac{Y\times_K \mathrm{Spec}(L_0)}{\{1\}\times G} \right]  \\
& \simeq 
\left[ \frac{Y}{\{1\}} \right]\times_{L_0} \left[ \frac{\mathrm{Spec}(L)}{G} \right] \\
& \simeq [Y] \times_{L_0} \mathrm{Spec}(L^G) \\ 
& \simeq [Y] \times_{L_0} \mathrm{Spec}(L_0) \simeq [Y]. 
\end{align*}

We have used that $\Spe(L)/G = \Spe(L^G)$ of Example \ref{exm:quotaffine} and that $L^G = L_0$, by construction of the action.
Note, all the stacks involved are actually ordinary varieties, since $X$ is quasi-projective variety (since algebraic varieties are quasi-projective) as mentioned in Example \ref{actiononcoh2}. That is $X/G \simeq Y$ in the category of varieties.
Observe also that the action of the Galois group $G=\Gal(L/L_0)$ on $\Spe(L)$ is free, hence, it is free on $X \simeq Y \times_{L_0} \Spe(L)$.  
\end{proof}

  \subsection{Galois Actions on Categories of Sheaves}   
 
The following results are expected and extend the usual quotient varieties to derived categories (and to categories of coherent sheaves). However, we want to finish this article with some ideas regarding de-equivariantization of categories. The next proposition is straight-forward and explains how this action induces an action on $\O_X$-modules.

\begin{proposition}
\label{prop:9}
Let $X$ be an algebraic variety defined over $L$ and $\{f_{\sigma} \}_G$ is a Weil descent datum on $X$, where $G = \Gal(L/L_0)$. Consider the action of $G$ on $X$ of Proposition \ref{prop: group action by weil descent}.  Let $\mathcal{F}$ be a sheaf of $\mathcal{O}_X$-modules on $X$, 
  the map 
  \[
\mathcal{F} \mapsto  
 (\tilde{\sigma } f_\sigma )^*
 \mathcal{F}
 \simeq f_\sigma ^*  \tilde{\sigma}^*  \mathcal{F}
\]
 is a right group  action of $G$ on the $\mathcal{O}_X$-modules of $X$. 
\end{proposition}
\begin{proof}
  Recall that the map $\tilde{\sigma}\colon {^\sigma \! X }\rightarrow X  $ is a moprhism of schemes, but not necessarily a map of $K_s$-schemes. 
  In any case we have that 
  \[
    (\widetilde{\sigma\tau } f_{\sigma\tau} )^*
    =
    (\tilde{\tau } f_\tau  \tilde{\sigma } f_\sigma )^* \simeq
    (\tilde{\sigma } f_\sigma )^* (\tilde{\tau } f_\tau )^*.
    \] 
As in Example \ref{actiononcoh1}, the data of the group action of Definition \ref{def:catactions} are satisfied and are readily checked. Note that we have to actually use the fact that the action is by automorphisms of $X$ since the functor $(\tilde{\sigma}f_\sigma)^*$ is actually an autoequivalence of the category of $\O_X$-modules. 
\end{proof}
This final assertion of the proof is true because of the following observation:
given an {\em open morphism of schemes} $\pi \colon X\rightarrow Y$ and a sheaf $\mathcal{F}$ of $Y$ we can define the sheaf $\pi^{-1}\mathcal{F}$ of $X$ and the $\mathcal{O}_X$-module $\pi^* \mathcal{F}$ as follows:
On an open set $U \subset X$ 
\[
  \pi^{-1} \mathcal{F}(U)= \mathcal{F}(\pi(U))
  \text{ and }
  \pi^{\ast} \mathcal{F}(U) 
  = \pi^{-1} \mathcal{F}(U) \otimes_{\pi^{-1}\mathcal{O}_Y(U)} \mathcal{O}_X(U). 
\] 
In the above construction the definition of a morphism of schemes provides us with a ring homomorphism 
\[
  \pi^{-1}\mathcal{O}_Y(U) \rightarrow \mathcal{O}_X(U)
\]
which allows us the computation of the tensor product, see \cite[II.5]{Hartshorne:77}. Indeed, for any continuous function $\pi\colon X \rightarrow  Y$ we know that $\pi^{-1}$ is a left adjoint of $\pi_{\ast}$, that is there is a natural map 
\[
  \pi^{-1} \pi_{\ast} \mathcal{O}_X \rightarrow \mathcal{O}_X,
\]
and moreover since the map $\pi$ is a morphism of schemes we have $\pi_{\ast} \mathcal{O}_X= \mathcal{O}_Y$, \cite[Exer. II.1.18]{Hartshorne:77}.
This means that $\pi^*$ sends $\O_Y$-modules to $\O_X$-modules. If, moreover, $\pi$ is an equivalence, in particular $\pi^{-1}$ is also an open immersion, then $\pi$ is an equivalence of categories of $\O_X$ modules and $\O_Y$ modules, since $\pi^* (\pi^{-1})^* \simeq (\pi^{-1}\pi)^* = e^* \simeq \mathrm{Id} \simeq (\pi\pi^{-1})^* \simeq  (\pi^{-1})^* \pi^* $

We combine the pieces we have {\em collected} so far to obtain the following:

\begin{proposition}\label{prop:equivariant descent theorem}
Let $X$ be an algebraic variety defined over $L$, $\{f_{\sigma} \}_G$ is a Weil descent datum on $X$, where $G=\Gal(L/ L_0)$, and $Y$ is the $L_0$-variety over which $X$ is definable. There is a well defined $G$-action on $\mathsf{D^b}(X)$ and the equivariant category $\mathsf{D^b}(X)^G$ is isomorphic to $\mathsf{D^b}(Y)$, provided that $|G|$ is invertible in $L_0$.
\end{proposition}

\begin{proof}
The (right) group action of automorphisms of $X$ of Proposition \ref{prop:9} restricts to an action on $\mathsf{Coh}(X)$ and we have that $\mathsf{Coh}^G(X) \simeq \mathsf{Coh}([X/G])$, by Example~\ref{actiononcoh2}. Additionally, $X$ is an algebraic variety (especially quasi-projective) and the action is free by Proposition~\ref{prop:actiononvariety}, thus the quotient stack is actually a scheme, i.e. $[X/G] = X/G$. 
The group action of $G$ on the category $\mathsf{Coh}(X)$ extends on $\mathsf{D^b}(X)$, by Subsection~\ref{subsec:action on derived cats}.
Moreover, by Example~\ref{actiononcoh3}, since $|G|$ is invertible in $L_0$, we have the equivalence 
\[
\mathsf{D^b}(X)^G \simeq \mathsf{D^b}(\mathsf{Coh}^G(X)) \simeq \mathsf{D^b}(X/G) \simeq \mathsf{D^b}(Y).
\]
\end{proof}

\begin{remark}\label{rem: de-equivariantization}
Note that the fact that $X/G \times L \simeq Y \times K \simeq X$ yields the equivalence $\mathsf{D^b}(X/G \times K ) \simeq \mathsf{D^b}(X)$, which suggests that taking extension of scalars for this particular  equivariant category, i.e. $\mathsf{D^b}(X)^G \simeq \mathsf{D^b}(X/G)$, might be a method of {\em de-equivariantization}. However, we merely propose this idea and do not pursue this further.

Note also that the above proposition and hence the de-equivariantization statment could have been formulated for categories of quasi-coherent sheaves or even unbounded derived categories of quasi-coherent sheaves in a similar manner, see for example \cite[Remark 7.1]{karakikes2025equivariantrecollementssingularequivalences}. However, we restrict ourselves in this situation which is better understood and has been examined more. 
\end{remark}

 \def\cprime{$'$}

\end{document}